\documentclass[12pt]{amsart}
\usepackage{amsmath,amssymb,amscd,latexsym,eufrak}
\input xy
\xyoption{all} 
\newcommand{\svskip}{\vspace{3mm}}

\newcommand{\ba}{\begin{array}}
\newcommand{\ea}{\end{array}}
\newcommand{\noin}{\noindent}

\newcommand{\med}{\medskip}

\newcommand{\ov}{\overline}

\newcommand{\To}{\rightarrow}

\newcommand{\A}{{\mathbb A}}
\newcommand{\C}{{\mathbb C}}

\newcommand{\Z}{{\mathbb Z}}
\newcommand{\BP}{{\mathbb P}}
\newcommand{\Q}{{\mathbb Q}}

\newcommand{\PP}{\mathbb{P}}

\newcommand{\SO}{{\mathcal O}}

\newcommand{\Aut}{{\rm Aut}}
\newcommand{\Bk}{{\rm Bk}\:}  %Bark
\newcommand{\GL}{{\rm GL}}

\newcommand{\lkd}{\ol{\kappa}}
\newcommand{\MNC}{{\rm MNC}}  %minimal normal crossings
\newcommand{\mult}{{\rm mult}}
\newcommand{\NC}{{\rm NC}}
\newcommand{\Pic}{{\rm Pic}\:}

\newcommand{\rank}{{\rm rank}\:}

\newcommand{\red}{{\rm red}\:}
\newcommand{\Sing}{{\rm Sing}\:}
\newcommand{\SNC}{{\rm SNC}}  %simple normal crossings

\newcommand{\Supp}{{\rm Supp}\:}

\newcommand{\dps}[1]{\displaystyle{#1}}
\newcommand{\wt}{\widetilde}
\newcommand{\ol}{\overline}
\newcommand{\wh}{\widehat}

\newcommand{\is}[2]{({#1}\cdot{#2})}

\newtheorem{thm}{Theorem}[section]
\newtheorem{lem}[thm]{Lemma}
\newtheorem{prop}[thm]{Proposition}
\newtheorem{cor}[thm]{Corollary}

\newtheorem{remark}[thm]{{\sc Remark}}

\newtheorem{setup}[thm]{Setup}

\newenvironment{claim}[1]{\par\noindent{\bf Claim {#1}.}~\em}{\vspace{0mm}}
\newtheorem{obs}[thm]{Observation}
\numberwithin{equation}{section}
\begin{document}
\title[Singularities of a homology plane of general type]
{A homology plane of general type can have at most a cyclic quotient singularity}
\author{R.V. Gurjar, M. Koras*, M. Miyanishi** and P. Russell***}
\thanks{Research supported by the RIP-program at Mathematisches Forschungsinstitut Oberwolfach}
\thanks{* Supported by Polish Grant MNiSW}
\thanks{** "Supported by Grant-in-Aid for Scietific
Research (c), JSPS}
\thanks{*** Supported by NSERC, Canada}
\date{}
\maketitle

\begin{abstract}
We show that a homology plane of general type has at worst a single cyclic quotient singular point.
An example of such a surface with a singular point does exist \cite[Theorem 3.1]{MS}. We also show that the automorphism group of a smooth contractible surface of general type is cyclic.
\end{abstract}

\setcounter{section}{-1}
\section{Introduction.}
Recall that a normal affine surface $V$ such that the singular homology groups $H_i(V,\Z)=(0)$
for $i>0$ is called a {\em $\Z$-homology plane} (simply a {\em homology plane}). Similarly,
{\em a $\Q$-homology plane} is defined. In \cite[Lemma 3.2]{GKMR} it was shown that if $V$ is
a $\Q$-homology plane such that the logarithmic Kodaira dimension $\lkd(V\setminus\Sing~V)=2$, then
$V$ has at most one singular point and it is a quotient singular point. In particular, $V$ is rational
by \cite{GPS, GP}.
\svskip

In this paper we will prove the following result.
\svskip

\noindent
\textbf{Theorem 1.\ } {\em Let $S'$ be a $\Z$-homology plane. If $\lkd(S'\setminus \Sing S')=2$, then
$S'$ has at most one singular point and it is a cyclic quotient singularity.}
\svskip

As an application of Theorem 1 we will prove the following result about the automorphism group
of a smooth contractible (hence affine) surface of general type.
\svskip

\noindent
\textbf{Theorem 2.\ } {\em Let $V$ be a smooth homology plane  with $\lkd(V)=2$. Suppose that $\Aut(V)\neq \{1\}$. Then we have the
following assertions.
\begin{enumerate}
\item[{\rm (1)}]
$\Aut(V)$ has a unique fixed point $p$ and the action of $\Aut(V)$ is free on
$V\setminus\{p\}$.
\item[{\rm (2)}]
$\Aut(V)$ is a finite cyclic group isomorphic to a small subgroup of $\GL(2,\C)$ (in the induced action on the tangent space at $p$).
\item[{\rm (3)}]
The quotient surface $V':=V/\Aut(V)$ has a unique singular point $q$ and the local fundamental group at q is isomorphic to
$\Aut(V)$. If V is simply connected, then $\pi_1(V'\setminus{q})=\Aut(V)$.
\end{enumerate}}
\svskip

In \cite{GKMR}, we proved that a normal affine contractible surface $V$ such that $\pi_1(V\setminus\Sing V)=(1)$
is actually smooth. The proof used a deep result due to C. Taubes involving gauge theory.
Our proof of Theorem 1 relies entirely on algebraic surface theory and avoids Taubes's result. We first establish that the boundary divisor for a possible counterexample is \textit{star-shaped}, see Lemma \ref{Lemma 4.3}. We then present two quite different arguments to finish the proof.  In the first approach we use detailed computations of numerical data for the boundary divisor to show in the end that a counterexample does not exist. In the second approach we present a more conceptual argument based on observations concerning singular points on normal affine surfaces with good $\C^*$-actions. We have kept both approaches since the second is quite interesting in its own right, and since we believe that the first can eventually be sharpened to give a global bound for the order of the local fundamental group in Theorem 1.

\section{Preliminaries}
Our proof of Theorem 1 uses several results and terminology from the theory of open algebraic surfaces.
The book \cite{Mi} is a good reference for most of the terminology and basic results from this theory
used here. In order to keep the length of the paper to a minimum we refer the reader to \cite{Mi}
for any terms not defined here. We will also use some results from T. Fujita's paper \cite{F},
particularly \S\S 3, 4, 5, 6.
\svskip

We will only deal with complex algebraic or analytic varieties. We will write the canonical divisor
$K_X$ of a normal surface $X$ simply  $K$ when there is no possibility of confusion. A smooth
projective rational curve $C$ on a smooth projective surface $Z$ is a $(-n)$-curve if $C^{2}=-n$.
We will need this terminology mainly for $n=0,1,2$. A $\BP^{1}$-fibration on a normal projective surface
$Z$ is a morphism $f: Z \to B$  onto a smooth projective curve $B$ such that a general fiber of $f$ is
isomorphic to $\BP^{1}$. A simple normal crossing divisor $D$ on a smooth quasi-projective surface will
be called an $\SNC$ divisor. An irreducible component of an SNC divisor $D$ will be simply called
a component of $D$ if there is no confusion. For any
component $D_0$ of $D$ the connected components of $D-D_0$ which meet $D_0$ are called the {\em branches} of $D$ at $D_0$  and their number is called the {\em branching number} of $D_0$. A component $D_0$ is called a {\em branching component}
of $D$ if the branching number is $\geq 3$, i.e., if $D_0$ meets at least three other components of $D$.  If $D$ is an $\SNC$ divisor such that any
$(-1)$-curve in $D$ is a branching component of $D$ then we say that $D$ is an $\MNC$ divisor to mean a minimal normal
crossing divisor.

For an $\SNC$ divisor $D$ which is a tree (i.e., a connected union of rational curves without loops), we use the terms {\em tip, twig, maximal rational twig}, etc,
from \cite[\S 3]{F}. For an $\SNC$  divisor $D$ with components $D_i$ we denote by $d(D)$
the determinant of the matrix $(-(D_i\cdot D_j))$. We refer to \cite[2.1.1]{KR} for rewriting $d(D)$
in terms of the determinants of a branching component and its branches.

For a smooth projective surface $X$ and a (possibly disconnected) \SNC -divisor $D$ on $X$,
if $\lkd(X\setminus D)\geq 0$, then $K+D$ has a Zariski-Fujita decomposition $K+D\approx P+N$ where $\approx$ denotes the numerical equivalence. Here $P,N$ are
$\Q$-divisors, $P$ is nef, $N$ is effective, the intersection form on the irreducible components of
$\Supp N$ is negative definite and $P\cdot D_i=0$ for any component $D_i$ in $\Supp N$. If $N$ is not
equal to the {\em bark} of $D$, denoted by $\Bk(D)$, then $X$ contains a $(-1)$-curve $C$ not
contained in $D$ which meets $D$ transversally in at most two points. If $C$ meets $D$ in two points,
the intersection points on $D$ belong to  components $D_1, D_2$ contained in distinct
connected components $A, B$. Furthermore, either $A$ and $B$ are linear chains, or $A$ is a fork and
$B$ is a linear chain \cite[Chapter 2, 3.7]{Mi}. See also \cite[Lemma 6.20]{F}. The bark of an
admissible rational fork $T$ is denoted by $\Bk(T)$ in \cite{Mi} and by $\Bk^*(T)$ (and called the
{\em thicker} bark) in \cite{F}. To avoid a confusion, we employ the notation in \cite{Mi} and
call it the bark of a fork. We say that the pair $(Y,F)$ is {\em almost minimal}, or that the divisor $F$ is {\em almost-minimal} if the reference to $Y$ is clear,
if $N=\Bk(F)$ and $P=K_Y+F^\#$ with $F^\#=F-\Bk(F)$. Given a pair $(X,D)$ as above, there exists a birational morphism
$\sigma : X \to Y$ such that $Y$ is a smooth projective surface, $F=\sigma_*(D)$ is an SNC divisor
and $(Y,F)$ is almost minimal. By \cite[Theorem 3.11]{Mi}, we have $\sigma_*(K_X+D^\#) \ge K_Y+F^\#$.
If $(X,D)$ is almost minimal, the nef part $P=K_X+D^\#$ is obtained by peeling off the bark $\Bk(D)$
by the theory of peeling \cite{Mi}.

We say that an $\SNC$ divisor $T$ is {\em contractible} if it contracts to a quotient singular point or
a smooth point. Note that the intersection matrix of $T$ is then negative definite and $d(T)>0$.

As a matter of notation, if $D, D'$ are reduced effective divisors  with D' a part of D, $D-D'$ denotes the union of the irreducible components of $D$ that are not components of of $D'$. This is subtly different from the the set-theoretic difference $D\setminus D'$.

We will need the following useful result about singular fibers of a $\BP^1$-fibration $f : Z \to B$
on a smooth projective surface $Z$ \cite[Chapter 1, Lemma 2.11.2]{Mi}.

%Lemma 1.1
\begin{lem}\label{Lemma 1.1}
Let $F_0$ be a scheme-theoretic fiber of $f$ which is not isomorphic to $\BP^1$. Then we have
the following assertions.
\begin{enumerate}
\item[{\rm (a)}]
$F_{0}$ contains a $(-1)$-curve.
\item[{\rm (b)}]
If a $(-1)$-curve $E$ in $F_{0}$ occurs with multiplicity $1$ in $F_{0}$ then $F_{0}$ contains another
$(-1)$-curve.
\item[{\rm (c)}]
By successively contracting $(-1)$-curves in $F_{0}$ we can contract $F_{0}$ to a smooth fiber of the
induced $\BP^1$-fibration on the surface obtained from $Z$ by blowing downs.
\item[{\rm (d)}]
The dual graph of $F_0$ is a tree.
\item[{\rm (e)}] A $(-1)$-curve in $F_0$ of multiplicity $m > 1$ meets at most two other irreducible components of $F_0$ and if $m=1$ it meets precisely one.
\item[{\rm (f)}]
Suppose that $F_0$ contains exactly one $(-1)$-curve $C$ and let $F_0\setminus C=R \cup T$ be
the decomposition into connected components. Then $F_0$ contains at most two components with
multiplicity $1$. If $R$ contains an irreducible component of multiplicity $1$, then $T$ is a chain.
If both $R$ and $T$ contain a component with multiplicity $1$, then $F_0$ is a chain.
\end{enumerate}
\end{lem}

We will use the following consequence of Lemma \ref{Lemma 1.1}, which is probably well-known
to experts.

%Lemma 1.2
\begin{lem}\label{Lemma 1.2}
Let $f : Z \to B$ be a $\BP^1$-fibration on a normal projective surface $Z$ and let $F_{0}$ be a
scheme-theoretic fiber of $f$ which is not a smooth fiber. Then $Z$ has at worst rational
singularities at points on $F_{0}$. If an irreducible component $C$ of $F_{0}$ occurs with multiplicity
$1$ in $F_{0}$ then $Z$ is smooth at any point on $C$ which does not lie on any other irreducible
component of $F_{0}$.
\end{lem}

Let $C$ be a $(0)$-curve on a smooth projective rational surface $Z$. Then the linear system $|C|$ induces a $\BP^1$-fibration on $Z$ with $C$ as a fiber. We record the following consequences of this fact which we will use repeatedly below.

%Lemma 1.2.0
\begin{lem}\label{Lemma 1.2.0}
Let $Z$ be a smooth projective rational surface and $T$ a $SNC$-divisor on $Z$ consisting of rational curves. Suppose one of the following.\\
(a) A maximal twig $L$ of $T$ is a non-contractible linear chain.\\
(b) A linear chain $L$ connecting two branching components of $T$ is not contractible.\\
Then by blowing up and down over $L$ we can replace $Z$ by a smooth projective surface $Z'$, $T$ by an SNC-divisor $T'$ and $L$ by a linear chain $L'$ such that $Z'\setminus T' = Z\setminus T$ and $L'$ contains a $(0)$-curve $C$ which is a tip of $L'$ and a tip of $T'$ in case (a).

$C$ induces a $\C$-fibration of $Z\setminus T$ in  case (a) and a $\C^*$-firation in case (b).
\end{lem}

Here by a $\C$- (resp. $\C^*$-) fibration of a surface $Y$ we mean a morphism to a curve with general fiber $\C$ (resp. $\C^*)$. By the easy addition theorem for Kodaira dimension (see \cite[Lemma 1.13.1]{Mi}) we have $\lkd (Y)=-\infty$ (resp. $\lkd (Y) \leq 1$) if $Y$ has a $\C$- (resp. $\C^*$-) fibration. So we obtain

%Corollary 1.2.0.1
\begin{cor}\label{Corollary 1.2.0.1}
If $\lkd (Z\setminus T) =2$, then an $L$ as in Lemma \ref{Lemma 1.2.0} does not exist.
\end{cor}

In several places below we use the following formula of Fujita \cite[4.15,4.16]{F}.

%Lemma 1.3
\begin{lem}\label{Lemma 1.2.1}
Let $f : Z \to B$ be a $\BP^1$-fibration on a smooth projective surface $Z$ and $D$ a reduced effective divisor on $Z$. Then
\[
h-\Sigma +\nu -2=\wt{b}_1-~\wh{b}_2=b_2(D)-b_2(Z).
\]
Here $h$ is the number of components of $D$ that are horizontal for $f$, $\nu$ is the number of fibers of $f$ contained in $D$, $\Sigma$ is the sum over all other fibers $F$ of $\sigma(F)-1$, with $\sigma(F)$  the number
of components of $F$ not in $D$, and $\wt{b}_1, ~\wh{b}_2$ are modified Betti numbers of $Z\setminus D$ introduced by Fujita in \cite[1.13]{F}.
\end{lem}

In our arguments an inequality of Bogomolov-Miyaoka-Yau type (simply the BMY-inequality) proved by R. Kobayashi,
S. Nakamura and F. Sakai plays a crucial role. The formulation best suited to our purpose is found in
\cite[Lemma 8 and Corollary 9]{GM} (see also see \cite[Chapter 2, Theorem 6.6.2]{Mi}) and stated as follows.

%Lemma 1.4
\begin{lem}\label{Lemma 1.3}
Let $(\ol{X}, \ol{D})$ be a relatively minimal log-projective surface with log-terminal singularities.
Namely, if $(X,D)$ be the minimal resolution of $(\ol{X},\ol{D})$ with $D$ the reduced sum of the
proper transform of $\ol{D}$ and the exceptional locus $E$, then $(X,D)$ is almost minimal. Suppose that
$\lkd(X\setminus D)\ge 0$ and that $X\setminus D$ is quasi-affine. \\
(i) We have the inequality
\[
\chi(X\setminus D)+\sum_{p}\frac{1}{|G_p|} \ge 0\ ,
\]
where $p$ ranges over all singular points of $\ol{X}\setminus \ol{D}$ and $G_p$ is the local fundamental group.\\
(ii) If $\lkd(X\setminus D)=2$ the above inequality is strict.
\end{lem}
Here and later we use $\chi$ to denote topological Euler characteristic. The original BMY-inequality is the case (ii), and A. Langer \cite{L}
has extended it in the form (i). We will use both versions.
\svskip

Now let $S'$ be as in the statement of Theorem 1. Let $\pi : S \to S'$ be a minimal resolution of
singularities. Let $E=\pi^{-1}(\Sing S')$. We consider $E$ as a reduced divisor. Let $\ol{S} \supset
S$ be a smooth projective surface such that $D:=\ol{S}\setminus S$ is an \SNC-divisor.
{\em We will assume that $D$ is an $\MNC$ divisor.} By \cite[Proposition 2.6 (iii)]{KR}, {\em the irreducible
components of $D\cup E$ form a basis of $\Pic(\ol{S})\otimes \Q$. Hence they are numerically
independent}.

%Lemma 1.5
\begin{lem}\label{Lemma 1.4}
Under the assumption of Theorem 1, $S'$ has only one singular point q and it is a quotient
singularity. Furthermore, $S'$ is a rational surface and $H_1(S;\Z)=H_1(S';\Z)=0$.
\end{lem}
\begin{proof}
The proof of the first assertion is the same as the proof of Lemma 3.2 in \cite{GKMR}. For the
second assertion, we have $\pi_1(S)=\pi_1(S')$ by \cite[Lemma 2.3.1]{KR}, which implies
$H_1(S;\Z)=H_1(S';\Z)$. Since $S'$ is $\Z$-acyclic, $H_1(S';\Z)=0$. For the rationality of $S'$, we refer
to \cite{GP}.
\end{proof}

An irreducible curve $C$ in $\ol{S}$ is called a {\em simple curve} if $C \simeq \BP^1, C \not\subset
D\cup E, C\cdot D=1$ and $C\cdot E\leq 1$. The assumption that $\lkd(S'\setminus \Sing S')=2$ implies the
following result.

%Lemma 1.6
\begin{lem}\label{Lemma 1.5}
There is no simple curve $C$ in $\ol{S}$.
\end{lem}

\begin{proof}
See the proof of 2.12 in [KR].
\end{proof}

In the terminology of \cite[Chapter 2, 4.3]{Mi} we state this result as:

%Corollary 1.7
\begin{cor}\label{Corollary 1.6}
The pair $(\ol{S}, D+E)$ is almost minimal.
\end{cor}

%Lemma 1.8
\begin{lem}\label{Lemma 1.7}
The following assertions hold.
\begin{enumerate}
\item[{\rm (1)}]
$D$ is a rational tree.
\item[{\rm (2)}]
$|d(D)|=d(E) > 0$.
\end{enumerate}
We put a=d(E).
\end{lem}

\begin{proof}
For (1), see for instance the proof of 2.6 (iv) in \cite{KR}. The proof of (2) is similar to the proof of 2.9
in \cite{KR}.
\end{proof}

We set up some notation.

%Setup 1.9
\begin{setup}\label{Setup 1.8}
Let $\ol{Y}$ be a smooth projective surface and $T$ a reduced $\MNC$ divisor in $\ol{Y}$. Let
$Y=\ol{Y}\setminus T$. Suppose that the pair $(\ol{Y},T)$ is not almost minimal. Let
$(\ol{Y'},T')$  be an almost minimal model of $(\ol{Y},T)$.
\begin{enumerate}
\item[{\rm(a)}]
$\ol{Y'}$ is obtained from $\ol{Y}$
by a sequence of birational morphisms $p_i : \ol{Y}_i \to \ol{Y}_{i+1}, \ol{Y}=\ol{Y}_0 \to
\ol{Y}_1 \to \cdots \to \ol{Y}_\ell=\ol{Y}'$. Let $T_i=(p_{i-1})_*(T_{i-1}), T_0=T, T'=T_\ell$. Let $Y_i=\ol{Y}_i\setminus T_i$. For every $i$ there exists a $(-1)$-curve $C_i \nsubseteq T_i$ such that
$p_i : \ol{Y}_i \to \ol{Y}_{i+1}$ is the contraction of all smoothly
contractible components of $C_i + T_i$ without losing the normal
crossing condition (which we call the {\em NC-minimalization}). Finally, for the almost minimal model $(\ol{Y}',T')$, the negative part
$(K_{\ol{Y}'}+T')^-$ coincides with the bark $\Bk(T')$. The contractions in this process involve only  curves (or their images) contained in the negative part of
$(K_{\ol{Y}}+T)^-$, and $l$ is the number of such curves that meet $Y$.
\item[{(b)}] We put
\[
e(\ol{Y_i},T_i)=\chi(\ol{Y_i}\setminus T_i)+\#\{\mbox{connected components of $T_i$}\}.
\]

\end{enumerate}
\end{setup}
\svskip
We point out that here $T_i$ is the direct image of $T_{i-1}$ in the sense of divisors. In particular, an isolated point in $p_{i-1}(T_{i-1})$ becomes part of $Y_i$. We find by an elementary calculation

%Lemma 1.10
\begin{lem}\label{Lemma 1.9}
\begin{enumerate}
\item[{(a)}]
  $\chi(Y_{i+1})=\chi(Y_{i})-1$ if and only if either
$\is{C_i}{T_i}=1$ and $T_i+C_i$ does not contract to a smooth point  or $\is{C_i}{T_i}=0$.\\ $\chi(Y_{i+1})=\chi(Y_i)$ (resp. $\chi(Y_{i+1})=\chi(Y_i)+1$) if and only if either $\is{C_i}{T_i}=1$ and $T_i+C_i$ contracts to a smooth point or $C_i$
meets exactly two connected components $T_i^{(1)}$ and $T_i^{(2)}$ of $T_i$ and $T_i^{(1)}+C_i+T_i^{(2)}$
does not contract (resp. does contract) to a smooth point.
\item[{(b)}]
$e(\ol{Y}',T')=e(\ol{Y},T)-\ell$.
\end{enumerate}
\end{lem}
\svskip

As to $(b)$, each $C_i$ contracted in the above process meets $T$, or its image, transversally in
at most $2$ points. One easily checks that the quantity $e(-,-)$ decreases by $1$ in such a contraction.

The following result will be used later.

%Lemma 1.11
\begin{lem}\label{Lemma 1.10}
There is no affine, normal, rational surface $Z'$ which satisfies the following conditions:
\begin{enumerate}
\item[{\rm (a)}]
$\lkd(Z'\setminus \Sing Z')=0$ or $1$.
\item[{\rm (b)}]
$Z'$ is a $\Z$-homology plane.
\item[{\rm (c)}]
$\Sing Z'$ consists of two points, each of type $A_1$.
\end{enumerate}
\end{lem}

\begin{proof}
Suppose $Z'$ exists. Let $Z \to Z'$ be a resolution of singularities and let $E=E_1+E_2$ be the exceptional
locus. $E_1, E_2$ are $(-2)$-curves. Let $p_1,p_2$ be the singular points of $Z'$. Let $Z_0:= Z'\setminus \Sing Z'$.
\svskip

\begin{claim}{\hspace{-1mm}}
There is no $\C^*$-fibration on $Z_0$.
\end{claim}
\svskip

Suppose that such a $\C^*$-fibration exists, say $f : Z_0 \to B$, where $B$ is a smooth curve.
If $f$ extends to a $\C^*$-fibration on $Z'$ then by \cite[Theorem 2.17]{MS} $Z'$ has at most one singular
point. Hence $f$ is not defined at one of the points $p_i$. By elimination of indeterminacies we see that
a resolution of singularities of $Z'$ has $\lkd=-\infty$. Since $p_1,p_2$ are rational double points
we deduce from this that $\lkd(Z'\setminus \Sing Z')=\lkd(Z')=-\infty$, a contradiction. This proves the claim.

By Kawamata's theorem (see \cite[Theorem 6.1.5]{Mi})for open surfaces with $\lkd=1$, it follows that $\lkd(Z_0)=0$. By standard topological
argument $\Pic(Z_0)$ is torsion (see \cite[Proposition 2.6, (iii)]{KR}). It follows that $Z_0$ is relatively
minimal (see 4.2 in \cite{KR}.) Let $Z \subset \ol{Z}$ be a smooth compactification of $Z$ such that
$T=\ol{Z}\setminus Z$ is an $\MNC$-divisor. Fujita classified connected components of boundary divisors
of an almost-minimal surface of Kodaira dimension $0$ in \cite[\S 8.8]{F}. We have $d(T)=-d(E)=-4$ by the same
argument as in Lemma \ref{Lemma 1.7}. By examining all possibilities for $T$ we come to a contradiction. Details are in section
$4$ of \cite{KR} (the assumptions in \cite{KR} are bit different, but the arguments used in section $4$ work
in the present situation, a lot of cases considered in \cite{KR} in fact are ruled out by condition $d(T)=-4$).
\end{proof}

The following result shows that the conclusion of Theorem 1 holds if we assume that $\Sing S'$ consists of a unique quotient singular point and $\lkd (S'\setminus \Sing S')\geq 0$.

%Lemma 1.12
\begin{lem}\label{Lemma 1.11}
Suppose that $U'$ is an
affine, normal $\Z$-homology plane with exactly one quotient singular point $q$ which is non-cyclic.
If $\lkd(U'\setminus \{q\})\geq 0$ then $\lkd(U'\setminus \{q\})=2$ .
\end{lem}
\begin{proof}
Suppose that $\lkd(U'\setminus\{q\})<2$. As in the proof of Lemma \ref{Lemma 1.10} above, $U'\setminus \{q\}$ is not
$\C^*$-ruled. Hence $\lkd(U'\setminus \{q\})=0$. This case is ruled out in \cite{KR}, section $4$. (In \cite{KR}
it is assumed that surfaces under consideration are topologically contractible. But in section $4$ it is enough
to assume that they are $\Z$-homology planes.)
\end{proof}

%Lemma 1.13
\begin{lem}\label{Lemma 1.12}
Suppose that $Z'$ is a $\Q$-homology plane with only quotient singularities. Then $Z'$ is affine.
Moreover, if $H_1(Z';\Z)=0$ then $Z'$ is a homology plane.
\end{lem}
\begin{proof}
For the affiness of $Z'$ we refer to \cite[Theorem 2.4]{F}. The proof given there in the smooth case
can be easily extented to our setup with quotient singularities. If $Z'$ is affine, then $H_2(Z';\Z)$
is a free abelian group. Hence it vanishes if $b_2(Z')=0$. So $Z'$ is a homology plane in this case.
\end{proof}

\svskip
Let $G$ be a non-cyclic small finite subgroup of $GL(2, \C)$. The surface $(\C^2\setminus \{0\})/G$ has
the structure of a Platonic fibration and its $\MNC$ boundary divisor consists of two forks, $P_0$ and $P_1$, say.
One, $P_0$ say, is the resolution of the origin in $\C^2 /G$ and contractible.
We will call $P_1$ the {\em complementary fork} of $P_0$.  The possibilities for the forks $P_0$ and their complementary forks are well known and a complete list is given in \cite[Chapter 3, section 2.5]{Mi}. If the fork $T$ has has branches $T_1, T_2, T_3$ we say that it is of type $(d(T_1), d(T_2), d(T_3))$. For a contractible fork or a complementary fork we have $d(T_i)\geq 2$ and
$$\frac{1}{d(T_1)}+\frac{1}{d(T_2)}+\frac{1}{d(T_3)}>1$$.

We will later need the following geometric characterization of this situation by Miyanishi and Tsunoda  \cite[5.1 and Theorem 5.1.2]{Mi}
\svskip

%Lemma 1.14
\begin{lem}\label{Lemma 1.13}
Let $X$ be a smoth projective surface and $(X,P)$ be an almost minimal pair. Put $X'=X\setminus P$. If $\lkd{(X')}=-\infty$ and if $X'$ is not affine ruled, then one of the following occurs.
\begin{enumerate}
\item[{\rm (a)}] All connected components of $P$ are contractible.
\item[{\rm (b)}] $P$ has two connected components, $P_0$ and $P_1$. One, $P_0$ say, is a contractible fork,
and the other, $P_1$, is the complementary fork of $P_0$, $X'$ has the structure of a Platonic fibration and $\chi{(X')}=0$. If $X''$ is the surface obtained from $X\setminus P_1$ by contracting $P_0$ to a quotient singular point $p$, then $X'' \simeq \C^2/G$ where $G$ is the local fundamental group at $p$.
\end{enumerate}
\end{lem}

 Throughout this article, we use the following terminology and
notation concerning a maximal twig T of a divisor Q.  Let $T_1$ (resp. $T_k$) denote the component of $T$ that meets a branching component (resp. is a tip) of $Q$. If $n=d(T)$, we write $\ol{n}=\ol{d}(T)=d(T-T_1)$ and $\wt{n}=\wt{d}(T)=d(T-T_k)$.
$\ol{n}/n=d(T-T_1)/d(T)$ (resp. $\wt{n}/n=d(T-T_k)/d(T)$) is called the {\em capacity} (resp. the {\em inductance}) of
$T$. These quantities are crucial in the calculation of $d(Q)$ (resp. $Bk(Q)$).

\section{Determination of the dual graph of $D$}

We keep the notation of Theorem 1 and Lemma \ref{Lemma 1.4}. We assume that we have a counterexample to Theorem 1, i.e., that $q$ is a non-cyclic quotient singularity, and will eventually arrive at a contradiction. We assume that the number of irreducible components of the boundary $D$ is minimal for such a counterexample. We record some facts and establish some notation.

%Setup 2.1
\begin{setup}\label{Setup 2.1}
\begin{enumerate}
\item[{\rm (a)}]
$E$ is a fork.
\item[{\rm (b)}]
Let $D_0$ be a branching component of $D$. Let $D-D_0=Z_1+\dots+Z_r$ be the decomposition
into connected components, where $r \ge 3$. Let $Z_i^{(1)}$ denote the irreducible component of $Z_i$
which meets $D_0$. It is possible that $D-D_0$ is not MNC. Let $f : \ol{S} \to \ol{Y}$ be a sequence
of blowing downs of (-1)-curves to make the image of $D-D_0$ an $\MNC$ divisor. Let
$Q_i=f_*(Z_i)\ (1 \le i \le r)$, $Q_0=f_*(E)$ and $Q=Q_0+(Q_1+\cdots+Q_r)$,
 If $Q_i$ contracts to a quotient singular point $q_i$, we let $G_i$ denote
the local fundamental group at $q_i$. Let $Y=\ol{Y}\setminus Q$.
\item[{(c)}]
Assume that no $Z_i$ contracts to a smooth point.Then we have
\[
\chi(Y)=2-r \quad \mbox{and} \quad \rank\Pic(\ol{Y})=\#Q+1.
\]
\end{enumerate}
\end{setup}

Here and later we denote by $\#Z$ the number of irreducible components of a (reduced, effective)
divisor $Z$. Note that by Corollary \ref{Corollary 1.2.0.1} the boundary divisor $D$ is not a linear chain.
So $D$ has a branching component $D_0$.

Furthermore, the $(-1)$ curves on $\ov{S}$ contracted in the $\NC$-minimal\-ization $f$ are among
the components $Z_i^{(1)}$. The assumption in $(c)$ is not automatic. When it is satisfied, $Y$ is seen to be the union of $S'\setminus \Sing S'$ and the component $D_0$ punctured in
$r$ points. This yields $\chi(Y)$. For $\rank\Pic(\ol{Y})$, see \cite[2.6, (iii)]{KR}.

\svskip

We will need the following technical result which plays an important role in the sequel.

%Lemma 2.2
\begin{lem}\label{Lemma 2.2}
Let $\ol{Y}$ be a smooth projective rational surface, let $Q$ be an $\MNC$-divisor consisting of
smooth rational curves and let $Y=\ol{Y}\setminus Q$. Assume that one of the following two cases takes place.
\begin{enumerate}
\item[{\rm (i)}]
The divisor $Q$ consists of $r+1$ connected components, $Q=Q_0+Q_1+\cdots+Q_r$, where $Q_0$ is a contractible fork
and $r \ge 3$. Further, $\chi(Y)=2-r, \rank\Pic\ol{Y}=\#Q+1$ and
$H_1(\ol{Y}\setminus(Q_1\cup\cdots \cup Q_r);\Z)=0$.
\item[{\rm (ii)}]
The divisor $Q$ consists of $r$ connected components, $Q=Q_1+\cdots+Q_r$, where $r \ge 3$. Further,
$\chi(Y)=2-r$ and $\rank\Pic(\ol{Y})=\#Q$.
\end{enumerate}
Assume furthermore that the following conditions are satisfied in each of the above two cases.
\begin{enumerate}
\item[{\rm (1)}]
$\lkd(Y) \ge 0$.
\item[{\rm (2)}]
If a $(-1)$-curve $C$ contained in $\Supp(K_{\ol{Y}}+Q)^-$ meets two connected components $Q_i$ and
$Q_j$, then $Q_i+C+Q_j$ does not contract to a smooth point, where $(K_{\ol{Y}}+Q)^-$ is the negative
part in the Zariski-Fujita decomposition of $K_{\ol{Y}}+Q$.
\item[{\rm (3)}]
The irreducible components of $Q$ are independent in $\Pic(\ol{Y})$.
\item[{\rm(4)}] No $Q_i$ contracts to a smooth point.
\end{enumerate}
Then the pair $(\ol{Y},Q)$ is almost minimal.
\end{lem}

\begin{proof}
Suppose that $(\ol{Y},Q)$ is not almost minimal. Let $(\ol{Y}',Q')$ be an almost minimal model of
$(\ol{Y},Q)$. Let $(p_i)_{1 \le i\le \ell}$ be a sequence of birational morphisms as in \ref{Setup 1.8}. We claim that $\chi(Y_{i+1}) \le \chi(Y_i)$. If it is not the case, then by \ref{Lemma 1.9}
there exists a $(-1)$-curve $C$ on $Y_i$ meeting exactly two connected components of the image of
the boundary $Q$ on $Y_i$ such that $C$ together with these two components contracts to a smooth point. Then we find curves $\tilde C_1, \ldots,
\tilde C_i$ on $\ol{Y}$, the pre-images of $C_1, \ldots,
C_i$, and we split $Q=Q^{(1)}+Q^{(2)}$, each $Q^{(i)}$ a sum of some $Q_j$, such that $\tilde C_1+\ldots + \tilde C_i+Q^{(1)}$ contracts to a smooth point and $\tilde C_1+\ldots + \tilde C_i+Q^{(1)}$ and $Q^{(2)}$ are disjoint.
Then the intersection matrix of $Q^{(1)}+\tilde C_1+\dots +\tilde C_i$ is negative definite.
It follows that the irreducible components of $Q+\tilde C_1+\cdots+\tilde C_i$ are numerically independent in
$\Pic(\ol{Y})$. This implies that $i=1$. However, such a $(-1)$-curve $C$ does not exist by  assumption (2)
since $C \subset (K_{\ol{Y}}+Q)^-$. Thus
\begin{equation}\label{eqn 2.1}
\chi(\ol{Y}'\setminus Q')\leq \chi(Y).
\end{equation}
Label as $Q'_1,\dots Q'_k$  all the connected components of $Q'$ which contract to quotient singularities
$q'_i$. ($Q_j'$ is not necessarily the image of $Q_j$.) Note that the image of $Q_0$ is contained in one of the $Q'_j$ in case (i).
Let $G_i'$ be the local fundamental group at $q_i'$. We apply the BMY inequality in Lemma \ref{Lemma 1.3}
and obtain the inequality
\begin{equation}\label{eqn 2.2}
\chi(\ol{Y}'\setminus Q')+\sum_{i=1}^{k}\frac{1}{|G_i'|} \ge 0\ .
\end{equation}
We have $\#Q'\geq \Pic(\ol{Y}')$ since some $(-1)$-curve $C \not\subset Q$ is contracted to obtain
the almost minimal model $(\ol{Y}',Q')$. Hence not all connected components of $Q'$ contract to quotient
singularities, i.e., the number of connected components of $Q'$ is strictly greater than $k$. Since
$e(\ol{Y},Q)=2-r+r+1 =3$ in case (i) and $e(\ol{Y},Q)=2$ in case (ii) and since $\ell \ge 1$,
we have (see \ref{Setup 1.8} and \ref{Lemma 1.9})
\begin{eqnarray}\label{eqn 2.3}
k & \leq & e(\ol{Y}',Q')-\chi(\ol{Y}'\setminus Q')-1 \nonumber \\
&=& e(\ol{Y},Q)-\ell-1-\chi(\ol{Y}'\setminus Q')    \\
    & \leq&  3-\ell-1-\chi(\ol{Y}'\setminus Q') \leq  1-\chi(\ol{Y}'\setminus Q').  \nonumber
    \end{eqnarray}
Hence we obtain $\chi(\ol{Y}'\setminus Q') \leq 1-k$. From (\ref{eqn 2.2}), since $|G_i|\geq 2$,
we have $1-k+\frac{k}{2}\geq 0$. This implies $k \leq 2$. Since $\chi(\ol{Y}'\setminus Q') \le
\chi(Y)\leq -1$, (\ref{eqn 2.2}) implies that $k=2$ and $\chi(\ol{Y'}\setminus Q')\geq -1$.  Hence
$\chi(\ov Y'\setminus Q)=\chi(Y)=-1$ and $r = 3$. Also $|G_1'|=|G_2'|=2$. It follows that
$\lkd(\ol{Y'}\setminus Q')=0$ or $1$, for otherwise we have a strict inequality in (\ref{eqn 2.2}).
Meanwhile, (2.3) gives $\ell=1$ and $e(\ol{Y},Q)=3$. This implies that we are in the case (i) and
the morphism $\ol{Y} \to \ol{Y}'$ consists of the contraction $p_1$ of a single $(-1)$-curve $C_1$.
Since $Q_0$ is a contractible fork and $|G_1'|=|G_2'|=2$, $Q_0$ must be touched under $p_1$, i.e. $C_1$ meets $Q_0$.
Either $C_1$ meets only $Q_0$ and $Q_0+C_1$ contracts to a smooth point, or $C_1$ meets $Q_0$ and $Q_1$, say
and $C_1+Q_0+Q_1$ contracts to a $(-2)$-curve $Q'_1$. In both cases, there remains a non-contractible
connected component $Q'_3$ in $Q'$, where one can identify $Q_3$ with $Q'_3$ because $Q_3$ is untouched
by $p$. In the former (resp. the latter) case, we may assume that $Q'_1,Q'_2$ are $(-2)$-curves
(resp. $Q'_2$ is a $(-2)$-curve). Consider the surface $Z=\ol{Y'}\setminus Q'_3$. Let $Z'$ be obtained
from $Z$ by contracting $Q'_1, Q'_2$ to singular points. Since $\ol{Y}\setminus(Q_1\cup Q_2\cup Q_3)
\subset \ol{Y}\setminus Q_3$ and $H_1(\ol{Y}\setminus(Q_1\cup Q_2\cup Q_3);\Z)=0$, we have
$H_1(\ol{Y}\setminus Q_3;\Z)=0$. Since the surface $Z$ is obtained from $\ol{Y}\setminus Q_3$
by a sequence of blowings-down, we have $H_1(Z;\Z)=0$ and also $H_1(Z';\Z)=0$. On the other hand,
$\Pic(\ol{Y}')\otimes \Q$ is freely generated by irreducible components of $Q'$. In fact, the components
of $Q_0+C_1+Q_1+Q_2$ are contracted to points and hence numerically independent. Since $Q_3$ is disjoint
from $Q_0+C_1+Q_1+Q_2$, the components of $Q_3$ are independent. Since $\rank\Pic(\ol{Y})=\#Q+1$,
the irreducible components of $Q+C_1$ form a $\Q$-basis of $\Pic(\ol{Y})\otimes \Q$. Hence
$\Pic(\ol{Y}')\otimes \Q$ is freely generated by the components of $Q'$ and $\Pic(Z')\otimes \Q=0$.
This implies, together with $H_1(Z';\Z)=0$, that $\Pic(Z')=0$. It follows that $b_2(Z')=0$. We have
\[
\chi(Z')=\chi(Z\setminus(Q_0'\cup Q_1'))+2=\chi(\ol{Y'}\setminus Q')+2=-1+2=1.
\]
Thus $b_3(Z')=0$, while $b_4(Z')=0$ because $Z'$ is not compact. Hence $Z'$ is $\Q$-acyclic.
By Lemma \ref{Lemma 1.12}, $Z'$ is affine, whence $H_2(Z';\Z)$ is a free group and $H_2(Z';\Z)=0$
because $H_2(Z';\Q)=0$. Thus $Z'$ is $\Z$-acyclic. Now we reach a contradiction by Lemma \ref{Lemma 1.10}.
\end{proof}

If the condition (2) in Lemma \ref{Lemma 2.2} is violated, it reflects upon the shape of the
divisor $Q$ on the surface $\ol{Y}$.

%Lemma 2.3
\begin{lem}\label{Lemma 2.3}
Let the situation be as in 2.1. Suppose that $\lkd(Y)\geq 0$ and that a $(-1)$-curve $C \subset(K_{\ol{Y}}+Q)^-$
meets two connected components, say $Q_{i_0}$ and $Q_{i_1}$, of $Q$, so that $Q_{i_0}+C+Q_{i_1}$ contracts
to a smooth point. Then we have
\begin{enumerate}
\item[{\rm (1)}]
$r=3$.
\item[{\rm (2)}]
$C$ meets $E$. After relabelling the connected components of $Q$, we have $Q_{i_0}=E$ and $Q_{i_1}=Q_1$.
\item[{\rm (3)}]
$Q_1$ is a contractible chain and $Q_2$, $Q_3$ are trees with unimodular intersection matrix.
None of $Q_2, Q_3$ is a chain. One of $Q_2, Q_3$ is non-contractible.
\end{enumerate}
\end{lem}

\begin{proof}
Let $\pi : \ol{Y} \to \ol{Y}'$ be the contraction of $Q_{i_0}+C+Q_{i_1}$ to a smooth point.
We infer as in the proof of Lemma \ref{Lemma 2.2} that the irreducible
components of $Q+C$ are independent in $\Pic(\ol{Y})$. Let $Q'=\pi_*(Q)$ and $Y'=\ol{Y}'\setminus Q'$.
We write $\pi(Q_i)=Q_i$ for $i \ne i_0,i_1$. Now $\rank \Pic(\ol{Y}')=\#Q'$ and $\chi(Y')=3-r$.

Suppose that $r\ge 4$. Then $\chi(Y') \le -1$. Suppose that there exists a $(-1)$-curve $C'$ which meets
precisely two connected components of $Q'$, say $Q_{j_1}$ and $Q_{j_2}$, and that $Q_{j_1}+C'+Q_{j_2}$
contracts to a smooth point. Again the irreducible components of $Q'+C'$ are independent in $\Pic(\ol{Y}')$.
This is a contradiction since $\rank\Pic(\ol{Y}') = \#Q'$. Note that $\chi(Y')=3-r=2-(1+r-2)$
since $r-1$ connected components of $Q$ survive in $Q'$. So, $(\ol{Y}', Q')$ satisfies the assumptions
of Lemma \ref{Lemma 2.2}, case (ii), and hence is almost minimal.

Let $k$ be the number of connected components of $Q'$ which contract to quotient singularities. We have
$k <$ (number of connected components of $Q'$) since otherwise we have a contradiction to the Hodge
Index Theorem. So $k < r-1$, i.e. $k \le r-2$. Now the BMY inequality gives
\begin{equation}\label{eqn 2.3}
\frac{k}{2}+\chi(Y')\geq 0\ .
\end{equation}
This implies $\frac{r-2}{2}+3-r\geq 0$, hence $r\leq 4$. So, $r=4$, $k=2$ for two singular points
of type $A_1$ and $\lkd(\ol{Y}'\setminus Q') < 2$ since equality occurs in (\ref{eqn 2.3}). Now $Q'$
has $3$ connected components, two of which are $(-2)$-curves. It follows that $E=Q_0$ must be contracted together with $C$ and $Q_1$, say. We may assume that $Q_2, Q_3$ are (-2)-curves.
Consider the surface $V'$ which is obtained from $\ol{Y}'\setminus Q_4$ by contracting $Q_2, Q_3$ to
singular points of type $A_1$. Note that $H_1(\ol{Y}\setminus Q_4;\Z)=0$ because $\ol{Y}\setminus Q_4$
contains $S=\ol{Y}\setminus(Q_1\cup\cdots\cup Q_4)$. Then we have $H_1(V'; \Z)=0$ since $V'$ is obtained
from $\ol{Y}\setminus Q_4$ by a proper birational morphism. Since $\Pic(V')\otimes \Q=0$ and $H_1(V';\Z)=0$,
it follows that $\Pic(V')=0$ and hence $H^2(V';\Z)=0$. An easy calculation gives $\chi(V')=1$, hence
$b_3(V')=0$. Thus $V'$ is $\Q$-acyclic. By Lemma \ref{Lemma 1.12}, $V'$ is affine $\Z$-acyclic.
Since $\lkd(V')=0$ or $1$ we obtain a contradiction by Lemma \ref{Lemma 1.10}.

Hence $r=3$. Suppose that $E\cdot C=1$ and $C\cdot Q_1=1$, where $Q_1 \ne Q_0=E$. Consider the surface
$W=\ol{Y}\setminus (Q_2\cup Q_3)$, which is the image of $\ol{S}\setminus (Z_2\cup Z_3)$ under a proper
birational morphism. Since $\ol{S}\setminus (Z_2\cup Z_3)$ contains $S$, $H_1(W;\Z)=0$. Note that
$Y'=\pi(W)$. Then $H_1(Y';\Z)=0$. Furthermore, $\Pic(Y')\otimes\Q=0$. The reasoning in the previous
paragraph shows that $\Pic(Y')=0$. This implies that $\Pic(\ol{Y}')$ is freely generated by the irreducible
components of $Q_2\cup Q_3$. It follows that the intersection matrix of $Q_i,\ i=2,3$ is unimodular.
Suppose that $Q_2$ is a chain. Since $Q_2$ is unimodular and does not contract to a smooth point by the assumption in \ref{Setup 2.1}, $Q_2$ contains an irreducible component with non-negative self-intersection number
and hence $\lkd(Y)=-\infty$ by the argument following \ref{Setup 2.1}. This is a contradiction to the
assumption. Note that since $E+C+Q_1$ contracts to a smooth point and $E$ is a fork, $Q_1$ is necessarily
a linear chain. Since the components of $C+Q_1+Q_2+Q_3$ give a basis of $\Pic(\ol{Y})\otimes\Q$ and since $C+Q_1$ is contractible, one of $Q_2, Q_3$ is not contractible by the Hodge index theorem. Hence we obtain the situation described in the statement.

Suppose, still with $r=3$, that $C$ meets two components $Q_1$ and $Q_2$ other than $E$. Consider
$U=\ol{Y}\setminus Q_3$ and $U'=\pi(U)$. Then $H_1(U';\Z)=0$ because $U$ contains $S$ and $U'$ is
the image of $U$ under a proper birational morphism. Let $U''$ denote $U'$ with $E$ contracted to
a singular point $q$. We find that $U''$ is $\Q$-acyclic by the same reasoning as above.
By Lemma \ref{Lemma 1.12}, it is affine and $\Z$-acyclic. By Lemma \ref{Lemma 1.11},
$\lkd(U''\setminus \{q\})=2$. Hence $U''$ satisfies the same assumptions as $S'$ does. However, $\ol{Y}'$ is a
smooth normal completion of the minimal resolution $U'$ of $U''$ with the boundary divisor $Q_3$ with $\#Q_3 < \#D$. This contradicts the assumption that $\#D$ is minimal for a counterexample to Theorem 1.

\end{proof}

Lemmas \ref{Lemma 2.2} and \ref{Lemma 2.3} give the following result.

%Corollary 2.4
\begin{cor}\label{Corollary 2.4}
Let things be as in 2.1. If $\lkd(Y) \ge 0$, then either the pair ($\ol{Y},Q$) is almost minimal,
or there exists a $(-1)$-curve $C$ in $(K_{\ol{Y}}+Q)^-$ as in Lemma \ref{Lemma 2.3}.
\end{cor}

Determination of the graph of $Q$ (or $D$) according to $\lkd(Y)$ will begin with the following result.

%Lemma 2.5
\begin{lem}\label{Lemma 2.5}
Let things be as in \ref{Setup 2.1}. If Y is  affine ruled, then $r=3$ and $Q_1,Q_2, Q_3$ are linear
chains.
\end{lem}

\begin{proof}
Since $Y$ is affine ruled, there exists a pencil $\Lambda=\{F_t\}_{t \in \BP^1}$ of rational curves in $\ol{Y}$
such that $F_t\cap Y$ contains $\A^1$ for general $t\in \BP^1$. We make the following observations, where we freely use Lemma \ref{Lemma 1.1}.

(a) If $\Lambda$ has a base point in $\ov{Y}$ it is unique and we call it $p$. We let $\wt{Y} \to \ol{Y}$ be the elimination of base points and let
$f : \wt{Y} \to \BP^1$ be the induced $\BP^1$-fibration. For a reduced divisor $T$ in $\ol{Y}$,
we denote by $\wt{T}$ its reduced inverse image in $\wt{Y}$. We may write $\wt{T}=T$ if $p \notin T$.
If $p \in Y$ we put $Y^\#=Y\setminus\{p\}$ and let $Q^\#=\wt{Y}\setminus Y^\# =\wt{Q} \cup  f^{-1}(p)$, that is, we incorporate $f^{-1}(p)$ as an extra connected component $\wt{Q}_{r+ 1}$ into the boundary. Otherwise $Y^\# =Y$ and $Q^\#=\wt{Q}$.

(b) If  there is a horizontal curve for $f$ in $Q^\#$, and this is the case if there is a base point, then it is unique. It is a $1$-section of $f$, and we call it $H$.

(c) Let $F$ be a fiber of $f$ and suppose $F \cap \wt{Q}_j \neq \emptyset$. Then there is no $(-1)$-curve in $F \cap \wt{Q}_j$ if $\wt{Q}_j =E$ or $\wt{Q}_{r+1}$ or if there is no horizontal curve in $\wt{Q}_j$. In fact, there are only $\leq (-2)$-curves in $\wt{E}$ and $\wt{Q}_{r+1}$ except for, possibly, $H$. In the last case, $\wt{Q}_j =Q_j \subset F$, and $Q_j$ is MNC.

(d) In the notation of Lemma \ref{Lemma 1.2.1} we have $b_2(Q^\#)-b_2(\wt{Y})=-1$.
The irreducible components of $Q^\#$ are numerically independent in
$\Pic(\wt{Y})$ and hence at most one fiber of $f$ is contained in $Q^\#$. Hence $\nu \leq 1$ in Lemma \ref{Lemma 1.2.1}.
\svskip

\noindent
(1)\ Suppose that $H$ exists and is a component of $\wt{E}$.  We have $h=1$ in Lemma \ref{Lemma 1.2.1}
and by (d) above, $\Sigma=\nu$. Clearly $\nu=0$, for otherwise there is a fiber of $f$ contained in $\wt{E}$, but the intersection matrix of $\wt{E}$ is negative definite. Hence $\Sigma=0$, i.e., any singular fiber
contains exactly one $Y$-component which, by (c) above, is its unique $(-1)$-component. In particular,
it is a multiple component in the fiber, hence it does not meet $H$. It follows that any singular fiber
contains an $\wt{E}$-component. Suppose that a fiber contains $\wt{Q}_i,\wt{Q}_j$ for $i,j>0, i \ne j$ and
let $C$ be its unique $Y$-component. Then $C$ meets $\wt Q_i$ and $\wt Q_j$ and a component of
$\wt{E}$, i.e., the $(-1)$-component $C$ meets $3$ other components of the fiber. This is impossible.
Hence any fiber contains at most one $\wt{Q}_j, j \geq 1$. Since $r \ge 3$, there exist at least $3$ singular fibers.
If $H$ is obtained by the elimination of a base point, then $\wt{E}-H$ would have at most two connected components. Hence
the pencil has no base points, and it follows that $\wt{Y}=\ol{Y}$. So $H \subset E$ and $H$ must be
the branching component of $E$, for otherwise we have at most two singular fibers. Therefore there are
exactly 3 singular fibers $F_1,F_2,F_3$ and we can label the twigs $E_1,E_2,E_3$ sprouting from $H$ so that $E_i \subset F_i$. It is clear now that $(F_i)_{\red}=E_i+C_i+Q_i$, where $C_i$ is the unique
$(-1)$-curve in $F_i$. Hence every $Q_i$ is a linear chain by 1.1(f) and the lemma is proved in this case. Note that we are in the case (i) in Proposition
\ref{Proposition 2.6} below.
\svskip

\noindent
(2) Suppose that $H$ exists and is a component of $\wt{Q}_k, k \geq 1$. We again have $\Sigma=\nu$.
Let $F_E$ be the fiber containing $E$. It is not a chain.

Suppose that $\nu=0$. Let $C$ be the unique $Y^\#$-component of $F_E$. Contracting $(-1)$-curves in
$F_E \cap \wt{Q}_k$ if necessary, we may assume that $C$ is the unique $(-1)$-curve in $F_E$.
Then $F_E-C$ has two connected components, one is the fork $E$ and the second one is a linear chain $T$.
In view of the structure of fibers with exactly one $(-1)$-component, $F_E$ contains two components of
multiplicity $1$, both contained in $E$. Hence $T$ consists of components of
multiplicity $> 1$, but the $1$-section $H$ meets $T$, a contradiction.

Suppose $\nu=1$.  Note that then $p \notin Y$ if there is a base point.
Now $f$ has a fiber $F_\infty$ contained in $\wt{Q}_r$, say, and $\Sigma=1$. If $F_E$ has a single
$Y^\#$-component, we reach a contradiction as in the case $\nu =0$.  So $F_E$ has two $Y^\#$-components, $C_1, C_2$ say. If there is a singular fiber $F_0$ with a unique $Y$-component, then this component has multiplicity $m > 1$. We can take a cyclic
covering of order $m$ of the base curve $\BP^1$ ramified over the points $f(F_0)$ and $\infty$ and the normalization of the fiber product of $f : \wt{Y} \to \BP^1$ with this cyclic covering. We obtain
an unramified abelian covering of $Y\cup E$, which is a contradiction because of $H_1(Y\cup E;\Z)=0$.
There remains only the case where $F_E$ is the only singular fiber other than $F_\infty$.
Contracting $(-1)$-curves if necessary, we may assume that $F_E\cap \wt{Q}_r$ does not contain a $(-1)$-curve.

We have
\begin{equation}\label{eqn 2.5}
F_E -(C_1 + C_2)=E +Q_1+\dots + Q_{r-1}+ (F_E\cap \wt{Q}_r)\ .
\end{equation}
Here $F_E\cap \wt{Q}_r$ could be reduced to a point.

Suppose that $F_E\cap \wt{Q}_r$ contains at least one component. Then $F_E\cap \wt{Q}_r$ is a connected component of $\wt{Q}_r-H$ and contains a component $Q_{r,0}$ meeting $H$.

Suppose that $F_E$ contains only one $(-1)$-component, say $C_1$. Let $F_E - C_1=R + T$ be the decomposition into connected components, where $R$ contains a curve with multiplicity $1$.
Then $T$ is a chain by Lemma \ref{Lemma 1.1}(f). Hence $E \subset R$. It follows that $C_2 \subset R$,
for otherwise $F_E -(C_1 + C_2)$ has at most $3$ connected components while there are at least
four by (\ref{eqn 2.5}). By the same reason, $C_2$ is a branching component of $R$. It follows that the multiplicity $1$ component $Q_{r,0}$ is in $R$. The fiber
$F_E$ can be contracted to $Q_{r,0}$. Reversing this process we regain the fiber $F_E$
by a unique succession of blowings-up. We order the resulting exceptional
components by their order of appearance. Let $T$ be the first branching component of $F_E$ with respect to this ordering of components. There are three branches, $T_1, T_2, T_3$ say, at $T$ where
$T_1,T_2$ are chains made up of the components appearing before $T$, including $Q_{r,0}$, and $C_1 \subset T_3$. It follows that $T=C_2$ or that $T_1+T+T_2 \subset F_E \cap \wt{Q}_r$. In either case $C_1, C_2$ and $E$ are in $T+T_3$ and the multiplicities of $C_1, C_2$ and all components of $E$ have
the multiplicity of $T$ as a common divisor $m > 1$.
 Now we take a covering of the base
curve of order $m$ ramified over the points $f(F_E)$ and $\infty$ and the normalization of the fiber
product of $f$ with this cyclic covering. We obtain an unramified abelian covering of $Y\cup E$. This
gives a contradiction as above.

Thus both $C_1, C_2$ are $(-1)$-components. It follows that each of $C_1, C_2$ meets at most two connected
components of $Q$ and that each meets a $Q_j$ in at most one point transversally. Furthermore,
$C_1, C_2$ are disjoint from each other. It is impossible to obtain a connected fiber by connecting
four trees by two curves satisfying these conditions.

Thus $F_E\cap \wt{Q}_r$ is a point on $H$ and one of $C_1, C_2$ meets $H$. Say $C_2$ does. If $C_2$ is a $(-1)$-curve, it is a tip of $F_E$ of multiplicity 1. It cannot be the only $(-1)$ component of $F_E$ by what we said above about the case of a unique $(-1)$-component. Hence $C_1$ is a $(-1)$-component. If it is the only one, then again $C_2$ is a tip of $F_E$ by what we said above.  As a
$(-1)$-component of $F_E$, $C_1$ meets at most two other components of $F_E$. It follows that the number
of connected components of $F_E - (C_1 + C_2)$ is at most two, which is a contradiction
since (\ref{eqn 2.5}) implies that there are at least $3$ of them.
\svskip

\noindent
(3)\ Suppose that $H$ does not exist. Then there is no base point and the general fiber is contained in $Y$. We have $h=0$ and find $0 \leq \Sigma=\nu -1$. Thus $\nu=1$
and $\Sigma=0$. The fibration has one fiber $F_\infty$ contained in some $Q_j$. In fact, $F_\infty$
coincides with $Q_j$ supportwise since $Q_j$ is an MNC divisor. So actually $Q_j$ is a $(0)$-curve and $Q_j \neq E$. We can assume $j=r$. If there exists a singular fiber $F_1 \neq F_E$ it has a unique $Y$-component of multiplicity $m > 1$. We reach a contradiction  by the covering technique in (2) above. Thus $F_E$ is the unique singular fiber of $f$ and hence contains $E, Q_1,\dots,
Q_{r-1}$. But this is a contradiction since these trees, at least three in number, cannot be connected in $F_E$ by a single
$(-1)$-component.
\end{proof}

The following result will limit the possibilities for the boundary divisor $D$.

%Proposition 2.6
\begin{prop}\label{Proposition 2.6}
Let things be as in \ref{Setup 2.1}. Then we have one of the following cases.
\begin{enumerate}
\item[{\rm (i)}]
All components $Q_1,\cdots,Q_r$ contract to quotient singularities.
\item[{\rm (ii)}]
$\lkd(Y)=-\infty$, $r=3$, there exists a $(-1)$-curve $C$ in $\ol{Y}$ such that $Q_1+C+Q_2$ contracts
to a smooth point, $C\cdot (Q_3+E)=0$ and $Q_3$ is the complementary fork of $E$.
\item[{\rm (iii)}]
$\lkd(Y) \ge 0$, $r=3$, $Q_1, Q_2$ contract to quotient singularities, $|G_1|=|G_2|=2$, $Q_3$ does not
contract to a quotient singularity. The pair $(\ol{Y}, Q)$ is almost minimal.
\item[{\rm (iv)}]
$r=3$, $Q_1$ is a linear chain, none of $Q_2, Q_3$ is a linear chain, one is non-contractible and
the intersection matrices of both are unimodular. Further, there exists a $(-1)$-curve $C$
such that $C\cdot E=C\cdot Q_1=1$, $E+C+Q_1$ contracts to a smooth point and $C\cdot Q_2=C\cdot Q_3=0$.
\end{enumerate}
\end{prop}

\begin{proof} (1) Suppose that $\lkd(Y) \ge 0$. If the pair $(\ol{Y}, Q)$ is not almost minimal, then
by Corollary \ref{Corollary 2.4}, we have the case (iv). Assume that the pair $(\ol{Y},Q)$ is almost
minimal. Let $Q_0,Q_1,\ldots,Q_k$ be all the connected components of $Q$ which contract to quotient
singularities. Now, since $\chi(Y)=\chi(\ol{Y}\setminus Q)=2-r$, the BMY inequality gives
\[
\frac{1}{|G_0|}+\sum_{i=1}^{k}\frac{1}{|G_i|} \ge r-2.
\]
Hence $\frac{1}{8}+\frac{k}{2} \ge r-2$ because $|G_0| \ge 8$. Since $r \ge k$, this implies that
$r \le 4$. If $r=4$, then $k=4$ and $|G_i|=2$ for $1 \le i \le 4$. So, we are in the case (i).
Suppose that $r=3$. Then $k=1$ is impossible. If $k=2$, then $|G_1|=|G_2|=2$ and we are in
the case (iii). If $k=3$ then we are in the case (i).

(2) Suppose that $\lkd(Y)=-\infty$. In view of Lemma \ref{Lemma 2.5}, we may assume that $Y$ is not
affine ruled.

(2.1) Suppose that $(\ol{Y},Q)$ is not almost minimal. Let $(\ol{Y}',Q')$ be an almost minimal
model. We keep the notation of Setup \ref{Setup 1.8} with $T=Q$ and $T_i=Q^{(i)}$, the image
of $Q$ on $\ol{Y}_i$. We are now in the situation of Lemma \ref{Lemma 1.13} with $(X,P)=(\ol{Y'},Q')$.\\

Suppose that
$\chi(Y') \leq \chi(Y)$. In case (b), $\chi(Y')=0$. Since $\chi(Y)=2-r<0$, case (a) of  Lemma \ref{Lemma 1.13} occurs. Hence the intersection matrix of $Q'$
is negative definite and the irreducible components of $Q'$ are independent in $\Pic(\ol{Y}')$. On the other hand, since at least one blowing down
is involved in passing from $(\ol{Y},Q)$ to $(\ol{Y}',Q')$, we have $\rank\Pic(\ol{Y}') \leq \#Q'$. This
a contradiction to the Hodge Index Theorem. So $\chi(Y')>\chi(Y)$.  As in the first part of the proof of
Lemma \ref{Lemma 2.2}, we find a $(-1)$-curve $C_0$ such that $C_0$ meets two components of $Q$,
say $Q_{i_0}$ and $Q_{i_1}$, and that $Q_{i_0}+C_0+Q_{i_1}$ contracts to a smooth point.
Again the irreducible components of $Q+C_0$ are independent in $\Pic(\ol{Y})$ and hence they form a basis
of $\Pic(\ol{Y})\otimes \Q$. It follows that the irreducible components of $Q^{(1)}$ form a basis of
$\Pic(\ol{Y}'_1)\otimes \Q$ and hence for $i \ge 1$ not all connected components of $Q^{(i)}$ contract to quotient
singularities. Hence we are in case (b) of Lemma \ref{Lemma 1.13}. For $i \geq 1$, the curve $C_i$ will not have the property that it meets
two connected components of $Q^{(i)}$ and together with these contracts to
a smooth point. For otherwise, the irreducible components of $Q^{(i)}+C_i$ are independent in
$\Pic(\ol{Y}_i)$, but $\rank\Pic(\ol{Y}_i) \le \#Q^{(i)}$. Hence $0=\chi(Y') \leq \chi(Y_1)=\chi(Y)+1$ by Lemma 1.12.
It follows that $\chi(Y)=-1$, i.e. $r=3$, $Q=Q_0+Q_1+Q_2+Q_3$. $Q^{(1)}$ has two connected components
and $\chi(Y_1)=0$. Hence no further contraction is possible, i.e. $Y_1=Y'$.

Suppose that $E \neq Q_{i_0}$ and $E \neq Q_{i_1}$. We may assume that $i_0=1, i_1=2$. In the notation of
Lemma \ref{Lemma 1.13} we have $P=Q^{(1)}=E+Q_3$ with $E$ contractible. Hence $Q_3$ is the complementary
fork of $E$ and we have the case (ii). (If $Y_1'$ denotes $\ol{Y}_1\setminus Q_3$ with $E$ contracted to a quotient
singularity, then $Y_1'\simeq \C^2/G$.)

Suppose that $E=Q_{i_0}$. We may assume that $Q_{i_1}=Q_1$. In the notation of
Lemma \ref{Lemma 1.13} we have $P=Q^{(1)}=Q_2+Q_3$ where we can assume $Q_2$ contractible and $Q_3$ is the complementary
fork of $Q_3$. In fact, we may assume that $Q_3$ is not contractible. Then $Q_2$
is contractible since $Y' = Y_1$ is a Platonic fiber space. Let $Y_1''$ denote $\ol{Y}_1\setminus Q_3$
with $Q_2$ contracted to a quotient singularity $p$ with local fundamental group $G$. Then $Y_1''$ is isomorphic to $\C^2/G$. Since
$H_1(Y\cup E;\Z)=0$ implies $H_1(Y \cup E\cup Q_1;\Z)=0$ and since $Y_1''\setminus \{p\}$ is the image of
$Y\cup E\cup Q_1$ under a proper birational morphism, we have $H_1(Y_1''\setminus \{p\};\Z)=0$. It follows that $G$ is
the binary icosahedral group, i.e., $Q_2$ is the $(-2)$-fork of type (2,3,5) (cf. \cite[Satz 2.11]{B})
and $Q_3$ is the complementary fork of $Q_2$, and both are unimodular. Since $C_1$ meets a fork $E$ and a connected component
$Q_1$ of $(K+Q)^-$, the theory of peeling implies that $Q_1$ is a linear chain. So, we have the case (iv).

(2.2) Suppose that $(\ol{Y},Q)$ is almost minimal. Since $\chi(Y)<0$, we are in case (i) by Lemma \ref{Lemma 1.13}.

\end{proof}

Let $S_0=S'\setminus \{q\}=S\setminus E$. Then $H_1(S_0;\Q)=H_2(S_0;\Q)=0$ (cf. \cite[Proposition 2.6]{KR}). Hence
$\wt{b}_1(S_0)=\wh{b}_2(S_0)=0$. We shall consider a branching component of $D$ with precisely three branches.
The following result, restricting the self-intersection number of such a component, will be used in proving Lemma \ref{Lemma 2.9} below.

%Lemma 2.7
\begin{lem}\label{Lemma 2.7}
The following assertions hold.
\begin{enumerate}
\item[{\rm (i)}]
There is no branching component $B$ in $D$ such that $B$ has precisely three branches in $D$ and $B^2 \geq 0$.
\item[{\rm (ii)}]
There is no branching component $B$ in D such that $B$ has precisely three branches in $D$, $B^2=-1$ and
two of the three branches consist of single $(-2)$-curves.
\end{enumerate}
\end{lem}

\begin{proof}
(i)\ Suppose on the contrary that $B$ is such a branching component with branches $T_1, T_2, T_3$. Let
$H_i$ be the component of $T_i$ meeting $B$ and write $T_i=H_i+T_i'$. By blowing up if necessary the point of intersection of
$B$ with (the inverse image of) $T_1$ we may assume that $B^2=0$. Here we temporarily allow the case $H_1^2=-1$ and by abuse of notation still write $D, \ov{S}, \cdots$ for the corresponding quantities after the blowings-up.
We apply Lemma \ref{Lemma 1.2.1} to $D+E$ and the
$\BP^1$-fibration $f : \ol{S} \to \BP^1$ induced by $B$. We have $h=3$ (the three components $H_i$ are $1$-sections of
$f$ and the only horizontal curves in $D$) and $\nu =1$  since clearly $B$ is the unique
fiber contained in $D$. By the above remark on $\wt{b}_1, \wh{b}_2$, we have $\Sigma =2$. Let $F_E$ be the fiber containing $E$. We make the following
\svskip

\noindent
{\bf Claim.}\ {\em $F_E$ is the only singular fiber of the fibration $f$.}
\svskip

If we admit the claim, then $S\setminus F_E \simeq \C^*\times \C^{**}$, which implies $\lkd(S_0)\leq 1$,
in contradiction to our standing assumption that $\lkd(S'\setminus \{q\})=2$. So it remains to prove the claim.
\svskip

\noindent
{\it Proof of the Claim.}\
Let $L$ be a component of a singular fiber $F$ and denote by $\mu$ its multiplicity in $F$.
We extract the following assertions from Lemma \ref{Lemma 1.1} and Lemma \ref{Lemma 1.5} for easy reference.
\begin{enumerate}
\item[(a)]
If $L^2=-1$, then $L$ meets at most two other components of $F$ and precisely one if $\mu =1$.
\item[(b)]
If $L$ is the unique $(-1)$-component in $F$, then $\mu \geq 2$. Moreover, $F$ has precisely two components
of multiplicity $1$.
\item[(c)]
If $L$ is an $S_0$-component of $F$, then $L\cdot E \le 1$ and $L\cdot D=\Sigma_{i=1}^3 L\cdot T_i\ge 2$. Moreover, $L\cdot T_i \le 1,\ i=1,2,3$.
\end{enumerate}
Suppose $C$ is a $(-1)$-component of $F\cap D$. Then $C$ is not one of the exceptional curves in the blow-up process. So there are two possibilities. First, no blow-up takes place on $C$. Then $C$ is branching in $D$, but non-branching in $F$ by (a), and hence meets one of the sections $H_i$ and at least two other components of $F\cap D$. This is ruled out by (a). Second, a single blow-up occurs on $C$. Then $C$ is a $(0)$-curve in the original $D$ and meets an end component of the exceptional locus which is a linear
chain  or $C$ meets a 1-section of $f$. By (a) this implies
that $C$ is non-branching in the original $D$ and defines a $\C$- or $\C^*$-ruling of $S_0$. This contradicts the assumption $\lkd(S'\setminus \{q\})=2$. Hence $F\cap D$ has no $(-1)$-components.

Suppose that $L$ is the only $S_0$-component in $F$. Then $L$ is the only $(-1)$-component and we have
$\mu \ge 2$ by (b). Hence $L$ does not meet any of the $1$-sections $H_i$. On the other  hand, $F$ meets each $H_i$ and it follows that $L\cdot T_i'=1, i=1,2,3$ and we have a contradiction to (a).

In particular $F_E$ has at least two $S_0$-components. Suppose there are precisely two, $L_1,L_2$ say. If $L_1$ is the only (-1)-component in
$F_E$, then $\mu \ge 2$ and by 1.8 we find that $L_1\cdot T_i' =1$ for two of the $T_i'$, say for $T_1', T_2'$, and that $L_1\cdot E=L_1\cdot T_3=0$.
Then $T_1',T_2'$ are part of $F_E$ and the components in
$T_1',T_2'$ meeting $H_1,H_2$ have multiplicity $1$ in $F_E$. By (b), $T_3'\cap F_E= \emptyset$. In fact, if
$T_3' \neq \emptyset$ the component meeting $H_3$ has multiplicity $1$ in its fiber. Since $F_E\cdot H_3=1$,
we have $L_2\cdot H_3=1$, so $L_2$ is a third component of multiplicity $1$ in contradiction to (b).
So we have $L_1^2=-1=L_2^2$ and we may assume $L_1\cdot E=1$. If $L_1 \cap T_i' \neq \emptyset$ for some $i$ then
$L_1$ has multiplicity $\geq 2$ by (a) and $L_1$ meets at least two of the $T_i'$ by (c), in contradiction to (a). So,
$L_1$ meets at least two of the $H_i$, say $L_1\cdot H_2=1=L_1\cdot H_3$. Then $T_2'\cap F_E=T_3'\cap F_E= \emptyset$ and $L_2$ meets only $T_1$ and $L_2\cdot D=L_2\cdot T_1=1$, a contradiction to
(c).

Hence $F_E$ has at least $3$ $S_0$-components.
Since $\Sigma=2$ there are exactly $3$ and any other singular fiber $F$ has exactly one. By what we said above this proves the claim.

\svskip

(ii)\ Let $Z_i\ (i=1,2)$ be the branches at $B$ consisting of a single $(-2)$-curve. Then the linear pencil
$|Z_1+2B+Z_2|$ induces a $\BP^1$-fibration on $\ol{S}$ for which the component of the third branch
$Z_3$ meeting $B$ is a $2$-section. Hence $S_0$ has a $\C^*$-fibration and $\lkd(S_0) \le 1$. This
contradicts the assumption $\lkd(S_0)=2$.
\end{proof}

In what follows we choose a various branching component $B$ of $D$ so that  Setup \ref{Setup 2.1} holds with $D_0=B$ and apply Proposition
\ref{Proposition 2.6} to $D-B$. When $D-B$ falls in one of the cases treated there, we often say
that $B$ is of type 2.6~(i) or, simply, of type (i), etc. We will write $r=r_B$ in Setup \ref{Setup 2.1} when we need to indicate the dependence on $B$.

First we mention an easy fact.

%Lemma 2.7.1
\begin{lem}\label{Lemma 2.7.1}
In the situation of Proposition \ref{Proposition 2.6}~(i) we have $d(D) < 0$.
\end{lem}
\begin{proof} Let $I(-)$ denote intersection matrix. Then $-I(D-D_0)$ is positive definite since the connected components of $D-D_0$ are
contractible. If $d(D)>0$,
$-I(D)$ is positive definite by Sylvester's criterion. Hence $I(D)$ is negative definite, which is
impossible because $D$ supports an ample divisor. Hence $d(D) \leq 0$. Since $|d(D)|=a > 0$ by
Lemma \ref{Lemma 1.7}, this entails that $d(D)<0$.
\end{proof}

%Lemma 2.8
\begin{lem}\label{Lemma 2.8}
Let things be as in \ref{Setup 2.1}. Suppose that all $Z_i$ are MNC. Then the case 2.6~(ii) does
not occur.
\end{lem}

\begin{proof}
Suppose that the case (ii) occurs. Let $B$ be the branching component of the fork $Q_3$ and $R_1,R_2,R_3$ the three branches. They are contractible chains. Let $D_1$ be the component of $Q_3$
such that $D_0\cdot D_1=1$. We assume that $D_1$ is a component of $R_1$ if $D_1 \neq B$.
Let $Q_3^\#=Q_3-D_1$. We have

\begin{equation}\label{eqn 2.6}
d(D)=d(D_0+Q_1+Q_2)d(Q_3)-d(Q_1)d(Q_2)d(Q_3^\#).
\end{equation}

By Lemma \ref{Lemma 2.7}, $B^2<0$ if $D_1\neq B$. Since $B$ is the branching component of the fork complementary
to $E$, $B^2\geq -1$. Hence $B^2=-1$ if $D_1\neq B$. We first establish
\svskip

\noindent
{\bf Claim.}\ \ $d(D_0+Q_1+Q_2)>0$.
\svskip

Suppose $d(D_0+Q_1+Q_2)\leq 0$ and apply Proposition \ref{Proposition 2.6} to the branching component $B$
(not $D_0$). Now $D_0+Q_1+Q_2$ is not contractible and hence the connected component of $D-B$ containing $D_0$ is not contractible since it contains $D_0+Q_1+Q_2$. So clearly the assumptions of \ref{Setup 2.1} are satisfied and $B$ is not of type 2.6~(i). If $B$ is of type 2.6~(iii), then $R_2$ and $R_3$ are $(-2)$-curves, $D_1\neq B$ and $|R_2+2B+R_3|$ induces
a $C^*$-fibration on $S'\setminus \{q\}$, a contradiction as before. Hence $B$ is not of type 2.6~(iii).
Clearly, $B$ is not of type 2.6~(iv). Hence $B$ is of type 2.6~(ii). In particular $r_B=3$. If $D_0$ meets $B$
then $D-B$ has $4$ connected components and $r_B=4$. Hence $D_0\cdot B=0$. Since $B$
is of type 2.6~(ii), there exists a $(-1)$-curve $C'$ such that $R_2+C'+R_3$ contracts to a smooth point
and $C'\cdot (Q_1+Q_2+D_0+R_1+E)=0$. Furthermore, $D_0+Q_1+Q_2+R_1$ is the complementary fork of $E$ with
$D_0$ as branching component.
Hence $D_0^2=-1$ by the argument which gave $B^2=-1$. Also, $D_0$ meets $Q_1, Q_2$ in tips.
Since  $d(Q_1+D_0+Q_2)\leq 0$,  $Q_1+D_0+Q_2$ is a non-contractible chain and is obtained from a linear chain by blowing up over a $(0)$-curve contained in it. Hence there exists a divisor $F$ with
$\Supp(F) \subset D_0+Q_1+Q_2$ which induces a $\BP^1$-fibration of $\ol{S}$. Then $B+R_2+C'+R_3$ is
contained in a fiber of this $\BP^1$-fibration. This is a contradiction since $B+R_2+C'+R_3$ is a loop.
\svskip

We continue the proof. We have $\pm d(D)=d(E) > 0$ by Lemma \ref{Lemma 1.7}, while $d(Q_3)=-d(E) <0$
because $Q_3$ is the complementary fork of $E$ on the surface obtained from $\ol{S}$ by contracting
$Q_1+C+Q_2$ to a smooth point. So (\ref{eqn 2.6}) can be written as
\[
-d(Q_1)d(Q_2)d(Q_3^\#)=d(E)(d(D_0+Q_1+Q_2)\pm 1).
\]
Since $Q_1, Q_2$ are contractible, $d(Q_1)>0$ and $d(Q_2)>0$. Hence the above claim implies that $d(Q_3^\#) \leq 0$. It then follows that $D_0\cdot B=0$,
for otherwise $Q_3^\#$ consists of three contractible chains and $d(Q_3^\#)>0$. Suppose that
$D_1\cdot B=1$. Then
$$d(Q_3^\#)=d(B+R_2+R_3)d(R_1-D_1)\leq 0.$$
 Since $R_1-D_1$ is either empty or
a contractible chain, we have $d(R_1-D_1) \geq  0$. Thus $d(B+R_2+R_3) \le 0$ and $B+R_2+R_3$ is
a non-contractible linear chain. As above there exists a divisor $F$ with $\Supp(F) \subset B+R_2+R_3$
such that $|F|$ gives rise to a $\BP^1$-fibration on $\ol{S}$ and $D_0+Q_1+Q_2+C$, which is a loop, is
contained in a fiber. This is a contradiction. Hence $D_1\cdot B=0$. Then in particular
$\#R_1>1$. As for any contractible or complementary fork, one of the $R_i$ is a single $(-2)$-curve. We may assume it is $R_2$. If the tip  meeting $B$ of $R_1$ or of $R_3$ is a $(-2)$-curve, $N$ say, then $N+2B+R_2$ induces a $\BP^1$-fibration with the loop $Q_1+D_0+Q_2+C$
contained in a fiber. Hence these tips are $(\leq -3)$-curves. It follows that $Q_3$ is a $(5,2,3)$-fork. Moreover  $R_1$ has two components
$D_1$ and $D_2$ with $D_2^2=-3$, $D_1^2=-2$ and $D_2$ meeting $B$, and $R_3$ consisting of a single $(-3)$-curve.

$D-B$ has connected components $R_2, R_3$ and $\wh{R}_1=D_0+Q_1+Q_2+{}^tR_1$, where we write ${}^tR_1$ for $R_1$ attached in reversed order to the branching component $D_0$ of $\wh{R}_1$ by $D_1$.
Since $D_0$ is the only possible $(-1)$-curve in $D-B$ the assumptions of \ref{Setup 2.1} are satisfied with respect to  $B$ and we can apply Proposition
\ref{Proposition 2.6}. We infer readily that either the case (i) or the case (ii) occurs. In case (ii),  $\wh{R}_1$ is the complementary fork of $E$
and identical to the fork $Q_3$. Since $R_1$ is not symmetric,
this is not the case. So $B$ is of type 2.6~(i),
i.e., $\wh{R}_1$ is a contractible fork. It follows from Lemma \ref{Lemma 2.7.1} that $d(D)<0$ and hence $d(D)=-d(E)$. $\wh{R}_1$ is either of type $(2,2,5)$ or $(2,3,5)$.
We may assume that $Q_1$ is a $(-2)$-curve. Then $Q_2$ is a $(-2)$-curve in the $(2,2,5)$-case
and either a $(-3)$-curve or a chain of two $(-2)$-curves in the $(2,3,5)$-case. We have
$-d(D)=a=d(E)=-d(Q_3)=7$. Let $D_0^2=-t$. Suppose that $Q_2$ is a $(-2)$-curve. We compute
$-7=d(D)=(4t-4)d(Q_3)-4d(Q_3^\#)=-(4t-4)7+12$. Suppose that $Q_2$ is a $(-3)$-curve. Then $-7=-(6t-5)7+18$.
If $Q_2$ consists of two $(-2)$-curves then $-7= -(6t-7)7+18$. In all three cases, $t$ is not an integer.
Hence the case 2.6~(i) is impossible.
\end{proof}

%Lemma 2.9
\begin{lem}\label{Lemma 2.9}
Suppose that $D-D_0$ is of type (iii) in Proposition \ref{Proposition 2.6} and that $Z_1,Z_2$ are MNC.
Let $D_1$ be the component of $Z_3$ which meets $D_0$. Then the following assertions hold.
\begin{enumerate}
\item[{\rm (1)}]
$D_1$ is contained in a twig of $Z_3$. If $Z_3$ is MNC, then $D_1$ is a tip of $Z_3$ and
$D_0^2=D_1^2=-2$. Moreover, $D_1$ is not a maximal twig of $Z_3$.
\item[{\rm (2)}]
Suppose that $Z_3$ contains two $(-1)$-components $B_1, B_2$ which meet each other. Then $Z_3$ is MNC, $D_1,~D_0$ are
$(-2)$-curves and $D_1\cdot (B_1+B_2)=0$.
\end{enumerate}
\end{lem}
\begin{proof}
Let $P=(K_{\ol{S}}+D+E)^+$ and $P_Y=(K_{\ol{Y}}+Q)^+$ be respectively the nef parts in the Zariski-Fujita
decomposition of $K_{\ol{S}}+D+E$ and $K_{\ol{Y}}+Q$. Here we note that $Q$ contains $Q_0=E$. We have
\[
0\leq P^2 \leq \frac{3}{|\Gamma|} \leq \frac{3}{8}, \quad
0\leq P_Y^2\leq 3(\frac{1}{2}+\frac{1}{2}+\frac{1}{|\Gamma|}-1)\leq \frac{3}{8}
\]
by the BMY inequality (not Lemma 1.6,
but the original estimation involving $P^2$ or $P_Y^2$
 (cf. [Mi, Chapter
2, Theorem 6.6.2])) where we recall that $\Gamma$ is the local fundamental group of $S'$ at the point $q$. Let $T_1,\dots, T_\ell$ be the maximal
twigs of $Z_3$.

(1)\   Suppose that $D_1$ is not contained in any $T_i$. Then the $T_i$ are untouched under the process $Z_3 \to Q_3$ of making $Z_3$
an MNC divisor since $D_1$ is the only possible nonbranching $(-1)$-component in $Z_3$ and is separated from any $T_i$ by a branching component of $Z_3$. Note that $(K_{\ol{Y}}+Q)^2=(K_{\ol{S}}+D-D_0+E)^2$ since we successively contract only $(-1)$-components meeting exactly
two other components of $Z_3$ or its image. Since $(\ol{S},D+E)$ and $(\ol{Y},Q)$ are almost minimal,
we have $P=K_{\ol{Y}}+(D+E)^\#$ and $P_Y=K_{\ol{Y}}+Q^\#$. Computing barks we find

\begin{eqnarray}
\label{eqn 2.7}(K_{\ol{S}}+D+E)^2&=&P^2-1-\sum_{i=1}^{\ell}e(T_i)+(\Bk E)^2 \\
\label{eqn 2.8}(K_{\ol{Y}}+Q)^2&=&P_Y^2-4-\sum_{i=1}^{\ell}e(T_i)+(\Bk E)^2 .
\end{eqnarray}
See \cite{Mi} and \cite{F} where $e(T_i)$ denotes the inductance  of $T_i$.\\
Subtract the first line from the second one. Since
\begin{eqnarray*}
\lefteqn{(K_{\ol{Y}}+Q)^2-(K_{\ol{S}}+D+E)^2} \\
& &= (K_{\ol{S}}+D-D_0+E)^2-(K_{\ol{S}}+D+E)^2 = -2+D_0^2\ ,
\end{eqnarray*}
we obtain $-2+D_0^2=P_Y^2-P^2-3$, whence
\[
D_0^2=P_Y^2-P^2-1 \geq -1-\frac{3}{8}.
\]
Hence $D_0^2\geq -1$. This is ruled out by Lemma \ref{Lemma 2.7}~(ii).

Suppose that $Z_3$ is MNC. Note that then \ref{Setup 2.1}~(c) holds. We may assume  that $D_1$ is contained in the twig $T_1$ of $Z_3.$
Suppose that $D_1$ is not a tip of $Z_3$. Let $T_1'$ be the maximal twig of $D$ which is contained in $T_1$.
Then the maximal twigs of $D$ are $Z_1,Z_2,T_1',T_2,\dots, T_\ell$. We have

\begin{equation}\label{eqn 2.9}
(K_{\ol{S}}+D+E)^2=P^2-1-e(T_1')-\sum_{i=2}^{\ell}e(T_i)+(\Bk E)^2 \ .
\end{equation}

Subtracting (\ref{eqn 2.9}) from (\ref{eqn 2.8}), we obtain
\[
-2+D_0^2=P_Y^2-P^2-3+e(T_1')-e(T_1).
\]
and hence
\[
D_0^2=P_Y^2-P^2-1+e(T_1')-e(T_1).
\]
Let the common tip of $T_1$ and $T_1'$
be a $(-k)$-curve. Then $e(T_1') \geq \frac{1}{k}$ and (since $T_1$ has at least two components)
$e(T_1)<\frac{1}{k-1}$ . We obtain
\[
D_0^2>-P^2-1-\frac{1}{k(k-1)} \geq -\frac{3}{8}-1-\frac{1}{2}>-2.
\]
We obtain $D_0^2\geq -1$, and this is a contradiction by Lemma \ref{Lemma 2.7}~(ii).

We have shown therefore that $D_1$ is a tip of $Z_3$. Hence the maximal twigs of $D$ are $T_2,\cdots, T_\ell$ and we obtain

\begin{eqnarray}
\label{eqn 2.10} (K_{\ol{S}}+D+E)^2&=&P^2-1-\sum_{i=2}^{\ell}e(T_i)+(\Bk E)^2 \\
\label{eqn 2.11} (K_{\ol{Y}}+Q)^2&=&P_Y^2-4-\sum_{i=1}^{\ell}e(T_i)+(\Bk E)^2\ .
\end{eqnarray}

Subtracting (\ref{eqn 2.10}) from (\ref{eqn 2.11}), we obtain
\[
-2+D_0^2=P_Y^2-P^2-3-e(T_1).
\]
and then
\[
D_0^2\geq -\frac{3}{8}-1-e(T_1).
\]
Since $e(T_1)<1$ we obtain $D_0^2 \geq -2$. Hence $D_0^2=-2$ in view of Lemma \ref{Lemma 2.7}.
If $D_1^2 \leq -3$ or if $D_1$ is a maximal twig of $Z_3$, then $e(T_1) \leq \frac{1}{2}$ and we obtain
$D_0^2>-2$ in contradiction to Lemma \ref{Lemma 2.7}.
\svskip

(2)\  By (1), we can assume that $D_1$ is contained in the maximal twig $T_1$ of $Z_3$. Since $D$ is MNC, no component of
$T_1$ other than $D_1$ is a $(-1)$-curve and $D_1$ is a branching component of $D$ if $D_1^2=-1$.

Suppose first that $D_1^2=-1$. Suppose further
that $D_1 \neq B_1,B_2$. We consider the $\BP^1$-fibration $f$
defined by $|B_1+B_2|$. If $D_1 \cdot(B_1+B_2)=0$, let $F$ be the fiber of $f$ containing $D_1$. Either all components of $D$ meeting $D_1$ are in $F$, or precisely one of them meets $B _1+B _2$ and is a $1$-section of $f$. Both possibilities are ruled out by Lemma \ref{Lemma 1.1}~(f).  Hence
$D_1\cdot (B_1+B_2) > 0$. We may assume that $D_1\cdot B_1=1$. Then $B_1$ is a branching component of
$Z_3$ to which the twig $T_1$ is attached. This is also the case if $D_1$ coincides with $B_2$.

We shall show that this situation leads to a contradiction. The process
$Z_3 \to Q_3$ of making $Z_3$ an MNC divisor involves only contractions of $T_1$-components since the
image of $B_1$ is a curve with non-negative self-intersection number.  Let $T_1'$ be the image of $T_1$. The maximal twigs of $Q_3$ different from $T_1'$ are precisely
$T_2,\dots, T_\ell$. Let $T_0$ be the maximal twig of $D$ contained in $T_1$. We have
\begin{eqnarray}
\label{eqn 2.12}\qquad (K_{\ol{S}}+D+E)^2 &=& P^2-1-e(T_0)-\sum_{i=2}^{\ell}e(T_i)+(\Bk E)^2 \\
\label{eqn 2.13}\qquad (K_{\ol{Y}}+Q)^2 &=& P_Y^2-4-e(T_1')-\sum_{i=2}^{\ell}e(T_i)+(\Bk E)^2\ .
\end{eqnarray}
Subtracting (\ref{eqn 2.12}) from (\ref{eqn 2.13}), we obtain
\[
-2+D_0^2=P_Y^2-P^2-3-e(T_1')+e(T_0).
\]
 If $T_1$ is contracted, $T_1'$ is a point and $e(T_1')=0$.
Then $D_0^2> -1-P^2$, which implies $D_0^2 \geq -1$ in contradiction to Lemma
\ref{Lemma 2.7}. So let $D_3$ be the tip of $T_1$ and let $D_3^2=-k$. $D_3$ is not contracted. Suppose that $D_3$ is untouched in the minimalization. Then $e(T_0) \geq \frac{1}{k}$
and $e(T_1')<\frac{1}{k-1}$. We obtain $D_0^2> -1-\frac{3}{8}-\frac{1}{k(k-1)} > -2$. This is again
a contradiction by Lemma \ref{Lemma 2.7}. Suppose that $D_3$ is touched. Then $T_1'$ is a single curve
with self-intersection number $\le -2$, whence $e(T_1')\leq \frac{1}2$. We obtain again $D_0^2\geq -1$, a contradiction. So the case $D_1^2=-1$ is impossible.

In particular $Z_3$ is an MNC divisor.  By (1), $D_1$ is a tip of $Z_3$ and $ D_1\neq T_1$.
Clearly $B_i$ is not contained in $T_1$. Hence $D_1\cdot (B_1+B_2)=0$. Further, $D_0^2=D_1^2=-2$
by (1).
\end{proof}

Lemma \ref{Lemma 2.9} implies the following result.

%Corollary 2.10
\begin{cor}\label{Corollary 2.10}
Suppose that $D-D_0$ is of type (iii) in Proposition \ref{Proposition 2.6} and that $Z_1, Z_2$
are $(-2)$-curves. Then $Z_3$ does not contain two $(-1)$-curves which meet each other.
\end{cor}
\begin{proof}
Suppose that two such $(-1)$-curves $B_1,B_2$ exist. We keep the notation of Lemma \ref{Lemma 2.9}.
Consider the $\BP^1$-fibration defined by $|B_1+B_2|$. By (2) in Lemma \ref{Lemma 2.9}, $T=Z_1+Z_2+D_0+D_1$
is contained in a fiber and consists of $(-2)$-curves. Since any component $C$ of $Z_3-(B_1+B_2)$ meeting
$B_1+B_2$ is a $1$-section, $C$ meets $D_0+Z_1+Z_2+D_1$, hence $D_1$, by Lemma \ref{Lemma 2.11} below.
But both $B_1, B_2$ are branching components of $Z_3$, so there are at least four such $C$.
On the other hand, only one can meet $D_1$ since $D_1$ is a tip of $Z_3$. This is a contradiction.
\end{proof}

The following easy remark on a singular fiber of a $\BP^1$-fibration is used in the proof of
Corollary \ref{Corollary 2.10}.

%Lemma 2.11
\begin{lem}\label{Lemma 2.11}
Suppose that $f : \ol{M} \to B$ is a $\BP^1$-fibration on a smooth projective surface $\ol{M}$.
Suppose that $T=H+L_1+L_2+L_3$ is a $(-2)$-fork contained in a fiber $F_T$, where $H^2=-2$, $L^2_i=-2\
(i=1,2,3)$ and $H$ is a branching component. Then every $1$-section of $f$ meets $T$.
\end{lem}

\begin{proof}
Let $L_i'$ be the  branch at $H$ of $F_T$ containing $L_i$. Considering the process of shrinking $F_T$ to a
$(0)$-curve we see that for precisely one $i$, say for $i=1$, we  have $L_1\neq L_1'$ and that $L_1'$ shrinks to a point on $H$ in such a way that $H$ is turned into a $(-1)$-curve. It follows that $F_T=L_1'+H+L_2+L_3$ with each component of
$L_1'+H$ of multiplicity at least $2$ in $F_T$. Hence any $1$-section meets $L_2$ or $L_3$.

\end{proof}

The following result will determine the shape of $D$ when $D$ contains two $(-1)$-curves meeting
each other transversally in one point.

%Lemma 2.12
\begin{lem}\label{Lemma 2.12}
Suppose that $B_1, B_2$ are two $(-1)$-curves in $D$ such that $B_1\cdot B_2=1$. Then
$D=B_1+B_2+T_1+T_2+U_1+U_2$ where $T_i,U_j$, $i,j=1,2$ are linear chains, $B_1$ meets $T_1,T_2$ in tips,
$B_2$ meets $U_1,U_2$ in tips. Moreover $B_1+T_1+T_2$ and $B_2+U_1+U_2$ are contractible (possibly to smooth points).
\end{lem}
\begin{proof}
Let $Z_1,\dots, Z_r$ be the connected components of $D-B_1$, where we assume that $B_2\subset Z_r$.
Let $Z_i^{(1)}$ be the component of $Z_i$ meeting $B_1$. We denote by $f$ the $\BP^1$-fibration induced by $|B_1+B_2|$. Our proof consists of verifying several claims.
\svskip

\begin{claim}{1}
$Z_i^{(1)} \ (1 \le i \le r-1)$ is not a $(-1)$-curve. In particular, $Z_1,\dots, Z_{r-1}$ are MNC
divisors.
\end{claim}
\begin{proof}
Suppose that $Z_1^{(1)}$ is a $(-1)$-curve. We apply Proposition \ref{Proposition 2.6} with
$D_0=Z_1^{(1)}$. Let $R_1,\ldots,R_s$ be the connected components of $D-Z_1^{(1)}$, where we assume
that $B_1+B_2\subset R_s$. Suppose that $R_1^{(1)}$ is a $(-1)$-curve. It meets the $1$-section $Z_1^{(1)}$ of $f$. Since $R_1^{(1)}$ is branching in $D$ this gives a contradiction
to Lemma \ref{Lemma 1.1}~(e) as before.
Hence $R_1,\dots, R_{s-1}$ are MNC divisors. Note that $R_s$ is not contractible since it contains
$B_1+B_2$. In particular, \ref{Setup 2.1}~(c) is satisfied with respect to $D_0=Z_1^{(1)}$ and we can apply \ref{Proposition 2.6}.
By what we said, $D-Z_1^{(1)}$ is not of type (i). Suppose that $D-Z_1^{(1)}$ is  of type (iv).
Then one of $R_1, \ldots, R_s$, say $R_1$, is a linear chain, $s=3$, and there exists a $(-1)$-curve $C$
meeting $E$ and $R_1$ so that $C \cdot R_2=C\cdot R_3=0$. Hence $E+C+R_1$ is contained in a fiber of $f$. Since $Z_1^{(1)}$ is a $1$-section of $f$ and meets $R_1$, we have
$C\cdot Z_1^{(1)}=0$. Therefore $C\cdot D=1$ and we reach a contradiction because $C$ is then a simple
curve, which does not exist by Lemma \ref{Lemma 1.5}. By Lemma \ref{Lemma 2.7}~(ii), $D-Z_1^{(1)}$
is not of type (iii). Therefore $D-Z_1^{(1)}$ is of type (ii). Let $C$ be as specified in Proposition
\ref{Proposition 2.6}~(ii). Then $C$ connects $R_1$ and $R_2$, and $C+R_1+R_2$ is in a fiber of the
fibration $f$ for which $Z_1^{(1)}$ is a $1$-section. But $Z_1^{(1)}$ meets two
components of the fiber, which is a contradiction.
\end{proof}

\begin{claim}{2}
$B_1$ and $B_2$ are the only $(-1)$-curves contained in $D$.
\end{claim}
\begin{proof}
Suppose that $B_3$ be a $(-1)$-curve in $D$ different from $B_1,B_2$. By Claim 1, $B_3$ does not meet
$B_1$ and, by symmetry, it does not meet $B_2$. Hence $B_3$ is in a fiber $F$ of $f$.  As a (-1)-component of D, $B_3$ meets
at least three components of D. Among them, at most two are
in F. Hence one meets $B_1+B_2$ and is a $1$-section of $f$. This leads to a
contradiction with Lemma \ref{Lemma 1.1} ~(e).
\end{proof}

\begin{claim}{3}
$Z_i\ (1 \le i \le r-1)$ is a linear chain and $B_1$ meets the $Z_i$ in their tips.
\end{claim}
\begin{proof}
Suppose that $D_0\subset Z_1$ is a branching component of $Z_1$. Then it is a branching component of $D$.
One of the connected components of $D-D_0$ contains $B_1+B_2$ hence is not contractible. No other connected component is contractible to a point by Claim 2. So we can apply \ref{Proposition 2.6}. $D_0$
is not of type 2.6~(i). $D_0$ is not of type 2.6~(iv) or 2.6~(ii) by the same argument as in the proof
of Claim 1. So $D_0$ is of type 2.6~(iii) and this contradicts Corollary \ref{Corollary 2.10}. This proves the first assertion.

Now suppose that $Z_1^{(1)}$ is not a tip of $Z_1$. Then $Z_1^{(1)}$ is branching in $D$ and as above Proposition \ref{Proposition 2.6} applies. $Z_1^{(1)}$ is not of type 2.6~(i), 2.6~(iv).
It is not of type 2.6~(iii) by Corollary \ref{Corollary 2.10}. Hence it is of type 2.6~(ii).
Now we argue as at the end of the proof of Claim 1.
\end{proof}

By symmetry on $B_1$ and $B_2$, all the connected components of $D-B_2$ except the component $B_1'$ containing $B_1$
are linear chains and MNC.  Suppose $r\geq 4$. Then $B_1'$ is MNC since $B_1$ is branching and the only $(-1)$-component. We can therefore apply Proposition \ref{Proposition 2.6}. By Lemma \ref{Lemma 2.8}, $B_2$ is not of type 2.6~(ii). Since $B_1$ is a branching
$(-1)$-component, $B_1'$ is not contractible  and $B_2$ is not of type 2.6~(i).
Clearly $B_2$ is not of type 2.6~(iv). Hence $B_2$ is of type 2.6~(iii).
However, this is a contradiction by Lemma \ref{Lemma 2.7}~(ii). So $r=3$. Now we change notations.
Put $Z_1=T_1,Z_2=T_2$. By symmetry, let $U_1,U_2$ be the two twigs of $D$ sprouting out of $B_2$. Suppose
that $B_1+T_1+T_2$ is not contractible. We can then again apply Proposition \ref{Proposition 2.6} to $B_2$. $B_2$ is not of type 2.6~(i) or 2.6~(iv).
It is not of type 2.6~(iii) by Lemma \ref{Lemma 2.7}~(ii). It is not of type 2.6~(ii) since none of the connected
components $T_1+B_1+T_2, U_1$ and $U_2$ is a fork. This is a contradiction. Hence $B_1+T_1+T_2$ is
contractible. By symmetry, $B_2+U_1+U_2$ is contractible. This completes the proof of
Lemma \ref{Lemma 2.12}.
\end{proof}

We recall the notation $P=(K_{\ol{S}}+D+E)^+ =K_{\ol{S}}+(D+E)^\#=K_{\ol{S}}+D+E-Bk(D+E)$. Let $B$ be
a branching component of $D$. We say that a connected component $Z$ of $D-B$ is {\em $P$-quasi-orthogonal}
if $D_i\cdot P=0$ for every irreducible component $D_i \subset Z$ such that $D_i\cdot B=0$. We make a few straightforward observations that we will use below.

%Observation 2.12.2
\begin{obs}\label{Observation 2.12.2}
(i) If $D$ is a fork and $B$ its branching component, then each connected component of $D-B$ is $P$-quasi-orthogonal.\\
(ii) Suppose $D$ is not a fork. Let $T$ be an irreducible  component of $D$ and $\beta_T$ its coefficient in $(D+E)^\#$.\\
(a) If $T$ is not in a maximal twig of $D$, then $\beta_T=1$.\\
(b) If $T$ does not meet a maximal twig of $D$, then $T\cdot P =T\cdot (K_{\ol{S}}+D+E)=-2+\alpha_T$, where $\alpha_T$ is the branching number of $T$. If $T$ is also non-branching, then $T\cdot P =0$.\\
(c) If $T$ is in a maximal twig $T'$ of $D$, then $T$ is in $Supp(Bk(D+E))$ and $T\cdot P =0$. If $T$ is the tip of $T'$ meeting $D-T'$, then $\beta_T=1-\frac{1}{d(T')}$.\\
(d) Let $B$ be a branching component of $D$ and $Z$ a connected component of $D-B$. Then $Z$ is $P$-quasi-orthogonal if $T\cdot P = 0$ for every component $T$ of $Z$ that does not meet $B$ and is branching in $D$.
\end{obs}

%Lemma 2.13
\begin{lem}\label{Lemma 2.13}
Assume that there are no two $(-1)$-curves in $D$ which meet each other. Suppose that $B$ is
a $(-1)$-curve in $D$. Then all connected components of $D-B$ are MNC divisors, contractible
and $P$-quasi-orthogonal.
\end{lem}

\begin{proof}
Let $D-B=Z_1\cup\cdots\cup Z_r$. Then the $Z_i$ are MNC divisors, i.e., $Q_i=Z_i$,
because no $Z_i^{(1)}$ is a $(-1)$-curve by the assumption that no two $(-1)$-curves meet each other. In particular Proposition \ref{Proposition 2.6} applies to any $(-1)$-curve in $D$.
\svskip

\begin{claim}{\hspace{-1mm}}
$B$ is not of type 2.6~(iv).
\end{claim}

\begin{proof} Suppose the contrary. We assume that $Z_1$ is the unique chain in $D-B$ and that $Z_3$ is not contractible. We will construct a sequence $H_0, H_1,\dots$ of branching
components of $D$ which satisfies the following conditions.
\begin{enumerate}
\item[(1)]
{\em $H_0=B$. For $n > 0, H_0,\ldots, H_n$ are contained in a connected component of $D-H_{n+1}$. This component is
a non-contractible unimodular tree. In particular $H_i\neq H_j$ for $i\neq j$.}
\item[(2)]
{\em Each $H_i$ is of type 2.6~(iv).}
\item[(3)]
{\em The connected components of $D-H_n$ are MNC divisors.}
\end{enumerate}
The construction proceeds as follows.
If $n >0$, let $T_1,T_2,T_3$ be the connected components of $D-H_n$ with $T_1$ linear and $T_3$ containing $H_0, \ldots, H_{n-1}$. If $n=0$, we put $T_i=Z_i$. $T_2, T_3$ are non-linear unimodular trees, $T_3$ is non-contractible.
Let $H$ be a branching component of $T_2$. Suppose that there exists
a $(-1)$-component $H'$ of $T_2$ which meets $H$. Since no other $(-1)$-component of $D$ meets $H'$,
all connected components of $D-H'$ are MNC. In particular, $H'$ is not of type (ii) by Lemma
\ref{Lemma 2.8}. Since the connected component of $D-H'$ containing $H_n+T_3$ is not contractible,
$H'$ is not of type (i). It is not of type (iii) by Lemma \ref{Lemma 2.9}. Hence $H'$ is
of type (iv), and we may put $H_{n+1}=H'$. Suppose that there is no $(-1)$-component in $T_2$
which meets $H$. Since the connected component of $D-H$ containing $H_n+T_3$ is not contractible, no connected component contracts to a smooth point and we can apply
Proposition \ref {Proposition 2.6} to $H$.
$H$ is not of type (i) and it follows from Proposition \ref{Proposition 2.6} that
$H$ has three branches in $D$. This implies that $H \cdot H_n=0$, for otherwise $H$ has four branches
since it is a branching in $T_2$. Hence no $(-1)$-curve in $D$ meets $H$ and all connected components of $D-H$ are
MNC divisors. In particular, $H$ is not of type (ii) by Lemma \ref{Lemma 2.8}. Furthermore,
$H$ is not of type (iii), for otherwise $2$ divides $d(Z_2)$ which must be $\pm 1$. Hence $H$ is of
type (iv), and we may put $H_{n+1}=H$.

It is clear that such a sequence cannot exist since the number of irreducible components
of $D$ is finite. The claim is proved.
\end{proof}

$B$ is not of type 2.6~(ii) by Lemma \ref{Lemma 2.8}. It is not of type (iii) by Lemma \ref{Lemma 2.7}.
Hence $B$ is of type (i). A further $(-1)$-curve does not exist since it would be branching in $D-B$.
Suppose that $Z_1$ is not $P$-quasi-orthogonal. By \ref{Observation 2.12.2}~(d),  $Z_1$ is not a chain. Hence $Z_1$ is a fork, and again by \ref{Observation 2.12.2}~(d)
the branching component $B_1$ of $Z_1$ is disjoint from $B$.
 The connected components of $D-B_1$ are MNC
divisors since $B$ is the unique $(-1)$-component of $D$. We can apply Proposition \ref{Proposition 2.6} with  $D_0=B_1$. Clearly, $B_1$ is not of type (i) since $B$
is a branching $(-1)$-component in a connected component of $D-B_1$. Two of the connected components
of $D-B_1$ are also connected components of $Z_1-B$ and hence linear chains. So $B_1$ is not of type (iv).  By Lemma \ref{Lemma 2.8}, $B_1$ is not
of type (ii). Thus $B_1$ is of type (iii). We have $Z_1 = B_1+Q_1+Q_2+Q_3$, where $Q_1, Q_2$ are $(-2)$-curves contracting to quotient singularities and $Q_3$ is a linear chain that is part of the non-contractible connected component of $D- B_1$. Now $P=K_{\ol{S}}+
B_1+\frac{1}{2}Q_1+\frac{1}{2}Q_2+Q_3'+\cdots$ where $Q_3'$ is the component of $Q_3$ meeting $B_1$. We calculate $P\cdot B_1=-2+1+\frac{1}{2}+\frac{1}{2}=0$.
By \ref{Observation 2.12.2}~(d)
$Z_1$ is $P$-quasi-orthogonal and we have a contradiction.
\end{proof}

The following result determines more precisely the shape of $D$.

%Proposition 2.14
\begin{prop}\label{Proposition 2.14}
One of the following two cases takes place.
\begin{enumerate}
\item[{\rm (a)}]
There exists a branching curve $B$ in $D$ such that all connected components of $D-B$ are MNC divisors,
contractible and $P$-quasi-orthogonal.
\item[{\rm (b)}]
$D$ has exactly two branching components $B_1,B_2$ such that $D=T_1+T_2+B_1+T+B_2+U_1+U_2$,
where $T$ is a linear chain, possibly empty, connecting $B_1$ and $B_2$, and $T_1,T_2$ (resp. $U_1,U_2$)
are maximal twigs of $D$ sprouting out of $B_1$ (resp. $B_2$). Moreover, $B_1+T_1+T_2+T$ and
$B_2+U_1+U_2+T$ are contractible.
\end{enumerate}
\end{prop}

\begin{proof}
Lemmas \ref{Lemma 2.12} and \ref{Lemma 2.13} give the result if there is a $(-1)$-curve in $D$. We assume therefore that $D$ does not contain
a $(-1)$-curve. Then all divisors $D'\subset D$ are MNC. In particular we can apply Proposition \ref{Proposition 2.6} to any branching component $D_0$ of $D$. By Lemma \ref{Lemma 2.8}, the case 2.6~(ii) can never occur. These observation will be used repeatedly in the proof below.

It is straightforward to prove
by induction on $\#D$ that there exists a branching curve $B_1$ in $D$ such that, if $Z_1,\ldots, Z_r$
denote the connected components of $D-B_1$, $Z_1,\dots, Z_{r-1}$ are linear chains and $B_1$ meets them
in tips. Hence $Z_1,\cdots,Z_{r-1}$ are maximal twigs of $D$ and contractible Corollary 1.4.
\svskip

Suppose that $Z_r$ is not contractible. Then $B_1$ is not of type 2.6~(i). If $B_1$ were of
type 2.6~(iv) we have $r=3$ and only one of $Z_1, Z_2$ is a linear chain. So this does not occur. Suppose that $B_1$ is of type (iii), hence
that $r=3$ and $Z_1,Z_2$ are $(-2)$-curves. By Lemma \ref{Lemma 2.9}, $B_1$ meets a tip of a maximal twig
$R$ of $Z_3$. Since $Z_3$ is not contractible it is not a linear
chain for otherwise every component of $Z_3$ has self-intersection
number $\leq -2$ by Corollary \ref{Corollary 1.2.0.1} and $Z_3$ would be contractible.

Hence $R$ sprouts from a branching component $B_2$ of $Z_3$. Let $R_1,\ldots, R_s$ be
the connected components of $D-B_2$ with $R \subset R_1$. Then $R_1=B_1+Z_1+Z_2+R$ and we claim it is contractible. First, $B_1^2\leq -2$ by Lemma \ref{Lemma 2.7}. Next, $R$ is a contractible linear chain by
 Corollary \ref{Corollary 1.2.0.1}. Hence $R_1$ is a contractible fork of type $(2,2,m)$ for some $m$.
Since $R_1$ is apparently not a $(-2)$-curve, $B_2$ is not of type (iii).
 Further, $B_2$ is not of type (iv) since the intersection matrix
of $R_1$ is not unimodular as the determinant is divisible by $2$. So $B_2$ is of type (i), i.e.,
$R_i$ is contractible for $i=1,\ldots, s$.

Now $R_1$ is $P$-quasi-orthogonal. In fact, by  \ref{Observation 2.12.2}~(d),
we have to check only $B_1\cdot P =0$, and this is done as it was for $B _1$ in the proof of \ref{Lemma 2.13}.
Suppose that $R_2$ is not $P$-quasi-orthogonal.
This implies as before that $R_2$ is a fork and that $B_3 \cdot B_2=0$, where $B_3$ is the branching component of $R_2$. The connected component $U$ of $D-B_3$ containing $B_2$ is not
contractible since $B_1$ and $B_2$ are two branching components in it. Hence $B_3$ is not of type (i). Furthermore,
the connected component $U$ has two twigs consisting of single $(-2)$-curves, hence the intersection matrix
of $U$ is not unimodular as the determinant is divisible by $2$. Hence $B_3$ is not of type (iv)
as $U$ is not a linear chain. $B_3$ is not of type (iii) for otherwise we see that $R_2$ is
$P$-quasi-orthogonal as before. The case of type (ii) does not occur with $B_3$ by the remark at the beginning of
the proof. So, we are led to a contradiction. Hence $R_2$ must be $P$-quasiorthogonal, and similarly $R_i, i \geq 2$. So we have the case (a) with $B=B_2$.
\svskip

We may assume then that $Z_r$ is contractible. If $Z_r$ is $P$-quasi-orthogonal, we are done. Suppose that $Z_r$
is not $P$-quasi-orthogonal. Then again $Z_r$ is not a linear chain, hence a fork, and the branching
component $B_2$ of $Z_r$ is disjoint from $B_1$. Let
$R_1,R_2,R_3$ be the maximal twigs of $Z_3$, where we assume that $R_1$ meets $B_1$. Let $D_1$ be the component of
$R_1$ such that $D_1\cdot B_1=1$. Let $R_1'=B_1+Z_1+\dots+Z_{r-1}+R_1$. Then $R_1',R_2,R_3$ are
the connected components of $D-B_2$. Suppose that $r>3$ or $D_1$ is not a tip of $R_1$. Then $R_1'$
is not contractible, hence $B_2$ is not of type (i). $B_2$ is not of type (iii), for otherwise $Z_r$ is
$P$-quasi-orthogonal. So, $B_2$ is of type (iv), but two of the connected components of $D-B_2$ are
linear chains, a contradiction. So, $r=3$ and $D_1$ is a tip of $R_1$.  Suppose further that
$R_1'=B_1+Z_1+Z_2+R_1$ is not contractible. Then $B_2$ is not of type (i). It is not of type (iii),
for otherwise $Z_3$ is $P$-quasi-orthogonal. Hence $B_2$ is of type (iv), but $R_2, R_3$ are chains, a
contradiction. Thus $B_1+Z_1+Z_2+R_1$ is contractible. So we have the case (b) with $T=R_1, T_1=Z_1, T_2=Z_2, U_1=R_2, U_2=R_3$.
\end{proof}

We have a remark to add on the singular point $q$ of $S'$ in the case (a).

%Lemma 2.15
\begin{lem}\label{Lemma 2.15}
In the case (a) of Proposition \ref{Proposition 2.14}, $K_{\ol{S}}\cdot E>0$. This implies that
$q$ is not a rational double point.
\end{lem}
\begin{proof}
Suppose that $q$ is a rational double point. Then $K_{\ov{S}}$ is trivial when restricted to a small punctured neighborhood
$U$ of $q$. Since $S'$ is a $\Z$-homology plane, it is known that the natural homomorphism
$H_1(U,\Z) \to H_1(S_0,\Z)$ is an isomorphism, and hence the restriction map $\Pic(S_0) \to \Pic(U)$
is an isomorphism, see \cite[Proposition 2.6~(ii)]{KR}. Hence $K_{\ov{S}}$ is trivial when restricted to $S_0$ and $K_{\ov{S}}$ is supported on $D+E$. Since $E$ has negative definite intersection matrix and $E_j\cdot K_{\ov{S}}=0$ for every component $E_j$ of $E$, there actually is a canonical divisor
$K_{\ol{S}}$ supported by $D$. Write $K_{\ol{S}}=a_0B+\sum a_iD_i$. Since all the $Z_i$ are MNC divisors
and contractible, we can use \cite[Lemma 4.1]{GS} to conclude that either all $a_j\geq 0$ or $a_0<0$ and $a_j \leq 0$ for $j \geq 1$. In the first case the geometric genus $p_g(\ol{S})>0$. This is a contradiction, for $S'$
is a $\Z$-homology plane and hence rational. In the second case, if $D_i$ is the component
of some $Z_j$ meeting $B$ and $a_i=0$, then $D_i\cdot K_{\ov{S}}<0$, which is not the case as $D_i^2\leq -2$. Since the intersection form on $Z_j$ is negative definite, $a_s < 0$ for every component $D_s$ of $Z_j$. Hence $a_j<0$ for all $j$ and $K_{\ov{S}}+D \leq 0$. It follows that $\lkd(S_0)= \kappa(K_{\ov{S}}+D+E) \leq \kappa(E)$. Since $E$ contracts to a normal point we have $h^0(nE)=1$ for every $n$ so $\kappa(E)=0$ and reach a contradiction.
\end{proof}

\section{Determination of the nef part $P$}
We provide further details about the situation in \ref{Proposition 2.14}.

%Lemma 3.5
\begin{lem}\label{Lemma 3.5}
$d(D)<0$ and hence $a=d(E)=-d(D)$.
\end{lem}
\begin{proof}
Apply Lemma \ref{Lemma 2.7.1} with $D_0=B$ in the case (a) and $D_0=B_1$
in the case (b) of Proposition \ref{Proposition 2.14}.
\end{proof}

Lemmas \ref{Lemma 3.1} through \ref{Lemma 3.4} refer to the situation of Proposition \ref{Proposition 2.14} (a). Let $Z_1,\ldots, Z_r$ be the connected
components of $D-B$. Let $C_i$ be the irreducible component of $Z_i$ which meets $B$. We recall that
$\Pic(\ol{S})\otimes\Q \simeq H^2(\ol{S};\Q)$ has a basis consisting of the classes of the irreducible
components of $D+E$ and the intersection form on it is non-degenerate. Then we can determine uniquely
a $\Q$-divisor $P_i\ (1 \le i \le r)$ such that $P_i$ is supported by $D$, $C_i\cdot P_i=1$ and
$D_j\cdot P_i=0$ for every $D_j\subset D$ different from $C_i$ and also a $\Q$-divisor $P_0$ supported
by $D$ such that $P_0\cdot B=1$ and $P_0\cdot D_j=0$ for every $D_j$ different from $B$. Let
$\ol{Z}_i=Z_i-C_i\ (1 \le i \le r)$ and let $e_i=\frac{d(\ol{Z}_i)}{d(Z_i)}$ be the {\em capacity}
. Let $a_{ij}$ denote the coefficient of $C_i$ in $P_j$
with $1 \le i \le r$ and $0 \le j \le r$. Let $c_i\ ( 0 \le i \le r)$ denote the coefficient of $B$
in $P_i$. Finally, let $\beta=B\cdot P$ and $\beta_i=C_i\cdot P$. We recall that $\beta \geq 0$ and $\beta_i \geq 0$.

%Lemma 3.1
\begin{lem}\label{Lemma 3.1}
The following equalities hold.
\begin{enumerate}
\item[{\rm (i)}]
$P=\beta P_0+\beta_1P_1+\dots+\beta_rP_r$.
\svskip
\item[{\rm (ii)}]
$a_{i0}=-\frac{d(\ol{Z}_i)}{d(D)}\prod\limits_{j\neq i}d(Z_j)
= \frac{d(\ol{Z}_i)}{a}\prod\limits_{j\neq i}d(Z_j)=\frac{e_i}{a}\prod\limits_{j=1}^r d(Z_j)$.
\svskip
\item[{\rm (iii)}]
$c_0 = -\frac{1}{d(D)}\prod\limits_{j=1}^rd(Z_j)=\frac{1}{a}
\prod\limits_{j=1}^r d(Z_j)$, \\
\svskip

\noindent
$c_i = -\frac{d(\ol{Z}_i)}{d(D)}\prod\limits_{j\neq i}d(Z_j)=a_{i0}$.
\item[{\rm (iv)}]
$a_{ij}=-\frac{d(\ol{Z}_i)d(\ol{Z}_j)}{d(D)}\prod\limits_{k\neq i,j}d(Z_k)=\frac{e_ie_j}{a}\prod\limits_{k=1}^rd(Z_k)$.
\end{enumerate}
\end{lem}
\begin{proof}
(i)\ Since the intersection of $P-(\beta P_0+\beta_1P_1+\cdots+\beta_rP_r)$ with every component of $D+E$
is zero, we have $P=\beta P_0+\beta_1P_1+\cdots+\beta_rP_r$.

(ii)\ Write
\[
P_0=c_0B+\sum_{i=1}^ra_{i0}C_i+\sum_{i=1}^r\sum_{j=1}^{n_i}d_{ij}Z_{ij}\ ,
\]
where $\ol{Z}_i=\sum_{j=1}^{n_i}Z_{ij}$ is the decomposition into irreducible components. The condition that $P_0\cdot B=1$
and $P_0\cdot D_j=0$ for any other component $D_j$ of $D$ gives rise to a system of linear
equations in the variables $c_0, c_i$ and $d_{ij}$. Solving it by means of Cramer's rule we
obtain the equalities. For the computation of the determinant of an intersection matrix we refer
to \cite[2.1.1]{KR}.

(iii) and (iv)\ Consider $P_i$ written as a linear combination of $B, C_i$ and the $Z_{ij}$ and
obtain a system of linear equations from the given condition that $P_i\cdot C_i=1$ and
$P_i\cdot D_j=0$ for any other component $D_j$. The coefficients are obtained by solving it.
The computation is similar to the case of $P_0$.
\end{proof}

%Lemma 3.2
\begin{lem}\label{Lemma 3.2}
Let $\Pi=\prod\limits_{i=1}^rd(Z_i)$. We have the following equalities.
\begin{enumerate}
\item[{\rm (i)}]
$P\cdot P_0=\dps{\frac{\Pi}{a}(\beta +\sum\limits_{i=1}^r\beta_ie_i)}$.
\item[{\rm (ii)}]
$P\cdot P_j=\dps{\frac{\Pi}{a}(\beta+\sum\limits_{i=1}^r\beta_ie_i)e_j}$.
\item[{\rm (iii)}]
$P^2=\dps{\frac{\Pi}{a}(\beta+\sum\limits_{i=1}^{r}\beta_ie_i)^2}$.
\item[{\rm (iv)}]
$\dps{\Pi(\beta+\sum\limits_{i=1}^r\beta_ie_i)^2\leq\frac{3a}{|\Gamma|}\leq\frac{3}{2}}$.
\end{enumerate}
\end{lem}
\begin{proof}
(i)\ Writing
\[
P_0=c_0B+\sum_{i=1}^ra_{i0}C_i+\sum_{i=1}^r\sum_{j=1}^{n_i}d_{ij}Z_{ij}\ ,
\]
take the intersection with $P$, where we note that $P\cdot Z_{ij}=0$ because $Z_i$ is
$P$-quasi-orthogonal. Then we obtain the equality by Lemma \ref{Lemma 3.1}.

(ii)\ If we use the expression of $P_j$ in terms of $B, C_i$ and $Z_{ij}$, the same argument as above
applies.

(iii)\ Use the expression $P=\beta P_0+\beta_1P_1+\cdots+\beta_rP_r$ and the equalities (i) and (ii).

(iv)\ The BMY-inequality in the form in \cite[Theorem 6.6.2]{Mi} is stated as
\[
0 < P^2=(K_{\ol{S}}+(D+E)^\#)^2 \le 3\left\{\chi(S')+\frac{1}{|\Gamma|}-1\right\}\ .
\]
Since we assume that the singular point $q$ is not cyclic, $\Gamma$ is not abelian. Hence
$a=d(E)=|\Gamma_{\rm ab}| \le \frac{|\Gamma|}{2}$, where $\Gamma_{\rm ab}$ is the abelianization
of $\Gamma$. Since $\chi(S')=1$, we obtain the
stated inequality.
\end{proof}

%Lemma 3.4
\begin{lem}\label{Lemma 3.4}
With the notations of Lemma \ref{Lemma 3.1}, the following assertions hold.
\begin{enumerate}
\item[{\rm (i)}]
$\beta_id(\ol{Z}_i)=\beta_ie_id(Z_i)$ is an integer.
\item[{\rm (ii)}]
$\beta_i=0$ if and only if either $Z_i$ is a linear chain and $C_i$ is a tip of $Z_i$, or $Z_i$ is a
linear chain $C_i+Z_{i1}+Z_{i2}$ with $Z_{i1}^2=Z_{i2}^2=-2$ or $Z_i$ is a fork of type $(2,2,n)$ and
$C_i$ is the tip of the $n$-twig.
\end{enumerate}
\end{lem}
\begin{proof}
We freely use Observation \ref{Observation 2.12.2}.\\
(i)\ If $C_i$ is in $\Supp(Bk(D))$, then $Z_i$ is a linear chain with $C_i$ as a tip and $\beta_i=0$. So suppose $C_i$ is not in $\Supp(Bk(D))$. Then among the components of $D$ meeting $C_i$ only those in a maximal twig of $D$ are in $\Supp(Bk(D))$. So if $T_1,\cdots,T_\ell$ are the maximal twigs of $D$ meeting $C_i$
we have
\begin{eqnarray*}
\beta_i=C_i\cdot P &=& C_i\cdot(K_{\ol{S}}+D+E)-C_i\cdot \Bk(D) \\
&=&C_i\cdot (K_{\ol{S}}+D)-\sum\limits_{j=1}^{l}\frac{1}{d(T_j)}\ .
\end{eqnarray*}
Every $T_i$ is a connected component
of $\ol{Z_i}$ and
\[d(\ol{Z_i}) = d(\ol{Z_i} - (T_1 + \dots + T_l))\prod\limits_{j=1}^ld(T_j).
\]
Hence  $\prod\limits_{j=1}^{l}d(T_j)$ divides $d(\ol{Z}_i)$. It follows that $\beta_id(\ol{Z}_i)$ is an integer.

(ii)\ Note that $Z_i$ is a contractible linear chain or a contractible fork. If $C_i$ is in $\Supp(Bk(D))$, then $C_i$ is a tip of a linear chain and $\beta_i=0$. Suppose that $C_i$ is not in $\Supp(Bk(D))$. Then there are three
possibilities.\\
(a) $Z_i$ is a linear chain and $C_i$ meets two maximal twigs $T_1, T_2$ of $D$. We then have $\beta_i=0$ if and only if $d(T_1)=d(T_2)=\frac{1}{2}$, i.e., if $T_1, T_2$ are $(-2)$-curves.\\
(b) $Z_i$ is a fork with $C_i$ as branching component and branches $T_1, T_2, T_3$ such that $\sum\limits_{j=1}^{3}\frac{1}{d(T_j)}>1$. In this case $\beta_i=-1+\sum\limits_{j=1}^{3}\frac{1}{d(T_j)}\neq 0$.\\
(c) $Z_i$ is a fork with branching component $B_i\neq C_i$. In this case $C_i$ is a tip of one of the branches, $T_3$ say, of the fork $Z_i$. Since $Z_i$ is $P$-quasi-orthogonal, we have $P\cdot B_i=0$. This again implies that the branches $T_1, T_2$ are $(-2)$-curves, i.e. that $Z_i$ is a $(2,2,n)$-fork with $C_i$ the tip of the n-twig.
\end{proof}

We consider next the case (b) in Proposition \ref{Proposition 2.14}. Let $P_i\ (i=1,2)$ denote
a $\Q$-divisor supported by $D$ such that $B_i\cdot P_i=1$ and $D_j\cdot P_i=0$ for every $D_j\neq B_i$.
Let $b_{ij}$ denote the coefficient of $B_i$ in $P_j$ and let $\beta_i=B_i\cdot P$.

%Lemma 3.3
\begin{lem}\label{Lemma 3.3}
Set $L=B_1+T_1+T_2+T$ and $R=B_2+T+U_1+U_2$ with the notations in the case (b) of Proposition
\ref{Proposition 2.14}. Then we have the following equalities.
\begin{enumerate}
\item[{\rm (i)}]
$P=\beta_1P_1+\beta_2P_2$. \svskip
\item[{\rm (ii)}]
$b_{11}=\dps{\frac{d(T_1)d(T_2)d(R)}{a}}$, \quad $b_{21}=b_{12}=\dps{\frac{d(T_1)d(T_2)d(U_1)d(U_2)}{a}}$, \\
\par\noindent
$b_{22}=\dps{\frac{d(U_1)d(U_2)d(L)}{a}}$. \svskip
\item[{\rm (iii)}]
$P\cdot P_i=b_{1i}\beta_1+b_{2i}\beta_2$. \svskip
\item[{\rm (iv)}]
$P^2=b_{11}\beta_1^2+2b_{21}\beta_1\beta_2+b_{22}\beta_2^2$
\end{enumerate}
\end{lem}
\begin{proof}
The proof is similar to the one of Lemma \ref{Lemma 3.1}. (ii) requires Cramer's rule and some matrix
calculation. The rest is clear.
\end{proof}

\section{Further determination of $D$}

In this section, we prove that $D$ has exactly one branching component, i.e., the dual graph of $D$ is
star-shaped. We then call $D$ a {\em comb}. The following lemma shows that we are reduced to the case (a)
in Proposition \ref{Proposition 2.14}.

%Lemma 4.1
\begin{lem}\label{Lemma 4.1}
If $D$ is as in the case (b) in Proposition \ref{Proposition 2.14}, then it is also as in the case (a).
\end{lem}

\begin{proof}
We retain the notations of Lemma \ref{Lemma 3.3}. Since $d(R)>0$ and $d(L)>0$, $b_{11}$ and $b_{22}$ are
positive. As in the proof of Lemma \ref{Lemma 3.2}~(iv), we can also show in the present case
that $P^2\leq \frac{3}{|\Gamma|} \leq\frac{3}{2a}$. By Lemma \ref{Lemma 3.3} we obtain
\[
2ab_{21}\beta_1\beta_2 \le aP^2 \le \frac{3}{2}
\]
and hence
\begin{equation}\label{eqn 4.1}
2d(T_1)d(T_2)d(U_1)d(U_2)\beta_1\beta_2 \leq \frac{3}{2}\ .
\end{equation}
Since $\beta_i=P\cdot B_i$ and $P=K_{\ol{S}}+(D+E)^\#$, we have
\[
\beta_1=1-\frac{1}{d(T_1)}-\frac{1}{d(T_2)}\geq 0\ \ \mbox{and}\ \ \beta_2=1-\frac{1}{d(U_1)}-\frac{1}{d(U_2)}\geq 0\ .
\]
This implies that
$d(T_1)d(T_2)\beta_1$ and
$d(U_1)d(U_2)\beta_2$
are non-negative integers. If both are positive, we have a contradiction by (\ref{eqn 4.1}).
Hence we may assume that $\beta_2 = 0$, i.e., $d(U_1)=d(U_2)=2$. This means that $R=T+B_2+U_1+U_2$ is
$P$-quasi-orthogonal  and contractible. So, we are in the case (a).
\end{proof}

The following result will be used in the proof of Lemma \ref{Lemma 4.3} below.

%Lemma 4.2
\begin{lem}\label{Lemma 4.2}
If $a>4$, then there is no divisor $F$ with $\Supp(F) \subset D$ such that the condition {\rm (a)}
and one of the conditions {\rm (b)} and {\rm (c)} specified below are satisfied.
\begin{enumerate}
\item[{\rm (a)}]
$F$ is a fiber of a $\BP^1$-fibration $f : \ol{S} \to \BP^1$.
\item[{\rm (b)}]
$D-\Supp(F)=T_1 + T_2$ is a union of two disjoint  twigs of $D$, where $T_1$ is a single curve and
$T_2=T_2'+T_2''$ with a single curve $T_2'$ meeting $\Supp(F)$ and a possibly empty linear chain $T_2''$
which is a twig of $D$. Further, $T_1$ and $T_2'$ are $2$-sections of the fibration.
\item[{\rm (c)}]
$D-\Supp(F)=U_1+U_2$, where $U_1$ is a single curve, $U_2$ is a possible empty linear chain, $U_1+U_2$ is
a twig of $D$ and $U_1$ is a $4$-section of the fibration.
\end{enumerate}
\end{lem}

\begin{proof}
Suppose that a divisor $F$ exists. Let $F_E$ be the fiber containing $E$. Consider first the case (b).
With notation as in Lemma \ref{Lemma 1.2.1} (applied to $\ov{S}\setminus (D\cup E)$, we have $\Sigma=1$
since $\nu=1, h=2$ and $\wt{b}_1(S_0)=\wh{b}_2(S_0)=0$.

Let $C$ be an $S_0$-component $C$ of $F_E$. Then $C \cdot E \le 1$ since $C, E$ are both part of the fiber $F_E$. We have $C \cdot D > 0$ since $S'$ is affine.  Since a simple curve
does not exist by Lemma \ref{Lemma 1.5}, we have $C\cdot D>1$. If $C\cdot T_2'' > 0$, then $C, T_2''$ are part
of $F_E$ and $C\cdot T_2''=1$.

Suppose that $C$ has multiplicity $m>1$ in $F_E$. Then $C\cdot T_1\leq 1$ and  $C\cdot T_2 \leq 1$ since $T_1, T_2'$ are $2$-sections. We now conclude  $C\cdot T_1= 1$ and  $C\cdot T_2 = 1$ and $m=2$. Moreover $F_E\cdot T_1=C\cdot T_1$ and $C$ is the only component of $F_E$ meeting $T_1$. If $C\cdot T_2'=1$, then $C$ is the only component of $F_E$ meeting $T_2$, and if $C\cdot T_2''=1$, then $C\cdot T_2'=0$.
If $C$ is the only $S_0$-component of
$F_E$, then it is the unique $(-1)$-component. By induction on the number of components it follows then from $m=2$ that $F_E-C$
consists of $(-2)$-curves. More precisely, $F_E-C = E$  is a fork of type $(2,2,n)$, with branching component a $(-2)$-curve and $C$ meeting
the tip of the $n$-twig, which consists of $(-2)$ curves well.  This implies $d(E)=a =4$, contrary to our
hypothesis. Hence $F_E$ contains another $S_0$-component $C'$. Since $\Sigma=1$, $C$ and $C'$ exhaust
the $S_0$-components of $F_E$.
By Lemma \ref{Lemma 1.5} and what we said above, we have $C'\cdot T_2>1$. Since $C'\cdot T_2'' \leq 1$ we have $C'\cdot T_2' > 0$ and $C\cdot T_2'=0$.   This means that $C$
meets $T_2''$, in particular that $T_2''\neq \emptyset$. It now follows that $C'$ has multiplicity $m'=1$ in $F_E$ since otherwise $T_2'\cdot F_E>2$. For the same reason, $C'\cdot T_2' =1$. This gives $C'\cdot T_2''=1$.
If $C'$ is not a $(-1)$-curve then $C$ is the unique $(-1)$-component in $F_E$, and we reach
a contradiction to the hypothesis $a > 4$ as above. So $C'$ is a $(-1)$-curve and it is a tip of $F_E$
since $m'=1$. It follows that $C'\cdot E=0$.  We now contract $C'$ and consecutively contractible components of $C'+T_2''$. Then $E$
is not touched in this process. At some stage of the process, $C$ becomes the unique $(-1)$-component of
the fiber. Since $m=2$, the remaining components of the fiber, in particular those of $E$, are
$(-2)$-curves. This leads to a contradiction to the hypothesis as above.

Thus any $S_0$-component of $F_E$ has multiplicity $1$. At least one of them is a $(-1)$-curve. By Lemma \ref{Lemma 1.1}~(a) there is another $(-1)$-curve, necessarily an $S_0$-component. Since $\Sigma=1$ these two
exhaust the $S_0$-components of $F_E$, i.e., $F_E$ has precisely two $(-1)$-components, both of multiplicity $1$. By \ref{Lemma 1.1}~(f), $F_E$ is a linear chain, but it contains $E$, which we
assumed is not a linear chain.

Consider next the case (c). We have $\Sigma=0$ in Lemma \ref{Lemma 1.2.1} since now $h=1$. So $F_E$ has a unique $S_0$-component $C$, which is also a unique $(-1)$-component of $F_E$. If $\mult(C)=2$,
then we have a contradiction as above. So, $\mult(C) \geq 3$. By Lemma \ref{Lemma 1.5}, $C\cdot (U_1+U_2)>1$.
Hence $C$ meets $U_1$ and $C\cdot U_1=1$ since $U_1$ is a $4$-section, and $C$ must also meet $U_2$. It follows that $\mult(C)=3$ and that
the component of $U_2$ which meets $U_1$ has multiplicity $1$ in $F_E$. Now Lemma \ref{Lemma 1.1}~(f) implies
that the connected component of $F_E \setminus C$ containing $E$ is a linear chain. In particular,
$E$ is a linear chain. This is a contradiction.
\end{proof}

The following result shows that the divisor $D$ (or its dual graph) is {\em star-shaped}.

%Lemma 4.3
\begin{lem}\label{Lemma 4.3}
$D$ has exactly one branching component.
\end{lem}
\begin{proof}
In view of Lemma \ref{Lemma 4.1}, we may assume that $D$ is as in the case (a) of Proposition
\ref{Proposition 2.14}. Let $D-B=Z_1\cup\dots\cup Z_r$ be the decomposition into connected components.
Let $Z_1,\dots Z_s$ exhaust all connected components $Z_i$ such that $Z_i$ is a linear chain and $B$
meets $Z_i$ in a tip. Let $C_j$ be the component of $Z_j$ such that $C_j\cdot B=1$. Then
$C_j\not\subset Supp(Bk(D))$ for $j>s$. Hence we have
\[
\beta=B\cdot P=B\cdot (K_{\ol{S}}+D+E)-B\cdot \Bk(D)=r-2-\sum\limits_{i=1}^s\frac{1}{d(Z_i)}\ .
\]
Suppose that $s<r$. If $s=r$ then $D$ is star-shaped. \\
(i) Suppose further that $\beta>0$. Take an index $k$ with $r \ge k > s$. Looking at one term in Lemma
\ref{Lemma 3.2}~(iv), we obtain
\[
2\beta\beta_ke_k\Pi < \frac{3}{2}\ ,
\]
where $\Pi=\prod\limits_{i=1}^rd(Z_i)$. We group this as
\[
(2\beta\prod\limits_{i=1}^sd(Z_i))\cdot(\beta_kd(\ol{Z}_k))\cdot (\prod\limits_{j>s,j\neq k}d(Z_j))<\frac{3}{2}\ .
\]
Since
\[
\beta\prod\limits_{i=1}^sd(Z_i)=(r-2-\sum\limits_{i=1}^s\frac{1}{d(Z_i)})\prod\limits_{i=1}^sd(Z_i)
\]
is a positive integer and since $d(\ol{Z}_k)\beta_k$ is a non-negative integer by Lemma \ref{Lemma 3.4}~(i), we must have $\beta_k=0$. Since $\beta_i=0$ for $i=1,\ldots, s$, we obtain from Lemma \ref{Lemma 3.2}~(iv) that
\[
\left(r-2-\sum\limits_{i=1}^s\frac{1}{d(Z_i)}\right)^2\prod\limits_{i=1}^s d(Z_i)
\prod\limits_{j>s}d(Z_j)\leq\frac{3}{2}\ .
\]
Since  $d(Z_i)\geq 2$ for $i=1,\ldots,s$ we have
\begin{equation}\label{eqn 4.2}
\left(r-2-\frac{s}{2}\right)^22^s\prod\limits_{j>s}d(Z_j)\leq\frac{3}{2}\ .
\end{equation}
Since $s\leq r-1$ we obtain
\begin{equation}\label{eqn 4.3}
\left(\frac{r-3}{2}\right)^22^s\prod\limits_{j>s}d(Z_j)\leq\frac{3}{2}\ .
\end{equation}
Since $\beta_r=0$, Lemma \ref{Lemma 3.4}~(ii) implies that either $Z_r=C_r+U_1+U_2$ is a linear chain
composed of $3$ components with $U_1^2=U_2^2=-2$, or $Z_r$ is a $(2,2,n)$-fork and $C_r$ is the tip of the $n$-twig.

In the first case, we have
$C_r^2 \le -2$ by Lemma \ref{Lemma 2.7}, whence $d(Z_r) \geq 4$. In the second case, $d(Z_r)$ is divisible
by $4$, whence $d(Z_r)\geq 4$ again. Now (\ref{eqn 4.3}) implies that $r=3$ and (\ref{eqn 4.2}) then
yields $s=2$. Lemma \ref{Lemma 3.2}~(iv) again gives
\begin{equation}\label{eqn 4.4}
d(Z_1)d(Z_2)d(Z_3)\left(1-\frac{1}{d(Z_1)}-\frac{1}{d(Z_2)}\right)^2\leq \frac{3}{2}\ .
\end{equation}
We may assume that $d(Z_1)\leq d(Z_2)$. If $d(Z_1)=d(Z_2)=2$ we obtain $\beta =0$ contrary to our assumption.
Hence $d(Z_1)d(Z_2)\geq 6$ and $\frac{1}{d(Z_1)}+\frac{1}{d(Z_2)}\leq \frac{5}{6}$. Now (\ref{eqn 4.4})
gives $d(Z_3)\leq 9$, $d(Z_1)=2$ and $d(Z_2)=3$. Also $\beta =\frac{1}{6}$, and $Z_2$ is a $(-3)$-curve or a chain of two $(-2)$-curves.

(i.1) We assume that $Z_3$ is a $(2,2,n)$-fork. Then $d(Z_3)$ is divisible by $4$. Hence $d(Z_3)=4$ or
$d(Z_3)=8$. Let $B_1$ be the branching component
of $Z_3$ and let $U_1,U_2$ be the $(-2)$-twigs of $Z_3$. Let $T$ be the third twig (the $n$-twig) of $Z_3$
which meets $B$ in the tip $C_3$. Let $b=-B^2, b_1=-B_1^2$.

Suppose that $Z_2$ is a single $(-3)$-curve. Suppose first that $d(Z_3)=4$. Then $Z_3$ is a fork composed only of $(-2)$-curves and $b_1=2$. Let $\ol{Z}_3=Z_3-C_3$. Then $d(\ol{Z_3})=4$ as well. We have $d(D)=d(Z_3)(6b-5)-6d(\ol{Z}_3)$.
Hence $d(D)=24b-44$. Since $d(D)<0$ and since $b > 0$ by Lemma \ref{Lemma 2.7}, we obtain $b=1$,
i.e., $B$ is a $(-1)$-curve, and $a=-d(D)=20$. The divisor $F=2(B+T+B_1)+U_1+U_2$ induces
a $\BP^1$-fibration $f : \ol{S} \to \BP^1$, for which $Z_1$ and $Z_2$ are $2$-sections. We reach
a contradiction by Lemma \ref{Lemma 4.2}. Suppose next that $d(Z_3)=8$. This implies that $C_3^2=-3$ and all other
components of $Z_3$ are $(-2)$-curves. We compute $d(D)=48b-64$, whence $b=1$ as above. We consider
the $\BP^1$-fibration given by $2Z_1+4B+2T+2B_1+U_1+U_2$, for which $Z_2$ is a $4$-section. Again we have a contradiction by Lemma \ref{Lemma 4.2}.

Now suppose that $Z_2=C_2+Z_2'$ consists of two $(-2)$-curves with $C_2$ meeting $B$. Suppose first
that $d(Z_3)=4$. As above we find $d(D)=24b-52$. Hence $b=1$ or $2$ since $d(D) < 0$. Suppose that $b=2$.
Then $D$ consists of $(-2)$-curves and $d(D)=-4$. This implies that $a=4$ and hence $E$ is a $(2,2,n)$-fork
consisting only of $(-2)$-curves. Thus the generators of $\Pic(\ol{S})\otimes \Q$ are $(-2)$-curves,
which implies that $K_{\ol{S}}$ is numerically equivalent to $0$. This is a contradiction because $\ol{S}$
is a rational surface. Therefore $B^2=-1$ and $d(D)=-28$. Consider again the $\BP^1$-fibration
defined by $2(B+T+B_1)+U_1+U_2$. Now $C_2$ and $Z_1$ are $2$-sections and  $Z_2'$ is contained in a fiber.
This is in contradiction to Lemma \ref{Lemma 4.2}. Suppose next that $d(Z_3)=8$. Then, as seen above, $C_3^2=-3$
and all other components of $Z_3$ are $(-2)$-curves. We compute $d(D)=48b-80$, whence $b=1$ as $d(D) < 0$.
We consider the $\BP^1$-fribration defined by $2Z_1+4B+2T+2B_1+U_1+U_2$, for which $C_2$ is a $4$-section.
This is again in contradiction to Lemma \ref{Lemma 4.2}.

(i.2) We assume next that $Z_3=C_3+U_1+U_2$ is a linear chain, where $U_1^2=U_2^2=-2$ and $B\cdot C_3=1$.
Suppose first that  $Z_2=C_2+Z_2'$. We compute $d(D)=(6b-7)(4b_1-4)-24<0$, where $b_1=-C_3^2$. By Lemma \ref{Lemma 2.7}, $b_1\geq 2$ and $b\geq 1$. Hence either $b=b_1=2$ or $b=1$. In the first case, $d(D)=-4$ and hence $a=4$.
So $E$ is a $(-2)$-fork of type $(2,2,n)$, $D+E$ consists of $(-2)$-curves and we reach a
contradiction as above. In the second case, consider the branching component $C_3$. The branch containing $B$ is a non-contractible linear chain. We find that $C_3$ only fits case
(iii) in Proposition \ref{Proposition 2.6}, and hence $\lkd(\ol{S}-(D-C_3+E)) \geq 0$. However,
$Z_1+2B+C_2$ defines a $\BP^1$-fibration on $\ol{S}$ and an $\A^1$-fibration on $\ol{S}-(D-C_3+E)$,
whence $\lkd(\ol{S}-(D-C_3+E))=-\infty$. This is a contradiction. Suppose next that $Z_2$ is a single
$(-3)$-curve. We compute $d(D)=24bb_1-24b-20b_1-4$. Since $d(D)<0$, we obtain $6bb_1-6b-5b_1-1<0$,
i.e., $(6b-5)(b_1-1)<6$. Since $b_1\geq 2$ and $b\geq 1$, we have $b=1$. Since $\Pi=d(Z_1)d(Z_2)d(Z_3)
=24(b_1-1)$ and $\beta=\frac{1}{6}$ and since $\beta_i=0$ for $i=1,2,3$, Lemma \ref{Lemma 3.2}~(iv) yields
\[
\Pi\cdot \beta^2=\frac{1}{36}\cdot 24(b_1-1) \leq \frac{3}{2}\ .
\]
This implies that $b_1\leq 3$. Hence $b_1=2$ or $b_1=3$. We consider a $\BP^1$-fibration defined by
$2B+2B_1+U_1+U_2$ in the case $b_1=2$ and by $2Z_1+4B+2B_1+U_1+U_2$ in the case $b_1=3$ and obtain
a contradiction by Lemma \ref{Lemma 4.2}.

(ii) Suppose now that $\beta=0$. Then
\begin{equation}\label{eqn 4.5}
\beta= r-2-\sum\limits_{i=1}^s\frac{1}{d(Z_i)}=0\ .
\end{equation}
Since $d(Z_i) \ge 2$ for $1 \le i \le s$ and $s < r$, it follows from (\ref{eqn 4.5}) that $r=3$, $s=2$
and $d(Z_1)=d(Z_2)=2$. Since $\beta_1=\beta_2=0$, Lemma \ref{Lemma 3.2}~(iii) yields
\[
P^2=\frac{4d(Z_3)}{a}\beta_3^2e_3^2=\frac{4}{a}\beta_3^2\frac{d(\ov Z_3)^2}{d(Z_3)}\ .
\]
Since $d(D)=4d(Z_3)(-B^2-1-e_3)<0$, we have $e_3>-B^2-1$. In view of Lemma \ref{Lemma 2.7}, $B^2\leq -2$.
Hence $e_3 > 1$ and $d(\ol{Z}_3)>d(Z_3)$. Note that $\beta_3 >0$, for otherwise $P^2=0$ and
$\lkd(S_0)<2$.
Let $U_1,\dots, U_{\ell}$ be the maximal twigs of $D$ which meet $C_3$ and let $d_i = d(U_i)$. Then we have
\[
\beta_3=P\cdot C_3=-2+\alpha-\sum\limits_{i=1}^\ell\frac{1}{d_i}\ ,
\]
where $\alpha$ = $C_3\cdot (D-C_3)$ is the branching number of $C_3$ in $D$. Since $\beta_3>0$ we have $\alpha \geq 3$. We deduce from  Lemma
\ref{Lemma 3.2}~(iv) that
\[
4\left(-2+\alpha -\sum\limits_{i=1}^{\ell}\frac{1}{d_i}\right)^2\frac{d(\ol{Z}_3)^2}{d(Z_3)} \leq\frac{3}{2}\ .
\]
Since $d(\ov{Z_3}) > d(Z_3)$ and $\beta_3 > 0$ we find
\[
4\left(-2+\alpha -\sum\limits_{i=1}^{\ell}\frac{1}{d_i}\right)^2d(\ol{Z}_3) < \frac{3}{2}\ .
\]
Since $Z_3$ is a linear chain or a fork the case $\ell =0$ is ruled out and we have the following possibilities.\\
(a) $\alpha =3$ and $\ell=1$.\\
(b) $\alpha =3$ and $\ell=2$.\\
(c) $\alpha =4$ and $\ell=2$.\\
Suppose we have (a). Then $Z_3$ is a contractible fork and $\ov{Z}_3$ has two connected components, $U_1$ and another linear chain $U$. The component in $U$ that meets $C_3$  is the branching component of $Z_3$. We have $d(\ov{Z_3})=d_1d(U)$ and find
$4(1-\frac{1}{d_1})(d_1-1)d(U) <\frac{3}{2}$. This is not possible.\\
Suppose we have (b). Then $Z_3=U_1+C_3+U_2$ is a linear chain and $d(\ov{Z_3})=d_1d_2$. We obtain $4(1-\frac{1}{d_1}-\frac{1}{d_2})^2d_1d_2<\frac{3}{2}$.
We may assume $d_1 \leq d_2$. Since  $d_1 =2= d_2$ is ruled out by $\beta_3 >0$ we find $d_1=2$ and $d_2=3$. The configuration we obtain, with the roles of $B$ and $C_3$ reversed, has been ruled out as a possibility in (i) above.\\
Suppose we have (c). Then $Z_3$ is a fork with $C_3$ as branching component and $d(\ov{Z_3})=d_1d_2d_3$. We obtain $4(1-\frac{1}{d_1}-\frac{1}{d_2}-\frac{1}{d_3})^2d_1d_2d_3<\frac{3}{2}$. This is not possible.\\
\end{proof}

\section{First proof of Theorem 1}

We shall finish the proof of Theorem 1 in two different ways. The approach in this section is a continuation of the foregoing arguments and the detailed determination of $D$ and
$E$. The second approach to be given in the next section is a more conceptual one based on the topology of surfaces with
a good $\C^*$-action.  Both approaches start with Lemma \ref{Lemma 4.3}, where it was proved that the divisor $D$ (or its dual graph) is
{\em star-shaped}.\\

In the first approach we now proceed to narrow down the possibilities of the numerical data
for $D$ and $E$. We give precise data on $B^2$ and the number $r$ of maximal twigs $Z_i$
sprouting from $B$, and give bounds for the
determinants $d(Z_i)$. This leads to a finite list of configurations for $D$, which, when matched with the corresponding possibilities for $E$, are shown not to exist. \\

We begin by summarizing some information about $E$.

Proof of the following results owes much to \cite[6.18]{F}. We note again that the
thicker bark $\Bk^*E$ in \cite{F} for a fork $E$ is denoted here by $\Bk E$.

%Lemma 5.3
\begin{lem}\label{Lemma 5.3}
Let $H$ be the branching component of $E$ and let $H^2=-h$. Let $R_1,R_2,R_3$ be the maximal twigs of $E$.
We assume that $2=d(R_1)\leq d(R_2)\leq d(R_3)$. Let $e(R_i)$ denotes the inductance of $R_i$. Let
$\delta=\sum\frac{1}{d(R_i)}$. Then the following assertions hold.
\begin{enumerate}
\item[{\rm (i)}]
If $E$ is of type $(2,3,3)$, then $a=d(E)$ is divisible by $3$.
\item[{\rm (ii)}]
If $E$ is of type $(2,3,4)$, then $a$ is even and not divisible by $4$.
\item[{\rm (iii)}]
If $E$ is of type $(2,3,5)$, then $a$ is odd.
\item[{\rm (iv)}]
If $E$ is of type $(2,2,n)$, then $a=d(D)=4n(h-1)-4\ol{n}$, where
$\ol{n}$ denotes the determinant of the chain which is obtained from the $n$-chain of $E$ by deleting
the component meeting  $H$.
\item[{\rm (v)}]
If $a\leq 4$, then $E$ is a $(-2)$-fork, i.e., every component is a $(-2)$-curve.
\item[{\rm (vi)}]
$a\neq 5,6,18,25,35$.
\item[{\rm (vii)}]
$(\Bk E)^2=-\dfrac{d(R_1)d(R_2)d(R_3)(\delta-1)^2}{a}-\sum e(R_i)$.\\
If $E$ is of type $(2,2,n)$, this gives
\[
(\Bk E)^2=-1-\frac{4}{na}-\frac{\wt{ n}}{n}\ ,
\]
where $\wt{n}$ denotes the determinant of the chain which is obtained from the $n$-chain of $E$ by
deleting the tip of $E$. In particular, $(\Bk E)^2<-1$.
\end{enumerate}
\end{lem}
\begin{proof}
 Let $r_i=d(R_i)$. The possible $\ol{r_i}$ are determined by
$1 \leq \ol{r_i}<r_i$ and $\gcd(r_i, \ol{r}_i)=1$. We have
\begin{eqnarray*}
a &=& r_1r_2r_3h-\ol{r}_1r_2r_3-r_1\ol{r}_2r_3-r_1r_2\ol{r}_3 \\
&=& 2r_2r_3h-r_2r_3-2\ol{r}_2r_3-2r_2\ol{r}_3\ .
\end{eqnarray*}
The proof of (i) through (vi) is now elementary. As a sample we show $a \neq 6$ and then omit further details.
Suppose that $a=6$.  By (i) through (iv), $E$ is not of type $(2,2,n)$ or $(2,3,5)$. Suppose that $E$ is of type $(2,3,3)$. We have
$6=18h-9-6\ol{r}_2-6\ol{r}_3$, which is a contradiction. Suppose that $E$ is of type $(2,3,4)$. Then
$6=24h-12-8\ol{r}_2-6\ol{r}_3$, i.e., $9=12h-4\ol{r}_2-3\ol{r}_3$. Here $\ol{r}_2$ = 1 or 2, and $\ol{r_3}$ = 1 or 3. This equality
does not hold. \\
The calculation of $Bk(E)^2$ is given in \cite[6.18]{F}.
\end{proof}

%Lemma 5.0.1
\begin{lem}\label{Lemma 5.0.1}
We have $a=d(E) \geq 7$ and $a
\neq 18, 25, 35$.
\end{lem}
\begin{proof} This follows from Lemma \ref{Lemma 2.15}, Lemma \ref{Lemma 4.1} and (v) and (vi) in Lemma \ref{Lemma 5.3}.
\end{proof}

%Lemma 5.1
\begin{lem}\label{Lemma 5.1}
Let the notations be the same as in Lemma \ref{Lemma 4.3}. Then we have $r=3$.
\end{lem}
\begin{proof}
Let $d_i=d(Z_i)$ for $1 \le i \le r$. We assume that $d_1 \leq \ldots \leq d_r$. We have
$\beta=r-2-\sum\limits_{i=1}^r\frac{1}{d_i}$ and $\beta_i=0$ for $1 \le i \le r$.
>From Lemma \ref{Lemma 3.2}, (iv), we obtain
\begin{equation}\label{eqn 5.1}
\left(r-2-\sum\limits_{i=1}^r\frac{1}{d_i}\right)^2\prod\limits_{i=1}^rd_i \leq \frac{3}{2}\ .
\end{equation}
Suppose that $r\geq 4$. Then
\[
r-2-\sum\limits_{i=1}^r\frac{1}{d_i} \geq r-2-\frac{r}{d_1} \geq r-2-\frac{r}{2} \geq 0\ .
\]
So, we obtain from (\ref{eqn 5.1})
\[
\left(r-2-\frac{r}{d_1}\right)^2d_1^r \leq \frac{3}{2}\ .
\]
or equivalently
\[
\left(d_1(r-2)-r\right)^2d_1^{r-2} \leq \frac{3}{2}\ .
\]
Since $d_1^{r-2} \geq 4$, it follows that $d_1(r-2)-r=0$. This entails $r=4$ and $d_1=2$. Hence
(\ref{eqn 5.1}) reads as
\begin{equation}\label{eqn 5.2}
\left(\frac{3}{2}-\sum\limits_{i=2}^4\frac{1}{d_i}\right)^2d_2d_3d_4\leq\frac{3}{4}\ .
\end{equation}
Since $\frac{3}{2}-\sum\limits_{i=2}^4\frac{1}{d_i} \geq \frac{3}{2}-\frac{3}{d_2}\geq 0$,
(\ref{eqn 5.2}) entails
\[
\left(\frac{3}{2}-\frac{3}{d_2}\right)^2d_2^3 \leq \frac{3}{4}\ .
\]
It follows that $d_2=2$. In a similar way, we obtain $d_3=2$. Hence (\ref{eqn 5.1}) becomes
\[
\left(\frac{1}{2}-\frac{1}{d_4}\right)^2d_4 \leq \frac{3}{16}\ .
\]
This implies that $d_4=2$ or $3$. If $d_4=2$, then $\beta=0$ and $P^2=0$, which contradicts the
hypothesis $\lkd(S_0)=2$. It remains to consider the case $d_4=3$. Let $B^2=-b$. We compute
$d(D)=24b-36-8\ol{d}_4$, where $\ol{d}_4=1$ or $2$.  Thus $d(D)$ is divisible by $4$ and not divisible by $3$.
It follows from Lemma \ref{Lemma 5.3} that $E$ is a fork of type $(2,2,n)$. Suppose that $\ol{d}_4=2$. Since $d(D)<0$
by Lemma \ref{Lemma 3.5}, we obtain $b \leq 2$. If $b=2$, then $a=-d(D)=4$, which is ruled out by Lemma \ref{Lemma 5.0.1}.
 If $\ol{d}_4=1$,
then $d(D)<0$ implies $b\leq 1$, i.e., $B^2\geq -1 $. Blowing up points on $B\cap Z_4$ if necessary,
we may assume that $B^2=-1$. Let $Z_4=T_1+T_2$, where $T_2$ is a possibly empty linear chain and $T_1$ is
the component meeting $B$. The divisor $F=Z_1+2B+Z_2$ induces a $\BP^1$-fibration on $\ol{S}$ for
which $Z_3$ and $T_1$ are $2$-sections. But this is a contradiction to Lemma \ref{Lemma 4.2}.
\end{proof}

The following result narrows down further the possibilities for $D$.

%Lemma 5.2
\begin{lem}\label{Lemma 5.2}
With the above notations, the following assertions hold.
\begin{enumerate}
\item[{\rm (a)}]
$B^2=-1$ or $-2$.
\item[{\rm (b)}]
\[
P^2=\frac{d_1d_2d_3}{a}\left(1-\frac{1}{d_1}-\frac{1}{d_2}-\frac{1}{d_3}\right)^2\ ,
\]
where
\[
0 < d_1d_2d_3\left(1-\frac{1}{d_1}-\frac{1}{d_2}-\frac{1}{d_3}\right)^2 \leq \frac{3}{2}.
\]
\item[{\rm (c)}]
$d_2>2$.
\item[{\rm (d)}]
$d_1\leq 3,\ d_2\leq 5,\ d_3\leq 19$.
\end{enumerate}
\end{lem}
\begin{proof}
(a)\ We have $B^2\leq -1$ by Lemma \ref{Lemma 2.7}. If $B^2 \leq -3$, then $d(D)>0$ because
\begin{eqnarray*}
d(D)&=&-(B^2)d_1d_2d_3-\ol{d}_1d_2d_3-d_1\ol{d}_2d_3-d_1d_2\ol{d}_3 \\
&\geq& (d_1-\ol{d}_1)d_2d_3+d_1(d_2-\ol{d}_2)d_3+d_1d_2(d_3-\ol{d}_3) > 0\ ,
\end{eqnarray*}
This is a contradiction to Lemma \ref{Lemma 3.5}.

(b)\ The assertion follows from Lemma \ref{Lemma 3.2}, (iii) and (iv).

(c)\ Suppose $d_2=2$. Then $d_1=2$ since we assumed that $d_1 \le d_2 \le d_3$, and $Z_1$ and $Z_2$
consist of single $(-2)$-curves. If $B^2=-1$, then $Z_1+2B+Z_2$ defines a $\BP^1$-fibration on $\ol{S}$
with a $2$-section $C_3$. Hence a $\C^*$-fibration is induced on $S_0$ and accordingly $\lkd(S_0) \le 1$.
This is a contradiction. Hence $B^2=-2 $ by (a). But then $d(D)=4(d_3-\ol{d}_3)>0$. This is again a
contradiction by Lemma \ref{Lemma 3.5}.

(d)\ The assertion is obtained from the assertion (b) by an elementary calculation.
\end{proof}

%Lemma 5.4
\begin{lem}\label{Lemma 5.4}
If $-2\leq \Bk(D)^2\leq -\frac{3}{14}$, then $(K+D+E)^2=-2$ or $-3$.
\end{lem}
\begin{proof}
Lemma \ref{Lemma 5.3}~(vii) implies $(\Bk E)^2<-1$ because $\sum e(R_i) \geq \sum\frac{1}{r_i}=\delta>1$.
Hence, in view of the evaluation of $P^2$ in the proof of Lemma \ref{Lemma 4.1}, we have
\[
(K+D+E)^2=P^2+(\Bk D)^2+(\Bk E)^2 < \frac{3}{|\Gamma|}+(\Bk D)^2  -1.
\]
Since  $|\Gamma|\geq 2a\geq 14$ by Lemma \ref{Lemma 5.0.1} \footnote{$\Gamma$ is not abelian because $E$ is a fork. Then ƒ$\Gamma/[\Gamma,\Gamma]$ has order equal to $d(E)$.
Hence $|\Gamma|\geq 2\cdot|\Gamma/[\Gamma,\Gamma]|\geq 2d(E)$.} and $(\Bk D )^2 \leq -\frac{3}{14}$ by assumption we get $(K+D+E)^2<-1$. On the other hand, noting that
$(\Bk E)^2 \ge -2$ holds in every case and $(\Bk D)^2 \ge -2$ by the assumption, we have
$(K+D+E)^2> (\Bk D)^2+(\Bk E)^2 \ge -2-2=-4$. Hence the statement  follows.
\end{proof}

%Lemma 5.5
\begin{lem}\label{Lemma 5.5}
With the notations of Lemma \ref{Lemma 5.1}, let $Z_i=Z_{i1}+Z_{i2}+\dots+Z_{ik_i}$ for $1 \le i \le 3$,
where $Z_{ik_i}$ is a tip of $D$ and $Z_{i1}$ meets $B$. Suppose that $B^2=-1$ and that
$k_1=\#Z_1\leq 2$. Suppose further that $Z_{11}$ is a $(-2)$-curve. Then $Z_{21}^2<-2$ and $Z_{31}^2<-2$.
\end{lem}
\begin{proof}
Suppose that $Z_{21}^2=-2$. Let $F=Z_{11}+2B+Z_{21}$. Then the pencil $|F|$ induces a $\BP^1$-fibration
$f$ on $\ol{S}$ such that $Z_{31}$ is a $2$-section and that $Z_{12}$ and $Z_{22}$, if they exist, are
$1$-sections. In the  notation of Lemma \ref{Lemma 1.2.1} we have\\
$(\ast)$ \ \ $\Sigma=0$,$1$ or $2$ depending on whether none, one or two of $Z_{12}$ and $Z_{22}$ exist.\\
We put $Z_2'=Z_2-(Z_{21}+Z_{22})$ and $Z_3'=Z_3-Z_{31}$ and note that they consist of fiber components (if non-empty), and these are the only fiber components not in $F$.

Let $F_E$ be the fiber containing $E$. We recall the following consequence of Lemma \ref{Lemma 1.5}.\\
$(\ast \ast)$ \ \
An $S_0$-component of $F_E$ meets $D$ at least twice.\\
Suppose that $F_E$ contains exactly one $(-1)$-component, say $C$.
Then $C$ is multiple in $F_E$ and does not meet $Z_{12}$ and $Z_{22}$. In fact, $\mult(C)>2$, for if $\mult(C)=2$, all
other components in $F_E$, including those of $E$, are $(-2)$-curves, which contradicts Lemma \ref{Lemma 2.15}.
Hence $C\cdot Z_{31}=0$. By $(\ast \ast)$ we must have $Z_2'\neq 0, Z_3'\neq 0$ and $C\cdot Z_2'=C\cdot Z_3'=1$. As a (-1)-component  in $F_E$, $C$ does not meet any further component of $F_E$, and in particular $C\cdot E=0$. So, there must exist an $S_0$-component $C'$ of $F_E$ connecting $Z_2'+C+Z_3'$ to some other component of $F_E$. Hence $C'\cdot Z_2'=1$ or $C'\cdot Z_3'=1$, but not both, since otherwise there is a loop in $F_E$. In a fiber with exactly one $(-1)$-component the only components of multiplicity $1$ are tips. Since $C'$ meets at least two components of $F_E$, $\mult(C')>1$. Hence $C'$ does not meet $Z_{12}$ and $Z_{22}$ and, since $Z_{31}$ is already met by $Z_3'$, also does not meet $Z_{31}$. So we have $C'\cdot D=1$, in contradiction to $(\ast \ast)$.
So, $F_E$ must contain at least two $(-1)$-curves, $C_1, C_2$ say.

Suppose that none of $C_1, C_2$ nor any other $(-1)$-component meets $E$. Then $F_E$ contains a third
$S_0$-component $C_3$ which meets $E$, for otherwise $F_E$ is disconnected. By $(\ast)$, the
$1$-sections $Z_{12}$, $Z_{22}$ exist and $C_3$ is unique. We can therefore contract all $(-1)$-components and consecutively contractible components
of $F_E-(E+C_3)$, without touching $E$, until $C_3$ becomes the unique $(-1)$-component in the fiber. Now the $1$-sections do not
meet $E$ and hence there exists a component of multiplicity $1$ not contained in $E+C_3$. By Lemma
\ref{Lemma 1.1}~(f), $E$ is a linear chain. This is a contradiction.

Suppose now that $C_1$ meets $E$. Suppose also that $C_1$ meets a component of $F_E\cap D$. Then $C_1$ is
a multiple component, whence $C_1\cdot Z_{12}=C_1\cdot Z_{22}=0$. Since $C_1$ is a $(-1)$-component, it meets at most one component
of $F_E\cap D$ and $(\ast \ast)$ implies that $C_1\cdot Z_{31}=1$ and
$\mult(C_1)=2$. Then $C_2\cdot Z_{31}=0$ because $Z_{31}\cdot F_E=2$. It also follows that no components
of $Z_3'$ are contained in $F_E$. Hence $C_1$ meets a component of $Z_2'$. We now have $C_2\cdot Z_{21}=0$ and
$C_2\cdot Z_3=0$. By $(\ast \ast)$ we must have $C_2\cdot Z_2'=1$ and $C_2\cdot Z_{12}=1$. In particular,
it follows that $\mult(C_2)=1$. Suppose that $C_1$ and $C_2$ are the only $S_0$-components of $F_E$.
We contract $C_2$ and consecutively contractible components of $C_2+Z_2'$ until $C_1$ is
the only $(-1)$-curve in the fiber. The image of $Z_{12}$ meets a component of the image of
$Z_2'$, hence this component has multiplicity $1$. By Lemma \ref{Lemma 1.1}~(f), $E$
is a linear chain. This is not the case. Hence $F_E$ contains a third $S_0$-component $C_3$. But now
$C_3\cdot Z_{12}=C_3\cdot Z_{22}=C_3\cdot Z_3=0$. By $(\ast \ast)$, $C_3$ meets $Z_2'$ at least twice.
This is impossible.

So, $C_1$ as well as any $(-1)$-component in $F_E$ meeting $E$ does not meet a component of $F_E\cap D$.
If $C_1$ meets a $1$-section, $\mult(C_1)=1$. If $C_1$ does not meet the $1$-sections, then it must meet
$Z_{31}$ at least twice. In particular, $\mult(C_1)=1$ in any case. Suppose that $C_2$ meets $E$. Then
$\mult(C_2)=1$ as well. Suppose that $C_2\cdot Z_{31}=0$. Then  it follows from $(\ast \ast)$ that $C_2\cdot Z_{12}=C_2\cdot Z_{22}=1$ and $C_1\cdot Z_{31}=2$. As a consequence,
$Z_2'$ and $Z_3'$ are not part of $F_E$. If $C_1, C_2$ are the only
$(-1)$-curves in $F_E$, then $F_E$, and hence $E$, is a linear chain, which is not the case. So, there
exists a third $(-1)$-curve $C_3$ in $F_E$.  $C_3$, like $C_1$ and $C_2$, meets only sections of the fibration contained in $D$ since $Z_2'$ and $Z_3'$ are not part of $F_E$.
Each $C_i$ meets $L=Z_{12}+Z_{22}+Z_{31}$ at least twice. Thus $L\cdot F_E\geq 6$, while $L\cdot F_E=4$,
a contradiction. So $C_2\cdot Z_{31}>0$ and, by symmetry, $C_1\cdot Z_{31}>0$. Therefore $C_1\cdot Z_{12}=C_1\cdot Z_{31}=C_2\cdot Z_{22}=C_2\cdot Z_{31}=1$ and $Z_2'$ and $Z_3'$ are not part of $F_E$. By the argument above there exists a $C_3$. But now $C_3\cdot D=0$; a contradiction.       So  $C_2\cdot E=0$. There must exist a third $S_0$-component $C_3$ in $F_E$ such that
$C_3\cdot E=1$, for otherwise the fiber is disconnected. If $C_3^2=-1$, then we reach a contradiction
by the above argument where we replace $C_2$ by $C_3$. Hence the only $(-1)$-curves in $F_E$ are $C_1$ and
$C_2$. Therefore $\mult(C_2)>1$, for otherwise $F_E$ is a chain, and so $C_2\cdot Z_{12}=0$. Clearly, $C_2\cdot Z_2=C_2\cdot Z_2'\leq 1$, and by $(\ast \ast)$, $C_2$ meets $Z_3$. It is clear that
$C_2\cdot Z_3 \leq 1$, so $C_2\cdot Z_3 = 1$ and $C_2\cdot Z_2'=1$. It follows that $C_1\cdot Z_2=0$.
If $C_1$ meets $Z_{31}$ twice, then $Z_3'\nsubseteq F_E$ and $C_2\cdot Z_3=0$, a contradiction.
Hence $C_1\cdot Z_{31}=C_1\cdot Z_{12}=1$. Also $C_2\cdot Z_{31}=0$. Hence $C_2$ connects
$Z_2'$ and $Z_3'$. Now $C_3$ does not meet $Z_{12}$ and $Z_{22}$. Also
$C_3\cdot Z_{31}=0$ since $Z_{31}$ meets $C_1$ and $Z_{32}$ which is in the fiber. Thus $C_3$ must meet
$Z_2'$ and $Z_3'$. We now have a loop in $F_E$, which is a contradiction.
\end{proof}

%Lemma 5.6
\begin{lem}\label{Lemma 5.6}
$B^2=-1$.
\end{lem}
\begin{proof}
Suppose this is not the case. By Lemma \ref{Lemma 5.2} we have $B^2=-2$.  We use the notations of Setup \ref{Setup 2.1}. We have $Q_i=Z_i$ for $i=1,2,3$ and $Q_0=E$.
 Consider the decomposition
\[
K_{\ol{S}}+Q=K_{\ol{S}}+Q^\#+\Bk(Q)\ .
\]
We compare $K_{\ol{S}}+Q^\#$ with $P=K_{\ol{S}}+(D+E)^\#$. Since $(K_{\ol{S}}+Q^\#)\cdot C_j=P\cdot C_j=0$
for every component $C_j$ of $D+E-B$, we have $K_{\ol{S}}+Q^\#=\alpha P$ for some $\alpha\in \Q$. Note that $P$
is a nef $\Q$-divisor. Suppose that $\lkd(Y)=-\infty$. Then $\alpha<0$, for otherwise $K_{\ol{S}}+Q^\#$ is nef, which would imply
$\lkd(Y)\geq 0$. Hence we have
\[
(K_{\ol{S}}+Q)\cdot B=K_{\ol{S}}\cdot B+3=\alpha(P\cdot B)+(\Bk(Q)\cdot B) <\sum_{i=1}^3\wt{c}_i\ ,
\]
where $P\cdot B > 0$ for otherwise $P$ is numerically equivalent to zero and $\wt{c}_i$ denotes the
coefficient of $Z_{i1}$ in $\Bk(Z_i)$. Since $\wt{c}_i\leq 1$, we obtain $K_{\ol{S}}\cdot B<0$. This is
a contradiction because $K_{\ol{S}}\cdot B=0$ as $B \simeq \BP^1$. Thus $\lkd(Y) \geq 0$. The pair $(\ol{S},Q)$ is almost minimal by Lemma \ref{Lemma 2.2}, where one uses the condition
(i). We apply
the BMY-inequality to $(\ol{S},Q)$. Note that every $Z_i$ contracts to a cyclic quotient singularity
whose local fundamental group has order $d_i$. Since $\chi(Y)=-1$, we have
\[
\frac{1}{d_1}+\frac{1}{d_2}+\frac{1}{d_3}+\frac{1}{|\Gamma|}\geq 1\ .
\]
By Lemma \ref{Lemma 5.0.1}, we have  $|\Gamma|\geq 2a \geq 14$. Hence we obtain
\begin{equation}\label{eqn 5.3}
\frac{1}{d_1}+\frac{1}{d_2}+\frac{1}{d_3}\geq \frac{13}{14}\ ,
\end{equation}
where $2 \leq d_1\leq 3$ and $3 \leq d_2\leq 5$ by Lemma \ref{Lemma 5.2}~(c). Now (\ref{eqn 5.3}) gives
$d_3 \leq 10$. Suppose $d_1=2$. If $d_2=5$ then (\ref{eqn 5.3}) gives $d_3\leq 4$, which is a
contradiction to the hypothesis $d_2 \leq d_3$. Suppose $d_2=4$. Then $d_3 \leq 5$ by (\ref{eqn 5.3}).
Let $\ol{d}_i$ be the determinant of the chain $Z_{i2}+Z_{i3}+\ldots+Z_{ik_i}$. By Lemma \ref{Lemma 3.5},
we have
\begin{equation}\label{eqn 5.4}
d(D)=d_1d_2d_3(-B^2-\sum \frac{\ol{d}_i}{d_i}) < 0\ ,
\end{equation}
whence $\sum\frac{\ol{d}_i}{d_i}>2$ by the assumption $B^2=-2$. Since $d_2=4$, $Z_2$ consists of
three $(-2)$-curves or a single $(-4)$-curve. Accordingly, we have $\frac{\ol{d}_2}{d_2}=\frac{3}{4}$ or
$\frac{1}{4}$. In the latter case, we have $d(D)>0$, which is not the case. Hence (\ref{eqn 5.4}) gives
$\frac{\ol{d}_3}{d_3}>\frac{3}{4}$. Hence $d_3 \neq 4$, whence $d_3=5$. Further, $\ol{d}_3=4$ follows
as $1 \le \ol{d}_3 \le 4$. Now all these values give $d(D)=-2$, which contradicts Lemma \ref{Lemma 5.0.1}
as $a=-d(D)$. Suppose $d_2=3$. Then $\ol{d}_2=2$, for otherwise $d(D)>0$. We find
$\frac{\ol{d}_3}{d_3}>\frac{5}{6}$. This implies that $Z_3$ begins with at least five $(-2)$-curves and
that $\#Z_3>5$.  From (\ref{eqn 5.3}), we obtain $d_3\leq 10$. If $Z_3$ contains a $(\leq -3)$-curve,
then $d_3\geq d(R)=13$, where $R$ is the linear chain

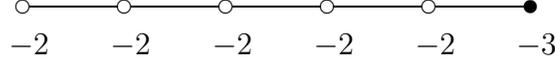
\raisebox{-15mm}{

\begin{picture}(85,20)(-40,-10)
\unitlength=0.9mm
\put(5,10){\circle{1.8}}
\put(6,10){\line(1,0){13}}
\put(20,10){\circle{1.8}}
\put(21,10){\line(1,0){13}}
\put(35,10){\circle{1.8}}
\put(36,10){\line(1,0){13}}
\put(50,10){\circle{1.8}}
\put(51,10){\line(1,0){13}}
\put(65,10){\circle{1.8}}
\put(66,10){\line(1,0){13}}
\put(80,10){\circle*{1.8}}
\put(3,3){$-2$}
\put(18,3){$-2$}
\put(33,3){$-2$}
\put(48,3){$-2$}
\put(63,3){$-2$}
\put(78,3){$-3$}
\end{picture}}

\noindent
Hence $Z_3$ is a linear chain of $(-2)$-curves of length $k$ with $6\leq k\leq 9$. We compute $d(D)=-k+5$.
Hence $a=k-5\leq 4$, which contradicts Lemma \ref{Lemma 5.0.1}. Now assume that $d_1=3$. (\ref{eqn 5.3})
gives $d_2=d_3=3$. None of the $Z_i$ is a single $(-3)$-curve, for otherwise $d(D)\geq 0$. Hence each $Z_i$
consists of two $(-2)$-curves. We compute $d(D)=0$, which is not the case.\end{proof}
Lemma \ref{Lemma 5.6} and Lemma \ref{Lemma 5.2}~(d) allow us to list all possible configurations for $D$, with
the self-intersection numbers of all components included. Once we determine $D$, we can compute $a=-d(D)$
and then find finitely many possibilities for $E$ with the help of Lemma \ref{Lemma 5.3}. Then we know
$b_2(\ol{S})$ and hence $K_{\ol{S}}^2$ by Noether's formula, in fact
$$b _2(\ol S)=\# D +\# E =10-K_{\ol{S}}^2.$$
Further, we can compute $(K_{\ol{S}}+D+E)^2$ since we also know
$ K_{\ol{S}}\cdot D$, $K_{\ol{S}}\cdot E$,
$P^2$ by using (b) of Lemma \ref{Lemma 5.2}, $\Bk(D)^2$ and $\Bk(E)^2$. Most of the cases are ruled out
by Lemma \ref{Lemma 5.4}. In the remaining cases, we check if all these numbers satisfy the equality
in the Zariski-decomposition formula
\[
(K_{\ol{S}}+D+E)^2=P^2+\Bk(D)^2+\Bk(E)^2\ .
\]
We find that the equality is not satisfied in each of the remaining listed cases.

We begin with listing the cases with $d_2=2, d_3=3$. We do not present the dual graphs of $E$ which can be
easily determined by the use of Lemma \ref{Lemma 5.3}. In the main table below, we give the columns
$(-Z_{31}^2,-Z_{32}^2,\ldots)$, $d_3$, $-d(D)=a$, $\#E$,  $b_2(\ol{S})$ and
$(K+D+E)^2$. In the more detailed table we add $Bk(D)^2, \Bk E^2$, $P^2$. We will omit the cases which can be excluded by Lemma \ref{Lemma 5.0.1}. Further,
we abbreviate a sequence like $(3,3,2,2)$ as $(3^2,2^2)$.

%Lemma 5.7
\begin{lem}\label{Lemma 5.7}
Suppose that $d_1=2$ and $d_2=3$. Then the following cases in {\bf Table $1$} exhaust all possibilities. Here $k$ denotes a non-negative integer.

\[\begin{array}{l|c|c|c|c|c|c}
(-Z_{31}^2,\dots) & d_3  & -d(D) & \#E & b_2(\ol{S})& K\cdot E & (K+D+E)^2 \\ \hline
(3,3) & 8 &10 & 6 & 11 & 1& -2 \\
(3,4) & 11 & 13 & 6 & 11 & 1 & -2 \\
(3,5) & 14 & 16 & 5 & 10 & 1 &  1 \\
(3,5) & 14 & 16 &k+4& k+9& 3 & 3-k \\
(3,5) & 14 & 16 & k+6 & k+11 & 1 & -1-k \\
(3,6) & 17 & 19 & 5 & 10 & 3 & 4 \\
(4,3) & 11 & 7 & 6 & 11 & 1& -1 \\
(4,4) & 15 & 9 & 5 & 10 & 1 & 1 \\
(4,5) & 19 &11 & 7 & 12 & 1 & 0 \\
(3,2^2) & 7 & 11 & 7 & 13 & 1 & -5 \\
(3^2,2) & 13 & 17 & 5 & 11 & 2 & -1 \\
(3,4,2) & 19 & 23 & 5 & 11 & 2 & 0 \\
(4,2^2) &10 & 8 & k+4    & k+10   & 1  &-k-1   \\
(4,2,3) & 17 & 13 & 6  &12   & 1  & -2 \\
(4,3,2) & 18 & 12 & 4  & 10   & 1  & 0\\
(4,3,2) & 18 & 12 & k+4  & k+10   & 2  & -k+1\\
(4,3,2) & 18 & 12 & k+5  & k+11   & 1  & -k-1\\
(3,2^3) & 9 & 15 & 4 & 11 & 2 & -2 \\
(3,2^2,3) & 16 & 26 & 7 & 14 & 1 & -5 \\
(3,2,3,2) & 19 & 29 & 4 & 11 & 4 & 1 \\
(3^2,2^2) & 18 & 24 & 7 & 14 & 1 & -5 \\
(3^2,2^2) & 18 & 24 & k+4 & k+11 & 5 & -k-2 \\
(3^2,2^2) & 18 & 24 & k+8 & k+15 & 1 & -k-6 \\
(4,2^3)   & 13 & 11 &  7  &  14  & 1 & -5 \\
(5,2^3)   & 17 &  7 &  6  &  13  & 1 & -3 \\
(3,2^4)   & 11 & 19 & 5 & 13 & 3 & -3 \\
(4,2^4)   & 16& 14 & 5 & 13 & 2 & -3 \\
(3,2^5)   & 13 & 23 & 5 & 14 & 2 & -5 \\
(4,2^5)   & 19 & 17 & 5 & 14 & 2 & -4 \\
(3,2^6)   & 15 & 27 & 5 & 15 & 2 & -6 \\
(3,2^7)   & 17 & 31 & 8 & 19 & 1 & -11
\end{array}
\]
In  the  cases where $(K+D+E)^2=-2$ or -3, the additional data of $P^2, (\Bk D)^2$ and $(\Bk E)^2$ are given in the
{\bf Table $1_{\rm bis}$} below.
\[
\begin{array}{l|c|c|c|c|c|c}
(-Z_{31}^2,\dots) & d_3  & -d(D) & \#E & P^2           & (\Bk D)^2       & (\Bk E)^2 \\ \hline
(3,3)             &   8  &   10  &  6  & \frac{1}{120} & -\frac{29}{24}  & -\frac{8}{5} \\
(3,4)             &   11 &   13  &  6  & \frac{25}{858}& -\frac{73}{66}  & -\frac{23}{13} \\
(3,5)             &   14 &   16  & k+4 & \frac{1}{21}  &  -\frac{22}{21} & -\frac{5}{4} \\
(3,5)             &  14  &   16  & k+6 & \frac{1}{21}  &  -\frac{22}{21} & -\frac{7}{4} \\
(4,2^2)           &   10 &     8 & k+4 &  \frac{1}{30} & -\frac{23}{15}   &  -\frac{3}{2}\\
(4,2,3)           &   17 &    13 & 6   & \frac{121}{1326}& -\frac{127}{102}&  -\frac{23}{13} \\
(4,3,2)            &   18 &    12 & k+4 &\frac{1}{9}     &-\frac{13}{9}    &  -\frac{4}{3}\\
(4,3,2)            &   18 &    12 & k+5 & \frac{1}{9}   & -\frac{13}{9}    & -\frac{5}{3}\\
(3,2^3)           &    9 & 15    & 4   & \frac{1}{90}  & -\frac{29}{18}   & -\frac{6}{5}\\
(3^2,2^2)         &  18  &   24  & k+4 & \frac{1}{18}  &  -\frac{14}{9}  & -\frac{7}{6} \\
(3^2,2^2)         &  18  &   24  & k+8 & \frac{1}{18}  &  -\frac{14}{9}  & -\frac{11}{6} \\
(5,2^3)           &  17  &   7   &  6  & \frac{11}{714}& -\frac{163}{102}& -\frac{11}{7} \\
(3,2^4)           & 11   & 19    & 5   &\frac{25}{1254}&-\frac{109}{66}  & -\frac{26}{19} \\
(4,2^4)           & 16   & 14    & 5   & \frac{25}{336}&-\frac{79}{48}    & -\frac{10}{7} \\
\end{array}
\]
None of the above cases exists.
\end{lem}
\begin{proof}
By Lemma \ref{Lemma 5.2}~(b), we have
\[
6d_3\left(\frac{1}{6}-\frac{1}{d_3}\right)^2 \leq \frac{3}{2}\ .
\]
>From this follows $d_3\leq 19$. Now we have
\[
\beta=B\cdot P=1-\frac{1}{d_1}-\frac{1}{d_2}-\frac{1}{d_3} > 0
\]
which implies $d_3>6$. So, $7\leq d_3\leq 19$.  By Lemma \ref{Lemma 5.6}, $B^2=-1$. Then $Z_2$ is a single $(-3)$-curve by
Lemma \ref{Lemma 5.5}. We find $d(D)=d_3-6\ol{d}_3$. It follows that $\ol{d}_3>1$, for otherwise
$d(D)>0$ as $d_3\geq 7$. By Lemma \ref{Lemma 5.5}, $Z_{31}^2 \leq -3$. We have also $Z_{31}^2>-7$ for
otherwise $\frac{\ol{d}_3}{d_3} \leq \frac{1}{6}$ which leads to $d(D)\geq 0$.

The first column of Table $1$ above  gives all the cases for $Z_3$. We omit most of the details, which are elementary and similar to the arguments given further on for other values of $d_1, d_2$. In each case the assumptions of
Lemma \ref{Lemma 5.4} are satisfied since $\Bk(D)^2=-\frac{1}{2}-\frac{1}{3}-\frac{\ol{d_3}}{d_3}=-\frac{5}{6}-\ol{ e_3}$ where $\ol{e_3}$ is the capacity of $Z_3$ and $0<\ol{e_3}<1$. In some cases, we have $(K+D+E)^2=-2$ or $-3$. We then compute $P^2, (\Bk(E))^2, \Bk(D)^2$, etc. in Table $1_{\rm bis}$. For all possibilities
$E$ is unique, except for $a=12$, $a=16$ and  $a=24$. These cases are explained in the graphs below, where a circle without assigned weight
corresponds to a $(-2)$-curve.

\raisebox{-85mm}{
\begin{picture}(140,85)(5,-30)
\unitlength=0.9mm
\put(5,60){\circle{1.8}}
\put(6,60){\line(1,0){18}}
\put(25,60){\circle*{1.8}}
\put(26,60){\line(1,0){18}}
\put(45,60){\circle{1.8}}
\put(25,59){\line(0,-1){8}}
\put(25,50){\circle{1.8}}
\put(22,66){$-3$}
\put(5,75){$a=12,\ (\Bk E)^2=-\frac{4}{3}$}

\put(70,60){\circle*{1.8}}
\put(71,60){\line(1,0){8}}
\put(80,60){\circle{1.8}}
\put(81,60){\line(1,0){4}}
\multiput(86,60)(1.3,0){6}{\circle*{0.2}}
\put(95,60){\line(1,0){4}}
\put(100,60){\circle{1.8}}
\put(101,60){\line(1,0){13}}
\put(115,60){\circle{1.8}}
\put(115,59){\line(0,-1){8}}
\put(115,50){\circle{1.8}}
\put(116,60){\line(1,0){13}}
\put(130,60){\circle{1.8}}
\put(67,66){$-4$}
\put(80,55){$\underbrace{\hspace{18mm}}_{k \ge 0}$}
\put(80,75){$a=12,\ (\Bk E)^2=-\frac{4}{3}$}

\put(50,20){\circle{1.8}}
\put(51,20){\line(1,0){8}}
\put(60,20){\circle*{1.8}}
\put(61,20){\line(1,0){8}}
\put(70,20){\circle{1.8}}
\put(71,20){\line(1,0){4}}
\multiput(76,20)(1.3,0){10}{\circle*{0.2}}
\put(90,20){\line(1,0){4}}
\put(95,20){\circle{1.8}}
\put(96,20){\line(1,0){8}}
\put(105,20){\circle{1.8}}
\put(105,19){\line(0,-1){8}}
\put(105,10){\circle{1.8}}
\put(106,20){\line(1,0){8}}
\put(115,20){\circle{1.8}}
\put(57,26){$-3$}
\put(70,15){$\underbrace{\hspace{23mm}}_{k \ge 0}$}
\put(55,35){$a=12,\ (\Bk E)^2=-\frac{5}{3}$}
\end{picture}}

\raisebox{-85mm}{
\begin{picture}(140,85)(5,-30)
\unitlength=0.9mm
\put(5,60){\circle{1.8}}
\put(6,60){\line(1,0){18}}
\put(25,60){\circle*{1.8}}
\put(26,60){\line(1,0){18}}
\put(45,60){\circle{1.8}}
\put(25,59){\line(0,-1){8}}
\put(25,50){\circle{1.8}}
\put(25,49){\line(0,-1){8}}
\put(25,40){\circle{1.8}}
\put(22,66){$-3$}
\put(5,75){$a=16,\ (\Bk E)^2=-\frac{19}{12}$}

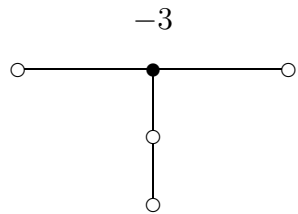
\put(70,60){\circle*{1.8}}
\put(71,60){\line(1,0){8}}
\put(80,60){\circle{1.8}}
\put(81,60){\line(1,0){4}}
\multiput(86,60)(1.3,0){6}{\circle*{0.2}}
\put(95,60){\line(1,0){4}}
\put(100,60){\circle{1.8}}
\put(101,60){\line(1,0){13}}
\put(115,60){\circle{1.8}}
\put(115,59){\line(0,-1){8}}
\put(115,50){\circle{1.8}}
\put(116,60){\line(1,0){13}}
\put(130,60){\circle{1.8}}
\put(67,66){$-5$}
\put(80,55){$\underbrace{\hspace{18mm}}_{k \ge 0}$}
\put(80,75){$a=16,\ (\Bk E)^2=-\frac{5}{4}$}

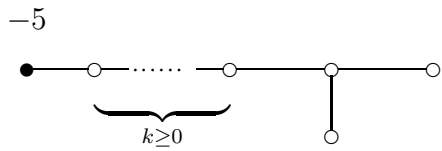
\put(40,20){\circle{1.8}}
\put(41,20){\line(1,0){8}}
\put(50,20){\circle{1.8}}
\put(51,20){\line(1,0){8}}
\put(60,20){\circle*{1.8}}
\put(61,20){\line(1,0){8}}
\put(70,20){\circle{1.8}}
\put(71,20){\line(1,0){4}}
\multiput(76,20)(1.3,0){10}{\circle*{0.2}}
\put(90,20){\line(1,0){4}}
\put(95,20){\circle{1.8}}
\put(96,20){\line(1,0){8}}
\put(105,20){\circle{1.8}}
\put(105,19){\line(0,-1){8}}
\put(105,10){\circle{1.8}}
\put(106,20){\line(1,0){8}}
\put(115,20){\circle{1.8}}
\put(57,26){$-3$}
\put(70,15){$\underbrace{\hspace{23mm}}_{k \ge 0}$}
\put(55,35){$a=16,\ (\Bk E)^2=-\frac{7}{4}$}
\end{picture}}

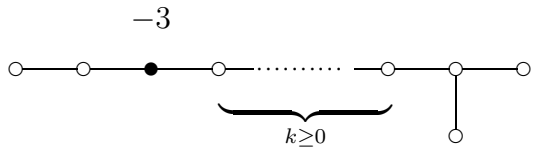
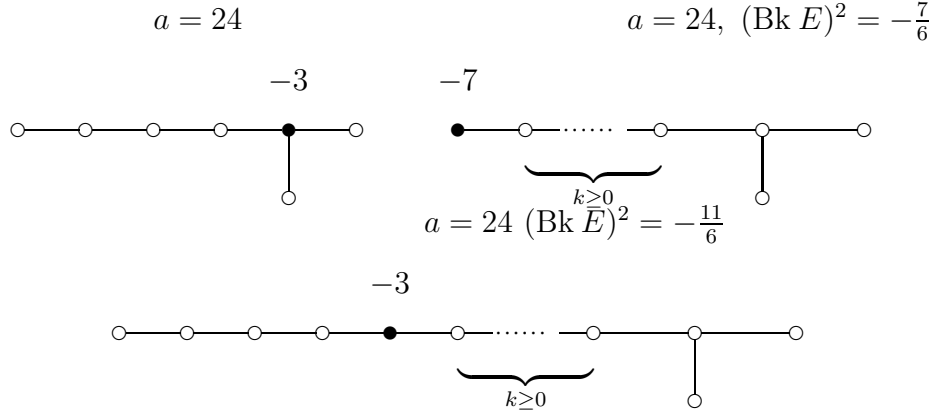
\raisebox{-65mm}{
\begin{picture}(135,65)(0,0)
\unitlength=0.9mm

\put(5,45){\circle{1.8}}
\put(6,45){\line(1,0){8}}
\put(15,45){\circle{1.8}}
\put(16,45){\line(1,0){8}}
\put(25,45){\circle{1.8}}
\put(26,45){\line(1,0){8}}
\put(35,45){\circle{1.8}}
\put(36,45){\line(1,0){8}}
\put(45,45){\circle*{1.8}}
\put(46,45){\line(1,0){8}}
\put(55,45){\circle{1.8}}
\put(45,44){\line(0,-1){8}}
\put(45,35){\circle{1.8}}
\put(42,51){$-3$}
\put(25,60){$a=24$}

\put(70,45){\circle*{1.8}}
\put(71,45){\line(1,0){8}}
\put(80,45){\circle{1.8}}
\put(81,45){\line(1,0){4}}
\multiput(86,45)(1.3,0){6}{\circle*{0.2}}
\put(95,45){\line(1,0){4}}
\put(100,45){\circle{1.8}}
\put(101,45){\line(1,0){13}}
\put(115,45){\circle{1.8}}
\put(115,44){\line(0,-1){8}}
\put(115,35){\circle{1.8}}
\put(116,45){\line(1,0){13}}
\put(130,45){\circle{1.8}}
\put(67,51){$-7$}
\put(80,40){$\underbrace{\hspace{18mm}}_{k \ge 0}$}
\put(95,60){$a=24,\ (\Bk E)^2=-\frac{7}{6}$}

\put(20,15){\circle{1.8}}
\put(21,15){\line(1,0){8}}
\put(30,15){\circle{1.8}}
\put(31,15){\line(1,0){8}}
\put(40,15){\circle{1.8}}
\put(41,15){\line(1,0){8}}
\put(50,15){\circle{1.8}}
\put(51,15){\line(1,0){8}}
\put(60,15){\circle*{1.8}}
\put(61,15){\line(1,0){8}}
\put(70,15){\circle{1.8}}
\put(71,15){\line(1,0){4}}
\multiput(76,15)(1.3,0){6}{\circle*{0.2}}
\put(85,15){\line(1,0){4}}
\put(90,15){\circle{1.8}}
\put(91,15){\line(1,0){13}}
\put(105,15){\circle{1.8}}
\put(105,14){\line(0,-1){8}}
\put(105,5){\circle{1.8}}
\put(106,15){\line(1,0){13}}
\put(120,15){\circle{1.8}}
\put(57,21){$-3$}
\put(70,10){$\underbrace{\hspace{18mm}}_{k \ge 0}$}
\put(65,30){$a=24 \ (\Bk E)^2=-\frac{11}{6}$}
\end{picture}}

With these data,  we can check if the formula $(K+D+E)^2=P^2+(\Bk D)^2+(\Bk E)^2$ is satisfied and find
out that none of the cases satisfies the equality.
\end{proof}

%Lemma 5.8
\begin{lem}\label{Lemma 5.8}
{\rm (1)}\ Suppose that $d_1=2$ and $d_2=4$. Then the following cases in {\bf Table $2$} exhaust all
possibilities. Here $k$ denotes a non-negative integer.
\[
\begin{array}{l|c|c|c|c|c|c|c}
(-Z_{31}^2,\ldots) & d_3 &-d(D) & b_2(\ol{S}) & (\Bk D)^2 & (\Bk E)^2 & P^2 & (K+D+E)^2 \\ \hline
(3,3) & 8 & 8 & k+9 & -\frac{9}{8} & -\frac{3}{2} & \frac{1}{8} & 1-k \\
(3,2^2) & 7 & 10 & 12 & -\frac{41}{28} & -\frac{8}{5} & \frac{9}{140} & -3 \\
(3,2^3) & 9 & 14 & 12 & -\frac{55}{36} &- \frac{10}{7} & \frac{25}{252} &-2
\end{array}
\]
{\rm (2)}\ Suppose that $d_1=2$ and $d_2=5$. Then the following cases exhaust all possibilities.
\begin{enumerate}
\item[{\rm (i)}]
$Z_2$ is a $(-5)$-curve and $Z_3=Z_{31}+Z_{32}$ with $Z_{31}^2=-3$ and $Z_{32}^2=-2$.
\item[{\rm (ii)}]
$Z_2=Z_{21}+Z_{22}$ with $Z_{21}^2=-3$ and $Z_{22}^2=-2$. $Z_3$ is a $(-6)$-curve or $Z_3=Z_{31}
+Z_{32}$ with $Z_{31}^2=-3$ and $Z_{32}^2=-2$.
\end{enumerate}
None of the above cases exists.
\end{lem}
\begin{proof}
(1)\ In the case $d_1=2$ and $d_2=4$, we have
\[
8d_3\left(\frac{1}{4}-\frac{1}{d_3}\right)^2 \leq \frac{3}{2}\quad \mbox{and}\quad
\beta=P\cdot B=\frac{1}{4}-\frac{1}{d_3} > 0\ .
\]
Hence we have $5 \leq d_3 \leq 9$. By Lemma \ref{Lemma 5.5}, $Z_2$ consists of a single $(-4)$-curve.
Since $B^2=-1$ by Lemma \ref{Lemma 5.6}, we find $d(D)=2d_3-8\ol{d}_3$. This implies $\ol{d}_3>1$, for
otherwise $d(D)>0$ as $d_3\geq 5$.

(2)\ In the case $d_1=2$ and $d_2=5$ we have
\[
10d_3\left(\frac{3}{10}-\frac{1}{d_3}\right)^2 \leq \frac{3}{2}\ ,
\]
whence $d_3=5$ or $6$ as $d_2 \le d_3$. Suppose $\#Z_2=1$. Then $\#Z_3>1$, for otherwise $d(D)>0$.
In view of Lemma \ref{Lemma 5.5} we have only one case to consider, which is $Z_{31}^2=-3$ and $Z_{32}^2=-2$.
Then $d_3=5$ and $a=-d(D)=5$. This is a contradiction by Lemma \ref{Lemma 5.0.1}. Suppose that $\#Z_2=2$.
Then $Z_{21}^2 < -2$ by Lemma \ref{Lemma 5.5}. Hence $Z_{21}^2=-3$ and $Z_{22}^2=-2$. If $Z_3$ is a single
$(-5)$-curve, then this is the case (i) with $Z_2$ and $Z_3$ interchanged. By Lemma \ref{Lemma 5.5},
$Z_{31}^2<-2$. So, either $Z_{31}^2=-3, Z_{32}^2=-2$ or $Z_3$ is a single $(-6)$-curve. In the first case,
we have $a=15$. So, $\#E=4$ and $K\cdot E=2$. We compute $b_2(\ol{S})=10, K^2=0$ and $(K+D+E)^2=-1$. This
is a contradiction by Lemma \ref{Lemma 5.4}. In the second case, $a=4$ which is a contradiction again
by Lemma \ref{Lemma 5.0.1}.
\end{proof}

The remaining case is $d_1=3$. In this case, we make a case by case argument.
\svskip
%Lemma 5.9
\begin{lem}\label{Lemma 5.9}
No case exists with $d_1=3$.
\end{lem}
\begin{proof}
Assume $d_1=3$.\\

{\sc Step 1.}\ By Lemma \ref{Lemma 5.2}~(b), we have
\begin{equation}\label{eqn 5.5}
3d_2d_3(\frac{2}{3}-\frac{1}{d_2}-\frac{1}{d_3})^2 \leq \frac{3}{2}\ .
\end{equation}
Since
\[
\beta=B\cdot P= \frac{2}{3}-\frac{1}{d_2}-\frac{1}{d_3} \geq \frac{2}{3}-\frac{2}{d_2} \geq 0\ ,
\]
(\ref{eqn 5.5}) yields
\[
3d_2^2(\frac{2}{3}-\frac{2}{d_2})^2\leq\frac{3}{2}\ .
\]
Hence $(d_2-3)^2 \leq \frac{9}{8}$ and we find $d_2=3$ or $4$. Now it is easy to determine $d_3$ and we
finally have to consider the following cases for the triple $(d_1,d_2,d_3)$ : $(3,3,3), (3,3,4),
(3,3,5), (3,3,6), (3,4,4)$. In the case $(3,3,3)$, $P^2=0$ by Lemma \ref{Lemma 5.2}~(b). Hence we can
exclude it.
\svskip

{\sc Step 2.}\ Suppose that $(d_1,d_2,d_3)=(3,3,4)$. Suppose that $Z_3$ is a single $(-4)$-curve.
Then one of $Z_1$ and $Z_2$, say $Z_1$, consists of two $(-2)$-curves, for otherwise $d(D)>0$. Then $Z_2$
is a single $(-3)$-curve by Lemma \ref{Lemma 5.5}. We find $d(D)=-9, b_2(\ol{S})=10, K^2=0, K\cdot D=2$
and $K\cdot E=1$. Hence $(K+D+E)^2=-1$. Since $(\Bk D)^2=-5/4$, this is a contradiction by Lemma \ref{Lemma 5.4}. Suppose that
$Z_3$ consists of $(-2)$-curves. Then $Z_1$ and $Z_2$ are single $(-3)$-curves by Lemma \ref{Lemma 5.5}.
We find $d(D)=-15, b_2(\ol{S})=10, K^2=0, K\cdot E=2$ and $K\cdot D=1$. Thus $(K+D+E)^2=-1$. Since $(\Bk D)^2=-17/12$, we come to a
contradiction as above.
\svskip

{\sc Step 3.}\ Suppose that $(d_1,d_2,d_3)=(3,3,5)$.  Suppose that $Z_3$ is a single $(-5)$-curve. Then
one of $Z_1$ and $Z_2$, say $Z_1$, consists of $(-2)$-curves, for otherwise $d(D)>0$. Then $Z_2$ is
a single $(-3)$-curve by Lemma \ref{Lemma 5.5}. We find $d(D)=-9$. Now $b_2(\ol{S})=10, K^2=0, K\cdot E=1$
and $K\cdot D=3$. Hence $(K+D+E)^2=0$, a contradiction by Lemma \ref{Lemma 5.4} since $(\Bk D)^2=-6/5$. Suppose that
$Z_{31}^2=-2$ and $Z_{32}^2=-3$. Then $Z_1$ and $Z_2$ are single $(-3)$-curves by Lemma \ref{Lemma 5.5}.
We find $d(D)=-12$. We have $3$ possibilities of $E$ in case $a=12$ whose dual graphs are listed in the proof of Lemma \ref{Lemma 5.8}.

In the first case, $\#E=4$, $K\cdot E=1$ and $(\Bk D)^2=-\frac{16}{15}$. We find $(K+D+E)^2=0$, which is
a contradiction by Lemma \ref{Lemma 5.4}. In the second case, $\#E=k+4$ and $K\cdot E=2$. Thus
$(K+D+E)^2=1-k$. We find $(\Bk E)^2=-\frac{4}{3}, (\Bk D)^2=-\frac{16}{15}$ and $P^2=\frac{1}{15}$. These data do not satisfy the formula for the Zariski decomposition. In the third case, $\#E=k+5, K\cdot E=1,
(\Bk E)^2=-\frac{5}{3}, (\Bk D)^2=-\frac{16}{15}$ and $P^2=\frac{1}{15}$.  This again gives a contradiction.
Suppose that $Z_{31}^2=-3$ and $Z_{32}^2=-2$. If $Z_1$ and $Z_2$ are single $(-3)$-curves, then $a=3$
which contradicts Lemma \ref{Lemma 5.0.1}. So, one of $Z_1$ and $Z_2$, say $Z_1$, is a single
$(-3)$-curve and the second one consists of $(-2)$-curves. Then $a=18$ and this again contradicts
Lemma \ref{Lemma 5.0.1}. Suppose finally that $Z_3$ consists of $(-2)$-curves. Then $Z_1$ and $Z_2$ are single $(-3)$-curves by Lemma \ref{Lemma 5.5}. We find $a=21$. In this case, the corresponding $E$ is of
type $(2,3,3)$ with $h=3$ and two $3$-twigs consisting of single $(-3)$-curves. Then we have $b_2(\ov S)=11$, $K\cdot E=3$, $K\cdot D=1$. Hence $(K+D+E)^2=-1$. Since $(\Bk E)^2=-\frac{22}{15}$ we get a contradiction by Lemma \ref{Lemma 5.4}.

\svskip

{\sc Step 4.}\ Suppose that $(d_1,d_2,d_3)=(3,3,6)$.  Suppose that $Z_3$ is a single $(-6)$-curve. Then
one of $Z_1, Z_2$ is a single $(-3)$-curve and the second one consists of $(-2)$-curves. We find $a=9$, $K\cdot E=1$, $K\cdot D=4$, $b_2(\ov S)=10$. We find $(K+D+E)^2=1$;  a contradiction since $(\Bk E)^2=-\frac{7}{6}$. Suppose that $Z_3$ consists of $(-2)$-curves. Then both $Z_1,Z_2$ are single
$(-3)$-curves by Lemma \ref{Lemma 5.5}. We find $a=27$. Then $E$ is of type $(2,3,3)$ with $h=3$ and two
$3$-twigs consisting of a single $(-3)$-curve and a chain of two $(-2)$-curves. Hence $b_2(\ov S)=13$, $K\cdot E=1$, $K\cdot D=1$. We find $(K+D+E)^2=-4$; a contradiction by Lemma \ref{Lemma 5.4} since $(\Bk E)^2=-\frac{3}{2}$.

\svskip

{\sc Step 5.}\ Suppose that $(d_1,d_2,d_3)=(3,4,4)$. Suppose that $Z_1$ is a single $(-3)$-curve.
If $Z_2$ and $Z_4$ are $(-4)$-curves, then $d(D)=8$, which contradicts Lemma \ref{Lemma 3.5}. Suppose
that $Z_2$ and $Z_3$ are chains of $(-2)$-curves of length $3$. Then $a=40$, and $E$ is of type $(2,2,n)$
by Lemma \ref{Lemma 5.3}. Then we compute
\begin{eqnarray*}
(K+D+E)^2 &=& P^2+(\Bk D)^2+(\Bk E)^2 \\
&=&\frac{1}{30}-\frac{11}{6}-\left(1+\frac{1}{10n}+\frac{\wt{n}}{n}\right) = -\frac{28n+1+\wt{n}}{10n}\ .
\end{eqnarray*}
Since $(K+D+E)^2$ is an integer, $n$ divides $\wt{n}+1$. Since $n \ge \wt{n}+1$, we have $n=\wt{n}+1$,
which implies that $n$-chain of $E$ consists of $(-2)$-curves. Then $(K+D+E)^2=-\frac{29}{10}$, a
contradiction. Hence one of $Z_2$ and $Z_3$, say $Z_2$, is a single $(-4)$-curve and the second one consists
of $(-2)$-curves. Let $F=Z_2+4B+3Z_{31}+2Z_{32}+Z_{33}$. $F$ gives rise to a $\PP^1$-ruling of $\ov S$. $D-\Supp (F)=Z_1$ and $Z_1$ is a 4-section of the ruling. We get contradiction with Lemma \ref{Lemma 4.2}~(c).

Suppose next that $Z_1$ consists of $(-2)$-curves. Then $Z_2$ and $Z_3$ are single $(-4)$-curves
by Lemma \ref{Lemma 5.5}. We find $a=8$. By Lemma \ref{Lemma 5.3}, the corresponding $E$ is of type
$(2,2,n)$ with $\ol{n}+2=n(h-1)$. This implies that $h=2$ and $n=\ol{n}+2$. Hence $E$ has the dual
graph as given below.

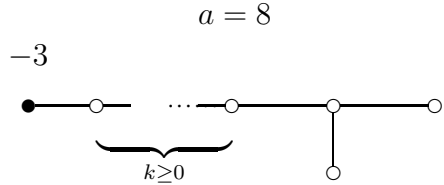
\raisebox{-35mm}{
\begin{picture}(70,35)(-60,-15)
\unitlength=0.9mm
\put(5,15){\circle*{1.8}}
\put(6,15){\line(1,0){8}}
\put(15,15){\circle{1.8}}
\put(16,15){\line(1,0){4}}
\multiput(26,15)(1.3,0){6}{\circle*{0.2}}
\put(30,15){\line(1,0){4}}
\put(35,15){\circle{1.8}}
\put(36,15){\line(1,0){13}}
\put(50,15){\circle{1.8}}
\put(50,14){\line(0,-1){8}}
\put(50,5){\circle{1.8}}
\put(51,15){\line(1,0){13}}
\put(65,15){\circle{1.8}}
\put(2,21){$-3$}
\put(15,10){$\underbrace{\hspace{18mm}}_{k \ge 0}$}
\put(30,27){$a=8$}
\end{picture}}

Then we have
\[
P^2+(\Bk E)^2+(\Bk D)^2=\frac{1}{6}-\frac{3}{2}-\frac{7}{6}
\]
which is not an integer. Hence this case does not take place. This completes the first proof
of Theorem 1.
\end{proof}

\section{Second proof of Theorem 1}

It follows from Lemma \ref{Lemma 4.3} that $D$ is {\em star-shaped}. We now use some results from the
theory of normal affine surfaces with a $\C^*$-action such that there is a unique fixed point and
the closure of any orbit passes through the fixed point. We say that the $\C^*$-action is {\em good}.

Now we can find a normal affine surface $V$ with a good $\C^*$-action satisfying the following
conditions:
\begin{enumerate}
\item[(1)]
The $\C^*$-action has a unique fixed point $p$.
\item[(2)]
$V$ has an MNC-completion $\ol{V}$ such that the boundary divisor $\Delta:= \ol{V}\setminus V$ is a star-shaped
divisor whose weighted dual graph is the same as that of $D$.
\item[(3)]
The $\C^*$-action on $V$ extends to $\ol{V}$.
\end{enumerate}
For the existence of such a surface $V$, see \cite[Thm. 2.1]{P}. Recall, with our initial set-up and
notations, that $H_1(S'\setminus \{q\},\Z)$ is finite and isomorphic to the divisor class group of $S'$.
We use $V$ to analyze the maximal abelian unramified covering $T \to S'\setminus \{q\}$. Let $U$ be a small
contractible neighborhood of $q$ in $S'$. From the isomorphism $H_1(U\setminus \{q\},\Z) \to H_1(S'\setminus \{q\},\Z)$, we deduce,
by making use of covering space theory, that the normalization $T'$ of $S'$ in the function field
of $T$ is obtained by adding a point $p$ to $T$ which lies over $q$. Since $K_{S'}$ is a torsion divisor
in the divisor class group of $S'$, we see that $K_T$ is a trivial line bundle and $p$ is a rational
double point. Let $\ol{T}$ be an MNC-completion of $T'$ which is smooth outside $p$. Set
$\wt{D}:=\ol{T}\setminus T'$. Let $N$ be a suitable small tubular neighborhood of $D$ in $\ol{S}$. Its inverse
image $\wt{N}$ is a small tubular neighborhood of $\wt{D}$ in $\ol{T}$. The natural map
$\wt{N}\setminus \wt{D} \to N\setminus D$ is a topological covering of degree $a$, which is equal to the order of
$H_1(S'\setminus \{q\},\Z)$. It is well-known that $N$ is obtained by the process of plumbing of disc-bundles on the
irreducible components of $D$, and the $4$-manifold with boundary, $N$, is uniquely determined by
the weighted dual graph of $D$. It follows that $\Delta$ has a tubular neighborhood $N_{\infty}$ in
$\ol{V}$ which is $C^{\infty}$-diffeomorphic to $N$. The first homology groups $H_1(S'\setminus \{q\},\Z), H_1(N\setminus D,\Z),
H_1(V\setminus \{p\},\Z), H_1(N_{\infty}\setminus \Delta,\Z)$ are all naturally isomorphic. We consider the maximal abelian
unramified covering $W' \to V\setminus \{p\}$. Let $W$ be the normalization of $V$ in the function field of $W'$.
It is easy to see that $W$ admits a good $\C^*$-action such that $W \to V$ is an equivariant morphism.
$W$ has an MNC-completion $X$ such that $X$ is smooth outside $W$ and $\wt{\Delta}:=X\setminus W$ is a star-shaped
divisor. This implies that $\wt{D}$ is also star-shaped. To see this, we may use a result about plumbed
$4$-manifolds $M$ which says that if $M$ is not a lens space, the boundary $3$-manifold $\partial M$
determines the weighted divisor \cite[Theorem 5.1]{Ne}. In our situation, even the fundamental group of
the $3$-manifold $\partial \wt{N}$ is infinite, hence it is certainly not a lens space.

%Remark 6.1
\begin{remark}\label{Remark 6.1}{\em
It is possible to give an elementary proof that $\wt{D}$ is star-shaped without using Neumann's result.}
\end{remark}

As in the proof of Lemma \ref{Lemma 2.15}, there is a canonical divisor $K_{\ol{T}}$ supported on
$\wt{D}$. Let $\wt{D}_0$ be the unique branch curve in $\wt{D}$. Since $\chi(S_0)=0$, we have $\chi(T)=0$.
If $\wt{D}_0$ is a rational curve, then we see easily that $H_1(\wt{N}\setminus \wt{D},\Z)$ is finite and
hence by a suitable Lefschetz theorem $H_1(T',\Z)$ is also finite \cite[Corollary 2.3]{N}. This easily
implies that $T'$ is a $\Q$-homology plane which is Gorenstein. Then we arrive at a contradiction as in
the proof of Lemma \ref{Lemma 2.15}.
\svskip

Now we assume that $\wt{D}_0$ is not rational. As in the proof of Lemma \ref{Lemma 2.15}, we see that
there is an effective canonical divisor supported on $\wt{D}$. Since $\wt{D}$ does not contain
a $(-1)$-curve, we see that a minimal resolution of singularities $\wt{T}$ of $\ol{T}$ has
no $(-1)$-curve. In fact, if $C$ is a $(-1)$-curve, $C$ is not contained in $\wt{D}$ or in the the
exceptional locus $\wt{E}$. So, $C\cdot \wt{D} > 0$, which contradicts $C\cdot K_{\wt{T}} =-1$. Hence
$\wt{T}$ is a smooth minimal projective algebraic surface. By a similar reason, it also follows that
$\wt{T}$ is not ruled or rational. First we prove the following result.

%Lemma 6.2
\begin{lem}\label{Lemma 6.2}
If $p$ is a smooth point of $T'$, then Theorem 1 is true.
\end{lem}
\begin{proof}
In this case the local $\pi_1$ at $q$ is a finite abelian group. This implies that the local
$\pi_1$ is actually finite cyclic by the well-known results about small finite subgroups of $\GL(2,\C)$.
Hence Theorem 1 is proved if $p$ is a smooth point.
\end{proof}

In view of this result, we will assume that $p$ is not a smooth point.

%Lemma 6.3
\begin{lem}\label{Lemma 6.3}
$\wt{T}$ is a surface of general type.
\end{lem}
\begin{proof}
We use the classification of minimal algebraic surfaces. $\wt{T}$ cannot be an abelian surface since
$p$ is not a smooth point and the exceptional divisor obtained by resolving $p$ is a tree of rational
curves contained in $\wt{T}$. Suppose that $\wt{T}$ is a $K$-3 surface. Then $\wt{T}$ is simply-connected.
It is not difficult to see that the components of $\wt{D}$ are rationally independent in the homology of
$\wt{T}$. From this we see that $T'$ is a $\Q$-homology plane, and arrive at a contradiction as $\wt{T}$
is rational. Now it follows that $K_{\wt{T}}$ is not trivial, and hence a strictly positive divisor
supported on $\wt{D}$. Suppose that $\wt{T}$ is an elliptic surface. Let $F$ be a general fiber of
the elliptic fibration. Then $F^2+K\cdot F=0$ and $F^2=0$. Hence $K\cdot F=0$. But $\wt{D}$ being ample,
$F$ has to meet $\wt{D}$, so that $K\cdot F>0$. This proves the result.
\end{proof}

Now $K_{\wt{T}}^2>0$. We will now use the BMY-inequality. Using $\chi(T')=0$ and the arguments
so far, we can see that the pair $(\wt{T},\wt{D})$ is strongly minimal in the sense of \cite[4.9]{Mi}.
This is proved in the same way as Lemma \ref{Lemma 1.3} (see also the proof of Lemma \ref{Lemma 7.2}
below). Note that $p$ is a rational double point and accordingly $\wt{E}^\#=0$. Now the BMY-inequality
gives
\[
(K_{\wt{T}}+\wt D^{\#})^2 \leq \frac{3}{|G|}\ ,
\]
where $G$ is the local $\pi_1$ at $p$ and $|G|$ is at least $2$. Note that $K_{\wt{T}}+\wt{D}^\#$ is
written as $K_{\wt{T}}+\wt{D}_0+\wt{D}'$, where $\wt{D}_0$ is the branching component of $\wt{D}$
and $\wt D'$ is an effective $\Q$-divisor supported on the maximal twigs of $\wt{D}$. Hence
$(K_{\wt{T}}+\wt{D}^\#)\cdot \wt{D}'=0$. Expanding the LHS, we have
\begin{eqnarray*}
{\rm LHS} &=& (K_{\wt{T}}+\wt{D}^\#)\cdot(K_{\wt{T}}+\wt{D}_0) \\
&=& K_{\wt{T}}^2+2K_{\wt{T}}\cdot\wt{D}_0+\wt{D}_0^2+\wt{D}_0\cdot \wt{D}'\ ,
\end{eqnarray*}
where $\wt{D}_0\cdot \wt{D}' \geq 0$. Clearly $K_{\wt{T}}^2>0$ is an integer by Lemma \ref{Lemma 6.3}.
Since $K_{\wt{T}}$ is nef, $K_{\wt{T}}\cdot \wt{D}_0\geq 0$ and  $K_{\wt{T}}\cdot \wt{D}'$. But $\wt{T}$ is a minimal surface of
general type, whence $K_{\wt{T}}\cdot\wt{D}_0=0$ would imply that $\wt{D}_0$ is a $(-2)$-curve.
But this is not true. Therefore, $K_{\wt{T}}\cdot\wt{D}_0>0$. Thus, the LHS is at least $2$. This
contradicts the BMY-inequality. Hence $p$ is a smooth point and Theorem 1 is proved.

\section{Proof of Theorem 2}

We shall deduce Theorem 2 from Theorem 1 and the following result, where the assumptions are slightly more general than in Theorem 2, but which otherwise is essentially a restatement of Theorem 2 without the conclusion that $\Aut(V)$ is cyclic.

%Theorem 7.1
\begin{thm}\label{Theorem 7.1}
Let $X$ be a smooth homology plane of general type. Let $G$ be a nontrivial subgroup of
$\Aut(X)$. Then the following assertions hold.
\begin{enumerate}
\item[{\rm (1)}]
There exists a unique fixed point $P$ of $X$ under the $G$-action and $G$ acts freely on $X\setminus \{P\}$.
\item[{\rm (2)}]
$G$ is a small finite subgroup of $\GL(2,\C)$ (in the induced action on the tangent space at $P$).
\item[{\rm (3)}]
The quotient surface $X'=X/G$ has a unique singular point $Q$ and the local fundamental group at $Q$ is $G$.
\item[{\rm (4)}] If $X$ is contractible, equivalently if $X$ is simply connected, then  $\pi_1(X_0)=G$, where $X_0 = X'\setminus \{Q\}$.
\end{enumerate}
\end{thm}

We need the following lemmas.

%Lemma 7.0
\begin{lem}\label{Lemma 7.0}
Let $X$ and $G$ be as above. Then the quotient $X'=X//G$ is $\Z$-acyclic.
\end{lem}
\begin{proof} This is a standard result in Smith Theory of finite group actions on simplicial complexes. It is
well-known that any complex algebraic variety is triangulable \cite{Gi}. Clearly the corresponding simplicial complex
is finite dimensional. Since $X$ is $\Z$-acyclic, by (\cite{Br}, Chapter 3, Theorem 5.4) $X/G$ is also $\Z$-acyclic.
\end{proof}

%Lemma 7.2
\begin{lem}\label{Lemma 7.2}
Let $X$ be as above. Then there are no contractible curves on $X$.
\end{lem}
\begin{proof}
Suppose that there exists a contractible curve $C$ on $X$. Set $Y=X\setminus C$. Then $Y$ is a smooth affine
surface of general type. Let $Z$ be an almost minimal model of $Y$. Since $Y$ is affine, $Z$ is
an open set of $Y$ and $Y\setminus Z$ is a disjoint union of affine lines \cite{GM}. Then $\chi(Z) \le
\chi(Y) = 0$. Meanwhile, the BMY-inequality implies $\chi(Z) > 0$, which is a contradiction.
\end{proof}

%Lemma 7.3
\begin{lem}\label{Lemma 7.3}
Let $X$ be a homology plane with an effective action of a cyclic group $H$ of prime order.
Then we have :
\begin{enumerate}
\item[{\rm (1)}]
The fixed point locus $X^H$ is connected.
\item[{\rm (2)}]
If $\dim X^H=1$ then $X^H$ is an affine line.
\end{enumerate}
\end{lem}
\begin{proof}
See \cite[Lemma 1.5]{MS}.
\end{proof}
\svskip
\textit{Proof} of Theorem \ref{Theorem 7.1}. It is well known that $\Aut(X)$ is finite (since $\overline\kappa(X)=2$).
We now proceed in several steps.
\svskip

{\sc Step 1.}\ By Lemma \ref{Lemma 7.0} $\chi(X')=1$. We  claim that there is a point $P\in X$ with nontrivial isotropy group. In fact, otherwise the
quotient morphism $q : X \to X'$ is an \'etale finite morphism. Hence $\chi(X)=|G|\chi(X')$. Since
$\chi(X)=\chi(X')=1$ this is a contradiction.
\svskip

{\sc Step 2.}\ Let $P \in X$ be a point with nontrivial isotropy group $H_P$. Then the tangential
representation of $H_P$ on $T_{X,P}$ does not contain a pseudo-reflection. In fact, if an element $h$
is a pseudo-reflection, we may assume that $h$ has prime order $p$. Let $H$ be the cyclic group
generated by $h$. Then $\dim X^H=1$. Now $X^H$ is an affine line by Lemma \ref{Lemma 7.3}, and
this is impossible by Lemma \ref{Lemma 7.2}. This implies that $X^{H_P}$ is a single point.
\svskip

{\sc Step 3.}\ Let $Q=q(P)$. Then $q^{-1}(Q)$ consists of $n$ points $P_1, \ldots, P_n$,
where $P_1=P$ and $n=[G : H_P]$. We shall show that there is a unique point $P$ with nontrivial isotropy
group, which becomes a fixed point under $G$. In fact, let $\{P_{11}, \ldots, P_{1n_1},P_{21}, \ldots,
P_{2n_2}, \ldots, P_{r1}, \ldots, P_{rn_r}\}$ exhaust all points of $X$ with nontrivial isotropy groups,
where $P_{i1}, \ldots, \\ P_{in_i}$ have conjugate isotropy groups for $1 \le i \le r$. Then
$P_{i1}, \ldots, P_{in_i}$ are mapped to the same point $Q_i$ of $X'$. Since
$q : X\setminus \{P_{11}, \ldots, P_{rn_r}\} \to X'\setminus \{Q_1, \ldots, Q_r\}$ is an \'etale finite morphism,
we have the equality $\chi(X)-(n_1+\cdots+n_r)=|G|(\chi(X')-r)$. Let $m_i$ be the order of the isotropy
group of $P_{i1}$. Then $n_im_i=|G|$. Since $\chi(X)=\chi(X')=1$, we deduce the equality
\[
r-1=\frac{1}{m_1}+\cdots+\frac{1}{m_r}-\frac{1}{|G|}\ .
\]
Since $m_i \ge 2$, we have
\[
r-1 \le \frac{r}{2}-\frac{1}{|G|} < \frac{r}{2}\ .
\]
Hence $r=1$ since $r > 0$ by step 1. So, $r=1$ and $m_1=|G|$, i.e. there is a unique isotropy
group and it is a $G$-fixed point. This proves the assertion (1).
\svskip

{\sc Step 4.}\ We shall prove the remaining assertions. Let $\rho : G \To GL(2,\C)$ be the
tangential representation at the point $P$. Then $\rho$ is injective. In fact, if $g$ is in the kernel
of $\rho$, then $g$ acts trivially on the completion of $\mathcal{O}_{X,P}$, and hence on $\SO_{X,P}$.
Hence $g$ is the identity. By step 2, $G$ is a small finite subgroup of $GL(2,\C)$. Let $Q=q(P)$.
Then $q : X\setminus \{P\} \To X'\
\setminus \{Q\}$ is a Galois \'etale covering with group $G$.  Also, $q$ induces a covering $U \To U'$, where $U$ is a simply connected punctured neighborhood of $P$ and $U'$ a punctured neighborhood of $Q$. This proves (3). Under the assumptions of (4), $X\setminus \{P\}$ is simply connected, so $\pi_1 (X'\setminus \{Q\})=G$.
$\hfill \Box$
\svskip

\textit{Proof} of Theorem 2. All statements except that $\Aut(V)$ is cyclic follow from Theorem \ref{Theorem 7.1}. In particular, $\Aut(V)$ is the local fundamental group at $q$. Now $V'=V/\Aut(V)$ is $\Z$-acyclic by Lemma \ref{Lemma 7.0} and $\lkd(V'\setminus\{q\})=\lkd(V\setminus\{p\})=\lkd(V)=2$. Hence $\Aut(V)$ is cyclic by Theorem 1.
$\hfill \Box$

\svskip

\med\noin
R.V. Gurjar, School of Mathematics, Tata Institute for Fundamental Research, 400005 Homi Bhabha Road,
Mumbai, India; \\ e-mail: gurjar@math.tifr.res.in

\med\noin
M. Koras, Institute of Mathematics, Warsaw University, ul. Banacha 2, Warsaw, Poland; \\
e-mail: koras@mimuw.edu.pl

\med\noin
M. Miyanishi, Research Center for Mathematical Sciences, Kwansei Gakuin University, Hyogo 669-1337,
Japan;\\
e-mail: miyanisi@kwansei.ac.jp

\med\noin
P. Russell, Department of Mathematics and Statistics, McGill University, Montreal, Canada;\\
e-mail: russell@math.mcgill.ca
\end{document}